\documentclass[11pt]{amsart}
\usepackage{tikz-cd}
\usepackage{enumitem}
\setlist[description]{leftmargin=\parindent,labelindent=\parindent}
\usetikzlibrary{matrix,arrows,decorations.pathmorphing}
\usepackage{url}
\usepackage{cite}
\usepackage{fullpage}
\usepackage{verbatim}
\usepackage{comment}
\usepackage{fancyvrb}
\usepackage{fvextra}
\usepackage{caption}
\usepackage{stmaryrd}
\usepackage{appendix}

\allowdisplaybreaks

\usepackage{amsmath}
\usepackage{mathtools}
\usepackage{amssymb}
\usepackage{amsthm}
\usepackage{mathrsfs}

\usetikzlibrary{calc,babel}
\tikzset{
  curve/.style={
    settings={#1},
    to path={
      (\tikztostart)
      .. controls ($(\tikztostart)!\pv{pos}!(\tikztotarget)!\pv{height}!270:(\tikztotarget)$)
      and ($(\tikztostart)!1-\pv{pos}!(\tikztotarget)!\pv{height}!270:(\tikztotarget)$)
      .. (\tikztotarget)\tikztonodes
    },
  },
  settings/.code={%
    \tikzset{quiver/.cd,#1}%
    \def\pv##1{\pgfkeysvalueof{/tikz/quiver/##1}}%
  },
  quiver/.cd,
  pos/.initial=0.35,
  height/.initial=0,
}

\usepackage{bbm}

\usepackage{hyperref}
\hypersetup{
	colorlinks=true,
	linkcolor=blue,
	citecolor=black,
	urlcolor=blue}

\newtheorem{theorem}{Theorem}[section]
\newtheorem{proposition}[theorem]{Proposition}
\newtheorem{lemma}[theorem]{Lemma}
\newtheorem{corollary}[theorem]{Corollary}

\newtheorem{conjecture}{Conjecture}

\theoremstyle{definition}
\newtheorem{definition}[theorem]{Definition}
\newtheorem{example}[theorem]{Example}

\theoremstyle{remark}

\newtheorem{remark}[theorem]{Remark}

\usepackage{wrapfig}

\DeclareMathOperator{\SHom}{\mathscr{H}\kern -.5pt om}

\DeclareMathOperator{\Spec}{Spec}

\DeclareMathOperator{\LogJac}{LogJac}
\DeclareMathOperator{\Ext}{Ext}
\DeclareMathOperator{\Span}{span}
\DeclareMathOperator{\im}{Im}

\DeclareMathOperator{\Aut}{Aut}
\DeclareMathOperator{\Conv}{Conv}
\def\log{\mathrm{log}}
\def\hom{\mathrm{hom}}
\def\gp{\mathrm{gp}}
\def\op{\mathrm{op}}
\def\trop{\mathrm{trop}}

\DeclareMathOperator{\LogSch}{\mathbf{LogSch}}
\DeclareMathOperator{\ShpMon}{\mathbf{ShpMon}}

\DeclareMathOperator{\RPC}{\mathbf{RPC}}

\DeclareMathOperator{\TroJac}{TroJac}
\def\colim{\qopname\relax m{colim}}
\DeclareMathOperator{\Hom}{\textup{Hom}}

\title{Birational Classification of Orbifold Compactified Jacobians}
\author{Jeremy Feusi and Sam Molcho}

\begin{document}

\begin{abstract}
	We study the equivariant orbifold birational classification problem for families of toroidal compactifications of a group $G$ over a toroidal base, in the cases where $G$ is an algebraic torus or a semiabelian scheme. The classification is reduced to the problem of finding the minimal orbifold toroidal compactifications of $G$ in the world of logarithmic geometry, which is shown to be a combinatorial problem. We solve the problem for families of algebraic tori, Jacobians of families of nodal curves, and semiabelian schemes with abelian generic fiber. The general semiabelian case is reduced to an open conjecture. These results generalize and geometrically interpret recent results of \cite{Schmitt}.
\end{abstract}

\maketitle

\tableofcontents

\section{Introduction}
The birational classification of algebraic varieties is one of the central
problems of algebraic geometry. The analogous problem for orbifolds is much
less studied, but exhibits several novel phenomena. From the outset, the definition
is interesting: it is an observation of Kresch and Tschinkel
\cite{KreschTschinkel} that to call two DM stacks $\mathcal{X}$ and
$\mathcal{Y}$ birationally equivalent, one should not simply require the
existence of a birational map between them -- a definition which would imply
for example that every orbifold is birational to a variety. Rather, one should
say that two Deligne-Mumford stacks $\mathcal{X}$ and $\mathcal{Y}$ are
birationally equivalent if there exists a correspondence
\[
\begin{tikzcd}
	& \mathcal{Z} \ar[rd,"p"] \ar[ld, swap, "q"] & \\ \mathcal{X}& &\mathcal{Y}
\end{tikzcd}
\]
with $\mathcal{Z}$ a DM stack and with the maps $p,q$ proper, birational and
\emph{representable}. We remark here that this definition is not equivalent to
the traditional one even when restricted to varieties, although it is
so if we demand that the varieties in question are proper.

The problem of birational classification of orbifolds is strictly
harder than that of algebraic varieties, and thus a complete classification is
currently out of reach. It is therefore natural to study the
classification problem upon restriction to subcategories (full or not) of the
category of orbifolds or DM stacks for which the classical problem is completely understood. For example, one may try to study the
birational classes of orbifolds $\mathcal{X}$ ``lifting'' a single given birational
class, i.e., the birational classification of orbifolds $\mathcal{X}$ whose
underlying good moduli space $X$ has a fixed birational type as an algebraic
variety. Alternatively, one may impose dimension constraints, group actions, et cetera.

Recently, Schmitt \cite{Schmitt} solved an instance of this problem, namely the
birational classification problem in the category of toric orbifolds with
respect to toric morphisms. The corresponding problem for varieties is trivial
here -- the only birational invariant of a complete toric variety is its
dimension -- but the orbifold problem remains interesting, with infinitely many
distinct orbifold birational classes. Using the correspondence between toric
orbifolds and stacky fans (see for example
\cite{Borisov_2004,fantechi2009smoothtoricdmstacks,Geraschenko_2014,ATheoryOfStaGillam2015}),
Schmitt gave a bijection between toric birational equivalence classes of toric
orbifolds and certain combinatorial objects which he called sublattice
colorings.

Our goal in this paper is twofold. On the one hand, we aim to explain the
geometric meaning of Schmitt's result and place it in what is, in our view,
its natural conceptual framework: Schmitt's sublattice colorings are precisely the fans of a generalized class of geometric objects which, while not traditional objects of algebraic geometry, arise naturally in logarithmic geometry. On the other hand, using this framework, we
can extend the result to other equivariant geometries -- namely compactified semiabelian schemes.
Both of those goals are achieved by studying this problem using the language of logarithmic geometry.

\subsection{Our results}
Let $S$ be a log smooth log
scheme (for example, a point, or a scheme $S$ with a normal crossings divisor),
and let $G$ be a semiabelian scheme over $S$. Assume that one of the two following
possibilities holds:

\begin{itemize}
	\item The generic fiber of $G$ is an abelian variety.
	\item $G$ is an algebraic torus over $S$.
\end{itemize}

\begin{definition}
	A log compactification of $\pi:G \to S$ is a diagram
	\[
	\begin{tikzcd}
		G \ar[r] \ar[dr,"\pi"] & X \ar[d,"f"] \\
		& S
	\end{tikzcd}
	\]
	such that
	\begin{itemize}
		\item $X$ is an orbifold containing $G$ as a dense open substack.
		\item $G$ acts on $X$ over $S$ extending the multiplication action of $G$ on itself.
		\item $f$ is log smooth and proper.
		\item If $G$ is an algebraic torus, then étale locally on $S$ the stack $X$ is isomorphic to $S\times Y$ for some toric DM stack $Y$.
	\end{itemize}
\end{definition}

Log compactifications of $G$ form a category $\mathcal{C}(G)$, where morphisms
$X \to Y$ are $G$-equivariant morphisms over $S$. An analogous subcategory
$\mathcal{C}_{\textup{rep}}(G)$ is obtained by demanding $f$ to be
representable. We remark here that equivariance of morphisms is a very strong condition; it forces morphisms in $\mathcal{C}_{\textup{rep}}(G)$ to be log modifications -- modifications with center disjoint from $G$ and which respect the induced stratification on $X$ -- and morphisms in $\mathcal{C}(G)$ to be log alterations, which are roughly speaking compositions of log modifications and root stacks along the strata.

We call two elements $X$, $Y$ in $\mathcal{C}(G)$ (resp.
$\mathcal{C}_{\textup{rep}}(G)$) $G$-birationally equivalent if there is a correspondence
\[
\begin{tikzcd}
	& Z \ar[rd,"p"] \ar[ld, swap, "q"] & \\ X& &Y
\end{tikzcd}
\]
with $p,q$ proper, birational and representable morphisms in $\mathcal{C}(G)$
(resp. $\mathcal{C}_{\textup{rep}}(G)$). We write
$\textup{Bir}(\mathcal{C}(G))$ (resp. $\textup{Bir}(\mathcal{C}_{\textup{rep}}(G))$)
for the set of $G$-birational classes in $\mathcal{C}(G)$ (resp. $\mathcal{C}_{\textup{rep}}(G)$).
In particular, contrary to the classical
results on the birational classification of universal compactified Jacobians, e.g. \cite{farkas2011classificationuniversaljacobiansmoduli},
our notion of birational equivalence is relative to the base $S$. In the case when $G=T$ is an
algebraic torus over a point, we have:
\begin{itemize}
	\item $\textup{Bir}(\mathcal{C}_{\textup{rep}}(T))=\left\{\bullet \right\}$.
	\item (Schmitt): $\textup{Bir}(\mathcal{C}(T))=\{\textup{Sublattice Colorings}\}$.
\end{itemize}
We extend the classification to arbitrary $G$ as above.

Our classification goes through a generalized class of geometric objects
which exist only in the setting of logarithmic geometry.

\begin{definition}
	A DM log stack over $S$ is a stack for the strict fppf
	topology (i.e., the topology whose covers are universally surjective strict fppf morphisms)
	on $\LogSch/S$ which has a log étale cover by a DM stack over $S$
	with a log structure and whose diagonal is representable by log algebraic spaces
	over $S$ with a log structure. We say that a DM log stack is an
	algebraic log space if it is fibered in setoids.
\end{definition}

\begin{remark}
	In \cite{LogarithmicAbeKajiwa2015}, algebraic log spaces are
	called log algebraic spaces in the second sense.
\end{remark}

We may then extend our category $\mathcal{C}(G)$ (resp.
$\mathcal{C}_{\textup{rep}}(G)$) to logarithmic versions
$\mathcal{C}(G)^{\log}$ (resp. $\mathcal{C}_{\textup{rep}}(G)^{\log}$) by
allowing $X$ to be a DM log stack (resp. algebraic log space). We call the additional objects introduced \emph{generalized log compactifications} of $G$. See Definition \ref{def:generalized_log_cpt}
for a precise definition of the two enlarged categories. It is easy to show that the additional objects do not change the birational classification:
\begin{align*}
	\textup{Bir}(\mathcal{C}(G)^{\log}) = \textup{Bir}(\mathcal{C}(G)) \\
	\textup{Bir}(\mathcal{C}_{\textup{rep}}(G)^{\log}) = \textup{Bir}(\mathcal{C}_{\textup{rep}}(G)).
\end{align*}

There is however an advantage in passing to the enlarged log category: every
birational class is now represented by a unique \emph{minimal} object. This is
interesting even for toric varieties. When $G=T$, the category
$\mathcal{C}_{\textup{rep}}(T)^{\log}$ has a minimal object, the logarithmic torus
$T_{\log}$, defined as the stack (in fact, sheaf) whose functor of points is given by
\[
\Hom_{\LogSch/S}(B,T_{\log}) = \Hom(M,M_{B}^{\gp}(B))
\]
where $M_B$ is the log structure of $B$ and $M$ is the character lattice of
$T$. Every toric variety $X$ maps to $T_{\log}$, and in a precise sense, this
map is a log modification. Thus, in the identification
\[
\textup{Bir}(\mathcal{C}_{\textup{rep}}(T))=\left\{\bullet \right\}
\]
the unique element can be canonically identified with the logarithmic torus $T_{\log}$.
This is in contrast with the standard minimal model program where, for example,
the minimal toric surfaces are the Hirzebruch surfaces or $\mathbb{P}^2$ (both
of which are in this setting blowups of $\mathbb{G}_{m,\log}^2$).

This situation is general. For any $G$ as above, the category $\mathcal{C}(G)^{\log}$
has a minimal object $G^{\log}$ and every map $X \to G^{\log}$
in $\mathcal{C}_{\textup{rep}}(G)^{\log}$ is birational. This implies
\[
\textup{Bir}(\mathcal{C}_{\textup{rep}}(G)) = \{[G^{\log}]\}.
\]
Orbifold birational classes are thus born on account of the Kresch-Tschinkel definition of
birational equivalence, which demands that the maps be representable. Any $X$ in $\mathcal{C}(G)$ maps to $G^{\log}$, but the corresponding map may be a log alteration, in which case $X$ is no longer in the birational equivalence class of $G^{\log}$. We thus
obtain
\[
\textup{Bir}(\mathcal{C}(G)^{\log}) = \{\textup{Minimal generalized log compactifications of } G\}
\]
Minimal means that the induced map $X \to G^{\log}$ does not factor through $X \to Y \to G^{\log}$ with $X \to Y$ representable --
see Definition \ref{def:minimal_cpt} for the precise condition.

The connection with combinatorics is crucial: it
gives an effective way to determine which of the generalized log compactifications above are minimal.
Recall that every toric variety $X$ has a fan $\Sigma_X$; similarly,
every toric orbifold $X$ has a ``stacky fan" $\Sigma_X$ (cf.
\cite{Borisov_2004,fantechi2009smoothtoricdmstacks,Geraschenko_2014,ATheoryOfStaGillam2015}). In fact, the same is true
for any $X$ in $\mathcal{C}(G)$ (as a special case of Theorem \ref{thm:trop_corresp}). What is more remarkable is that
every $X \in \mathcal{C}(G)^{\log}$ has a tropicalization (see Theorem \ref{thm:trop_corresp}), which however is no
longer a usual collection of convex rational polyhedral cones, but rather what
we call a \emph{tropical compactification} (see Definition \ref{def:partial_compact}). In other words, we have a ``tropicalization" functor
\begin{align*}
	\Sigma_?:\mathcal{C}(G)^{\log} &\to \{ \textup{tropical compactifications} \}\\
	X &\rightsquigarrow \Sigma_X.
\end{align*}
Corollary \ref{cor:trop_corresp_log} shows that the tropicalization functor is an equivalence of categories.
Furthermore, it makes sense to consider maps from a logarithmic scheme to a given tropical compactification $\Sigma$ (see Section \ref{sec:tropical_corresp} for an explanation), and Theorem \ref{thm:trop_corresp} asserts that there is an inverse to the tropicalization functor which we call the \emph{geometric realization functor}:
\begin{align*}
	\mathcal{R}_?:\{\textup{tropical compactifications}\} &\to \mathcal{C}(G)^{\log} \\
	\Sigma &\rightsquigarrow G^{\log} \times_{\Sigma_{G^{\log}}} \Sigma.
\end{align*}
Under this correspondence, morphisms to a given $X$ in $\mathcal{C}(G)^{\log}$ correspond to taking a finite index sublattice in a subdivision of $\Sigma_X$; and the representable morphisms are the subdivisions. Combinatorially, the result
\[
\textup{Bir}(\mathcal{C}_{\textup{rep}}(G))=\{[G^{\log}]\}
\]
can be interpreted as saying that the fan of any $X \in \mathcal{C}(G)$ is obtained by compatibly taking a finite index sublattice in some subdivision of $\Sigma_{G^{\log}}$. To find the minimal log compactifications of $X \to G^{\log}$, it suffices to find such choices of finite index sublattices in subdivisions which do not arise as subdivisions of a coarser object.

In Section \ref{sec:combinat_problem} we solve this problem by showing that there is a functor $\Sigma\mapsto M_{\Sigma}$ which sends $\Sigma$ to
the minimal object in the birational equivalence class of $\Sigma$. The space $M_{\Sigma}$ has slightly weaker representability
properties and we call it a generalized compactification.
We obtain a functor
\begin{align*}
	M_{?}: \{\textup{tropical compactifications}\}&\to \{\textup{generalized tropical compactifications}\}\\
	\Sigma &\to M_{\Sigma}.
\end{align*}
This map is a retraction onto the subcategory of minimal objects. When $G=T$, the minimal objects are exactly Schmitt's sublattice colorings (see Corollary \ref{cor:lattice_colorings}). This result can be summarized by the following picture, explained below:

	\noindent\resizebox{\textwidth}{!}
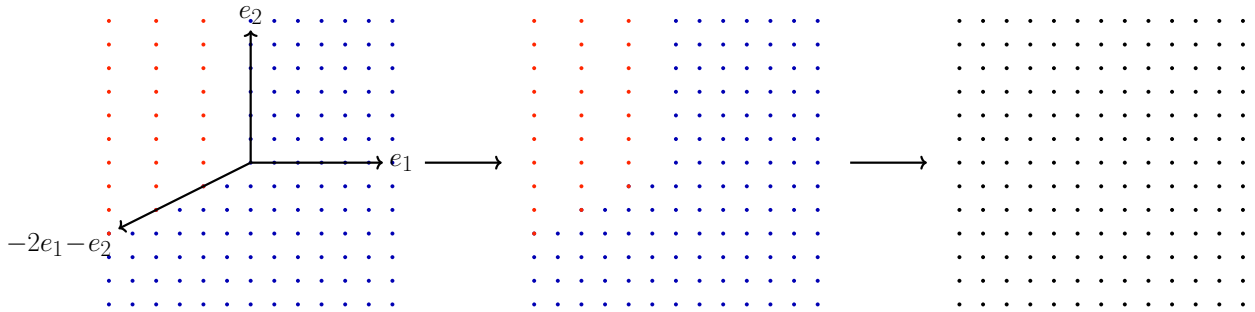
{\begin{tikzpicture}[scale=1.5]

  \def\lo{-3.0}
  \def\hi{3.0}
  \def\cliplo{-3.15}
  \def\cliphi{3.15}


  \begin{scope}
    \clip (\cliplo,\cliplo) rectangle (\cliphi,\cliphi);

    \foreach \m in {0,1,...,6} {
      \foreach \n in {0,1,...,6} {
        \fill[blue!70!black] (0.5*\m,0.5*\n) circle (1.2pt);
      }
    }

    \foreach \m in {-12,-11,...,6} {
      \foreach \n in {-6,-5,...,-1} {
        \pgfmathtruncatemacro{\bound}{2*\n}
        \ifnum\m<\bound
        \else
          \fill[blue!70!black] (0.5*\m,0.5*\n) circle (1.2pt);
        \fi
      }
    }

    \foreach \b in {1,2,...,6} {
      \foreach \a in {0,1,...,12} {
        \pgfmathtruncatemacro{\ly}{\a-\b}
        \fill[red!70!orange] (-\b,0.5*\ly) circle (1.2pt);
      }
    }
  \end{scope}

  \draw[line width=2pt, black, ->] (0,0) -- (2.8,0);
  \draw[line width=2pt, black, ->] (0,0) -- (0,2.8);
  \draw[line width=2pt, black, ->] (0,0) -- (-2.8,-1.4);

  \node[right, font=\Huge] at (2.85,0) {$e_1$};
  \node[above, font=\Huge] at (0,2.85) {$e_2$};
  \node[below left, font=\Huge] at (-2.85,-1.4) {$-2e_1\!-\!e_2$};

  \draw[->, line width=2pt, shorten >=8pt, shorten <=8pt] (3.5,0) -- (5.5,0);

  \begin{scope}[shift={(9,0)}]
    \begin{scope}
      \clip (\cliplo,\cliplo) rectangle (\cliphi,\cliphi);

      \foreach \m in {0,1,...,6} {
        \foreach \n in {0,1,...,6} {
          \fill[blue!70!black] (0.5*\m,0.5*\n) circle (1.2pt);
        }
      }

      \foreach \m in {-12,-11,...,6} {
        \foreach \n in {-6,-5,...,-1} {
          \pgfmathtruncatemacro{\bound}{2*\n}
          \ifnum\m<\bound
          \else
            \fill[blue!70!black] (0.5*\m,0.5*\n) circle (1.2pt);
          \fi
        }
      }

      \foreach \b in {1,2,...,6} {
        \foreach \a in {0,1,...,12} {
          \pgfmathtruncatemacro{\ly}{\a-\b}
          \fill[red!70!orange] (-\b,0.5*\ly) circle (1.2pt);
        }
      }
    \end{scope}

  \end{scope}

  \draw[->, line width=2pt, shorten >=8pt, shorten <=8pt] (12.5,0) -- (14.5,0);

  \begin{scope}[shift={(18,0)}]
    \begin{scope}
      \clip (\cliplo,\cliplo) rectangle (\cliphi,\cliphi);

      \foreach \m in {-6,-5,...,6} {
        \foreach \n in {-6,-5,...,6} {
          \fill[black] (0.5*\m,0.5*\n) circle (1.2pt);
        }
      }
    \end{scope}
  \end{scope}

\end{tikzpicture}}

The lattice on the right depicts $\Sigma_{T^{\log}}$, and the fan on the left
is the fan of an element $X$ of $\mathcal{C}(T)$, i.e. a toric orbifold. The
red and blue colors represent finite index sublattices of the lattice on the
right. Certainly there is a map $\Sigma_X \to \Sigma_{T^{\log}}$, but it is not
representable because the lattice structures disagree on at least one cone.
Coarsening the subdivision in the unique way that respects the sublattice
structure, we obtain the object in the middle, which (i) maps to $\Sigma_{T^{\log}}$, albeit the map is not representable, as the integral structures disagree (and hence the corresponding algebraic object
is not birationally equivalent to $T^\log$) (ii) is such that $\Sigma_X$ is obtained from it
by a subdivision that does not alter the integral structure (iii)
cannot be further coarsened. Thus, in our language, it is $M_{\Sigma_X}$, and
its geometric realization $\mathcal{R}(M_{\Sigma_X})$ is a minimal log
compactification of $T^{\log}$, which is the canonical representative of the birational class of $X$ in
$\textup{Bir}(\mathcal{C}(G))$. Our main result is formally as follows:

\begin{theorem}
	Let $S$ be a log smooth log scheme, and $G^{\log}$ be a log semiabelian variety over $S$ with dense open semiabelian scheme $G\subseteq G^{\log}$.
	Assume that either:
\begin{itemize}
	\item $G$ is the Jacobian of a log smooth curve over $S$ or
	\item $G$ is an algebraic torus.
\end{itemize}
	Then we have a diagram
	\[
	\begin{tikzcd}
		\mathcal{C}(G)^{\log}	\ar[d] \ar[r] & \ar[l] \parbox{5cm}{\centering log alterations $X \to G^{\log}$} \ar[r,"\Sigma"] \ar[d,"\mathcal{R} \circ M_{\Sigma_{?}}"] & \ar[l,"\mathcal{R}"] \parbox{5cm}{\centering tropical compactifications $\Sigma_X$} \ar[d,"M_{?}"] \\
		\textup{Bir}(\mathcal{C}(G)^{\log})	\ar[r]   & \ar[l] \parbox{5cm}{\centering minimal generalized log compactifications} \ar[r,"\Sigma"] & \ar[l,"\mathcal{R}"] \parbox{5cm}{\centering minimal generalized tropical compactifications}
	\end{tikzcd}
	\]
	where the middle- and right-hand vertical maps identify a birational class with its minimal representative and the horizontal maps are isomorphisms.
\end{theorem}

The proof is given in Corollary \ref{cor:main_thm_tor} for the case when $G$ is an algebraic torus
and in Corollary \ref{cor:main_thm_jac} for the case when $G$ is a Jacobian. For general $G$, our proof would go through verbatim if we knew that every abstractly defined log compactification of $G$ was a log alteration of $G^{\log}$. In other words, the theorem holds, with the same proofs, on the subcategory $\mathcal{D}(G)^{\log} \subset \mathcal{C}(G)^{\log}$ whose objects are log alterations of $G^{\log}$. The reason we are able to prove the theorem for abstractly defined log compactifications for tori and Jacobians is that we know that they satisfy the \emph{Néron mapping property}, whose formulation for the logarithmic Jacobian is as follows (and which is the same for logarithmic tori with a much easier proof):

\begin{theorem}[Theorem 1.7 in \cite{ModelsOfJacobHolmes2020}]
	Let $S$ be a log smooth log scheme and $C\to S$ a proper vertical log smooth log curve over $S$.
	Write $U\subseteq S$ for the dense open subset over which $C|_U\to U$ is smooth.
	The logarithmic Jacobian $\LogJac(C/S)$ satisfies the Néron mapping
	property for log smooth morphisms. That is, for every log smooth morphism
	$X\to S$, the restriction morphism
	$$\Hom_S(X,\LogJac(C/S))\to \Hom_U(X|_U,\LogJac(C|_U/U))$$
	is an isomorphism.
\end{theorem}

This theorem allows us to prove that every log compactification $X$ of the Jacobian $G$ of a log
smooth curve admits a $G$-equivariant morphism $X\to \LogJac(C/S)=G^{\log}$ (see Corollary \ref{cor:let_over_logjac}), and hence is in $\mathcal{D}(G)^{\log}$.

The Néron mapping property for more general log abelian varieties $A\to S$ is
not proven in the literature at the time of writing. However, it is expected to
hold and we formulate it as a conjecture:

\begin{conjecture}
	\label{conj:neron_mapping}
	Let $A\to S$ be a family of log abelian varieties over a log smooth base $S$. Let $U\subseteq S$ be the dense
	open subset of $S$ on which the log structure is trivial. Then $A$ satisfies the Néron mapping property for
	log smooth morphisms, i.e., if $X\to S$ is a log smooth morphism, then the restriction map:
	$$\Hom_S(X,A)\to \Hom_U(X|_U,A|_U)$$
	is an isomorphism.
\end{conjecture}

Assuming this conjecture, the results of this paper generalize to arbitrary families of semiabelian varieties
admitting a compactification by a log abelian variety.
\subsection{Organization of the paper.} The goal of sections \ref{sec:first_results}, \ref{sec:tropical_corresp} and \ref{sec:translate_geom} is to describe the tropicalization and geometric realization functors defined above. Many of the technical results required for this are currently missing from the literature in the necessary generality and are therefore developed within these sections. The expert comfortable with the foundations of tropicalization maps in logarithmic geometry may wish to skip some of the more technical results in these sections. In Section \ref{sec:first_results} we introduce the precise definition of a
logarithmic compactification and prove basic results, which require some
technical machinery related to stacks over the category of log schemes.
In Sections \ref{sec:tropical_corresp} and \ref{sec:translate_geom},
we establish the connection between the combinatorial and the logarithmic
category, which requires us to generalize the notion of an Artin fan
\cite{LogarithmicTauPandha2024} to our setting and prove several
generalizations of the results of \cite{SkeletonsAndFAbramo2015}.

\subsection*{Acknowledgements}

The authors would like to thank Johannes Schmitt for many discussions regarding
cone stacks and their combinatorics. They also thank Samir Canning, Lycka
Drakengren and Aitor Iribar Lopez for their comments regarding the present paper.

\subsection{Conventions}

\begin{itemize}
	\item All logarithmic schemes are assumed to be fs (fine and saturated) and locally of finite type. We denote the category
		of fs log schemes over a fixed connected Noetherian log smooth log scheme $S$ by $\LogSch/S$.
		We endow $\LogSch/S$ with the strict fppf topology (i.e., the topology whose covers are universally surjective strict
		fppf morphisms).
	\item All schemes are assumed to be schemes locally of finite type over an algebraically closed
		field $K$ of characteristic 0.
	\item Given an fs monoid (resp. lattice) $M$, we will denote by $M^{\vee}=N=\Hom(M,\mathbb{N})$ (resp. $M^{\vee}=N=\Hom(M,\mathbb{Z})$) its dual monoid (resp. dual lattice). Similarly, $N'$ is the dual of $M'$, etc. Given a lattice (i.e., a finitely-generated free abelian group)
		$N$, write $N_{\mathbb{R}}$ or $N\otimes \mathbb{R}$ for the real vector space $N\otimes _{\mathbb{Z}}\mathbb{R}$.
	\item We will often encounter a situation in which a rational polyhedral cone $\sigma$ is embedded in a vector space
		$N_{\mathbb{R}}$ for some lattice $N$. To lighten notation, we will often write $\sigma\subseteq N$ instead of
		$\sigma\subseteq N_{\mathbb{R}}$.
	\item Throughout the paper we will use the notation in \cite{AModuliStackCavali2017} for cone stacks, cone spaces
		etc. We will freely switch between cone stacks and combinatorial cone stacks which is justified by
		\cite[Proposition 2.19]{AModuliStackCavali2017}.
	\item We use the conventions of \cite{LogarithmicTauPandha2024} regarding Artin fans, see \cite[Definition 43]{LogarithmicTauPandha2024}.
	\item We write $\RPC$ for the category of rational polyhedral cones. We write $\ShpMon$ for the category of sharp integral saturated monoids.
	\item We call a DM stack equipped with a logarithmic structure a log DM stack.
\end{itemize}

\section{From Compactifications to Artin Fans}
\label{sec:first_results}

We start by setting up the situation in which we will be working. For the remainder of the paper, we fix a connected Noetherian log smooth log scheme $S$.

\begin{definition}
    Consider the sheaf $\mathbb{G}_{m,\log}$ on $\LogSch/S$ defined by
    \begin{align*}
		\Hom(Z,\mathbb{G}_{m,\log}) = M_Z^{\gp}(Z)
    \end{align*}
\end{definition}

\begin{definition}
	A log torus $T$ is a sheaf of groups on $\LogSch/S$ which is strict étale locally on $S$ isomorphic to
	$\mathbb{G}_{m,\log}^{r}$ for some $r\ge 0$.
	A log semiabelian variety is a sheaf $X$ of groups on $\LogSch/S$ which is an extension of a principally polarized log abelian
	variety by a log torus, such that strict étale locally on $S$, the extension $X\in \Ext(A,\mathbb{G}_{m,\log}^{r})$
	(with $A$ a log abelian variety) lies in the image of
	$$\Ext(A,\mathbb{G}_m^{r})\to \Ext(A,\mathbb{G}_{m,\log}^{r}).$$
\end{definition}

\begin{definition}
	Given an algebraic torus $T$, let $r$ be its rank. We write
	$$T_{\log}:=T\times^{\mathbb{G}_m^r}\mathbb{G}_{m,\log}^r.$$
	It is clear that $T_{\log}$ is a log torus.
\end{definition}

\begin{remark}
	All log abelian varieties are assumed to be principally
	polarized. Therefore we will henceforth drop the prefix
	``principally polarized''.
\end{remark}

\begin{lemma}
	A log torus $T$ is a log semiabelian variety.
\end{lemma}

\begin{proof}
	The question is strict étale local, so we may assume that $T=\mathbb{G}_{m,\log}^{r}$. In this case,
	$T$ is the trivial extension of the trivial log abelian variety $0$ and hence in particular lies in
	the image of $\Ext(0,\mathbb{G}_{m}^{r})\to \Ext(0,\mathbb{G}_{m,\log}^{r})$.
\end{proof}

\begin{example}
	\label{expl:log_cpt}
	Let $G'\to S$ be a semiabelian variety. By
	\cite[Proposition 2.9]{DegenerationOfFaltin1990}, we can
	write $G'$ as an extension of a semiabelian variety $G$
	by an algebraic torus $T$, where $G$ is an abelian variety
	over the generic point of $S$. Assume that $G=G_A\subseteq A$
	is the maximal semiabelian subsheaf of a log abelian variety $A$, i.e., $G$ admits a compactification
	by a log abelian variety. By \cite[Theorem 7.4 (2)]{LogarithmicAbeKajiwa2008}
	the pullback map $\Ext(A,\mathbb{G}_m)\to \Ext(G,\mathbb{G}_m)$
	is an isomorphism, so we can extend $G'$ uniquely to an
	extension of $A$ by $T$. Taking the image of this extension
	under $\Ext(A,T)\to \Ext(A,T_{\log})$,
	we obtain a log semiabelian variety $X$.
\end{example}

The space $X$ constructed in the above example is an example of
a log compactification of $G$.

\begin{lemma}
	Let $X$ be a log semiabelian variety, which is an extension of a log abelian variety $A$ by a log torus $T_{\log}$. Then there exists a unique
	maximal subsheaf $G_X\subseteq X$ represented by a commutative
	group scheme over $\underline{S}$ with connected fibers. When $X = T_{\log}$ is a log torus, $G_X = T$ is a torus, and when $X=A$ is a log abelian variety, $G_X$ is the semiabelian subscheme of $X$
	from \cite[4.4]{LogarithmicAbeKajiwa2008}. In general, $G_X$ is an extension of $G_A$ by $T$.
\end{lemma}

\begin{proof}
	Note that connectedness and representability satisfy strict étale descent (for descent of representability, use that any
	abelian group object in the category of algebraic spaces is a group scheme by \cite[Theorem 1.9]{DegenerationOfFaltin1990}.
	For descent of connectedness, use that $S$ is irreducible and $G_X\to S$ is smooth.) and hence
	by uniqueness, we may assume that $X\in \Ext(A,\mathbb{G}_{m,\log}^{r})$ lies in the image of
	$\Ext(A,\mathbb{G}_{m}^{r})\to \Ext(A,\mathbb{G}_{m,\log}^{r})$. Let $X'\in \Ext(A,\mathbb{G}_{m}^{r})$
	such that $X=X'_{\log}$ (i.e., the pushout $X'\times^{\mathbb{G}_m^r}\mathbb{G}_{m,\log}^{r}$) and write $G_A\subseteq A$ for the semiabelian subscheme of $A$
	from \cite[4.4]{LogarithmicAbeKajiwa2008}. We set $G_X:=X'\times_{A}G_A$. This is an extension
	of $G_A$ by $\mathbb{G}_m^{r}$ and hence in particular representable by a commutative group scheme.
	To show that $G_X$ is the unique maximal such object, it remains to show that any morphism from
	a connected commutative group scheme $H$ over $\underline{S}$ with connected fibers to $[X/G_X]$ is trivial. Let $H\to [X/G_X]$
	be an arbitrary morphism. By construction of $G_X$, we have a projection map $[X/G_X]\to [A/G_A]$. The
	composition $H\to [X/G_X]\to [A/G_A]$ is trivial by \cite[9.2]{LogarithmicAbeKajiwa2008}. Hence we
	obtain a lift of $H\to [X/G_X]$ to $H\to [X/G_X]\times_{[A/G_A]}\{0\}=[\mathbb{G}_{m,\log}^{r}/\mathbb{G}_{m}^{r}]$.
	By \cite[Lemma 6.1.1]{LogarithmicAbeKajiwa2008}, this map is trivial, whence the claim.
\end{proof}

From now on we will fix a log semiabelian variety $X$ over $S$ with subsheaf $G_X\subseteq X$ as above.

\begin{definition}
	\label{def:log_cpt}
	A partial log compactification of $G_X\to S$ is a $G_X$-equivariant
	log étale monomorphism $Y\to X$ representable by (fs) log DM stacks
	of finite presentation, such that $Y$
	admits a section $S\to Y$ lifting the unit section of $X$. A partial log compactification $Y\to X$ of $G_X\to S$ is called a log compactification if $Y\to X$ is proper.
\end{definition}

\begin{definition}
	Morphisms of partial log compactifications of $G_X\to S$ are $G_X$-equivariant
	morphisms over $X$.
\end{definition}

\begin{lemma}
	\label{lem:morphs_commute_sec}
	The sections in Definition \ref{def:log_cpt} are unique
	and morphisms of partial log compactifications preserve this section.
\end{lemma}

\begin{proof}
	The uniqueness of the section follows from the fact that
	$Y\to X$ is a monomorphism.
	Let $f:Y\to Y'$ be a morphism of partial log compactifications of
	$G_X\to S$. Since $f$
	is a morphism over $X$, the composition of the section
	$S\to Y$ with $f$ is a section of $Y'$ over the unit section
	of $X$ and the result follows from uniqueness.
\end{proof}

The goal of this paper is to study the $G_X$-equivariant birational
classification problem for partial log compactifications of $G_X\to S$.
For convenience, we repeat here the precise definition of
$G_X$-equivariant birational equivalence:

\begin{definition}
	Let $Y_1,Y_2\to X$ be partial log compactifications of $G_X\to S$. We
	say that $Y_1$ and $Y_2$ are birationally equivalent if there exists a
	diagram:
	$$\begin{tikzcd}
		&Y_3\arrow[dr]\arrow[dl]&\\
		Y_1&&Y_2
	\end{tikzcd}$$
	of proper morphisms of partial log compactifications of $G_X\to S$
	representable by log algebraic spaces. A morphism of partial
	log compactifications is said to be a birational equivalence
	if it is proper and representable by log algebraic spaces.
\end{definition}

\begin{remark}
	Since any fppf descent datum for log algebraic spaces is
	effective by \cite[0ADV]{stacks-project}, representability
	is local in the strict fppf topology.
	A fortiori, it is local in the strict étale topology.
\end{remark}

From the above definition it is not obvious that the morphisms in the diagram
are isomorphisms on dense open subsets. This is the content of the following lemma:

\begin{lemma}
	\label{lem:bir_is_bir}
	Let $S$ be a log smooth log scheme, $X\to S$ a log semiabelian variety and
	let $Y\to X$ be a partial log compactification of $G_X\to S$. Let $U\to Y$ be a morphism
	from a log scheme such that the composition $U\to Y\to S$ is log smooth. Then
	$U\times_Y G_X\to U$ has dense image in $U$. Moreover, given a morphism $Y\to Y'$ of
	partial log compactifications, the induced morphism $U\times_Y G_X\to U$ lifts to
	the identity $U\times_Y G_X\to U\times_{Y'}G_X$.
\end{lemma}

\begin{proof}
	Note that $U\times_Y G_X\cong U\times_X G_X$ since $Y\to X$ is a monomorphism. Hence it suffices
	to consider the case $Y=X$. Let $Y\to Y'$ be a morphism of partial log compactifications. Note that $G_X$
	is a torsor under itself and since $Y\to Y'$ preserves the zero-section by Lemma \ref{lem:morphs_commute_sec}, it follows that $Y\to Y'$
	restricts to the identity $G_X\subseteq Y\to G_X\subseteq Y'$. In particular, the same holds after base-change
	to $U$. It remains to show density of $G_X$:

	Clearly, we may assume that $U$ is non-empty. Since $U\to S$ is log smooth, the locus
	in $U$ lying over the locus in $S$ with trivial log structure is dense in $U$. Hence
	we may assume that $S$ has trivial log structure. In this case, $X$ is an extension of
	an abelian variety $A$ by $\mathbb{G}_{m,\log}^{r}$ for some $r\ge 0$. The question is local
	on $A$, so replacing $S$ with $A$ and working locally on $S$ we may assume that $X=\mathbb{G}_{m,\log}^{r}$.
	We must show that $U\times_{\mathbb{G}_{m,\log}^{r}}\mathbb{G}_{m}^{r}\subseteq U$ is dense.
	Since
	$$U\times_{\mathbb{G}_{m,\log}^{r}}\mathbb{G}_{m}^{r}\cong (U\times_{\mathbb{G}_{m,\log}}\mathbb{G}_{m})\times\dots\times(U\times_{\mathbb{G}_{m,\log}}\mathbb{G}_{m}),$$
	it suffices to consider the case $r=1$. By \cite[Lemma 2.2.7.3]{TheLogarithmicMolcho2022},
	there is a log modification
	$\mathbb{P}^{1}\to \mathbb{G}_{m,\log}$. The fiber product
	$\mathbb{G}_{m}\times_{\mathbb{G}_{m,\log}}\mathbb{P}^{1}$
	is the dense open torus $\mathbb{G}_m$
	in $\mathbb{P}^{1}$. Let $U':=U\times_{\mathbb{G}_{m,\log}}\mathbb{P}^{1}$. Then $U'\to U$ is a log modification,
	so in particular $U'$ is log smooth. It follows that the preimage of $\mathbb{G}_{m}$ under $U'\to \mathbb{P}^{1}$ is dense
	in $U'$. Since $U'\times_{\mathbb{P}^{1}}\mathbb{G}_{m}\to U$ factors through $U\times_{\mathbb{G}_{m,\log}}\mathbb{G}_{m}$,
	it follows that $U\times_{\mathbb{G}_{m,\log}}\mathbb{G}_{m}$ is dense in $U$.
\end{proof}

To give a satisfying answer to the birational classification problem
for partial log compactifications, we establish a link to combinatorics. Recall
that for log abelian varieties, this link is achieved by studying the tropical
abelian variety $[A/G_A]$ associated with $A$. Motivated by this construction,
we define:

\begin{definition}
	Let $X$ be a log semiabelian variety. If $Y\to X$ is a partial log
	compactification of $G_X\to S$, we say that $[Y/G_X]$ is a partial log compactification
	of $[X/G_X]$. Morphisms of partial log compactifications are morphisms over $[X/G_X]$ preserving
	the lift of the zero-section $S\to [Y/G_X]$ induced by the lift $S\to Y$.
	The association $Y\mapsto [Y/G_X]$
	defines a functor from the category of partial log
	compactifications to the category of partial log compactifications of $[X/G_X]$.
	We call a partial log compactification a log compactification of $[X/G_X]$
	if $Y$ is a log compactification of $G_X\to S$, or equivalently, if $[Y/G_X]\to [X/G_X]$
	is proper.
\end{definition}

\begin{lemma}
	\label{lem:equiv_cat}
	The functor $Y\mapsto [Y/G_X]$ is an equivalence of categories. An inverse
	is given by sending a partial log compactification $F$ of $[X/G_X]$
	to $F\times_{[X/G_X]}X$.
\end{lemma}

\begin{proof}
	First, let $[Y/G_X]$ be a partial log compactification of $[X/G_X]$.
	Then:
	$$([Y/G_X]\times_{[X/G_X]}X)/G_X\cong[Y/G_X]\times_{[X/G_X]}[X/G_X]\cong [Y/G_X]$$
	showing that $\ast\mapsto [(\ast\times_{[X/G_X]}X)/G_X]$ is the
	identity functor. Similarly, if $Y$ is a partial log compactification
	of $G_X\to S$, the space $([Y/G_X])\times_{[X/G_X]}X$ is the
	base-change of the universal $G_X$-bundle $X\to [X/G_X]$
	over $[Y/G_X]$ and hence canonically isomorphic to $Y$.
\end{proof}

\begin{remark}
    The space $[Y/G_X]$ is commonly referred to as the \emph{relative Artin fan} or \emph{Olsson fan} of $Y$ in the literature.
\end{remark}

\begin{lemma}
	\label{lem:bir_on_quot}
	A morphism $Y_1\to Y_2$ of partial log compactifications of $G_X\to S$
	is a birational equivalence if and only if the induced morphism
	$[Y_1/G_X]\to [Y_2/G_X]$ is proper and representable by log algebraic spaces.
\end{lemma}

\begin{proof}
	Assume first that $[Y_1/G_X]\to [Y_2/G_X]$ is proper and representable by log algebraic spaces.
	Then, taking the base-change over $Y_2$, the same is true for $Y_1\to Y_2$, i.e.,
	$Y_1\to Y_2$ is a birational equivalence.
	For the converse, let $Z\to [Y_2/G_X]$ be a morphism. Strict étale locally on $Z$ there exists a lift
	$Z\to Y_2$. Since properness and representability are strict étale local on the target,
	we may assume that $Z$ admits a global lift. In this case, we have $Z\times_{[Y_2/G_X]}[Y_1/G_X]\cong Z\times_{Y_2}Y_1$
	and hence after replacing $Z$ with a strict étale cover, we have that $Z\times_{[Y_2/G_X]}[Y_1/G_X]$ is a log
	algebraic space, proper over $Z$.
\end{proof}

\begin{lemma}
	A morphism $Y\to [X/G_X]$ from a log stack $Y$ is a partial log compactification
	of $[X/G_X]$ if and only if $Y\to [X/G_X]$ is
	a log étale monomorphism representable by log DM stacks of finite presentation such that
	there exists a section $S\to Y$ over the zero-section $S\to [X/G_X]$.
\end{lemma}

\begin{proof}
	First assume that $\overline{Y}\to [X/G_X]$ is a log étale monomorphism
	representable by log DM stacks of finite presentation which admits a section $S\to \overline{Y}$
	over the zero-section $S\to [X/G_X]$. Then $\overline{Y}\times_{[X/G_X]}X\to X$
	is a $G_X$-equivariant log étale monomorphism representable
	by log DM stacks of finite presentation by base-change.
	Moreover, the zero-section $S\to X$ and the section $S\to \overline{Y}$
	over the zero-section of $[X/G_X]$ define a morphism
	$S\to \overline{Y}\times_{[X/G_X]}X$, i.e., $\overline{Y}\times_{[X/G_X]}X$ is a partial log compactification
	of $G_X\to S$. Since $\overline{Y}\cong [(\overline{Y}\times_{[X/G_X]}X)/G_X]$, we
	get that $\overline{Y}$ is a partial log compactification of $[X/G_X]$.
	Conversely assume that $[Y/G_X]$ is a partial log compactification
	of $[X/G_X]$. Note that strict étale locally any morphism $Z\to [X/G_X]$
	lifts to a morphism $Z\to X$ and $Z\times_{[X/G_X]}[Y/G_X]\cong Z\times_{X}Y$.
	Since being a log étale monomorphism representable by log DM stacks of finite presentation
	is strict étale local on the base, we get that $[Y/G_X]\to [X/G_X]$ satisfies these
	conditions. Finally, note that the composition $S\to Y\to [Y/G_X]$ of the zero-section
	$S\to Y$ with the projection to $[Y/G_X]$ lies over the zero-section $S\to [X/G_X]$, finishing
	the proof.
\end{proof}

In the next section we will see that partial log compactifications of $[X/G_X]$ have a
combinatorial description similar to the description of Artin fans as
stacks over the category of rational polyhedral cones.

We end this section by showing that partial log compactifications are DM log stacks:

\begin{lemma}
	\label{lem:part_cptifications_are_nice}
	Let $X$ be a log semiabelian variety, and let $Y$ be a partial log compactification of $G_X\to S$. Then:
	\begin{itemize}
		\item $Y$ has a log étale cover by log schemes, i.e., a universally surjective log étale morphism $U\to Y$ with $U$ a log scheme.
		\item The diagonal of $Y$ is representable by log algebraic spaces of finite presentation.
	\end{itemize}
	In particular, any partial log compactification is a DM log stack.
\end{lemma}

\begin{proof}
	Note that these properties are strict étale local on $S$ (see also
	\cite[Theorem 1.1 (3)]{fortman2025descentalgebraicstacks})
	and hence
	we may assume that $X$ is an extension of a log abelian variety
	$A$ by $\mathbb{G}_{m,\log}^{r}$ which is obtained from an
	extension of $A$ by $\mathbb{G}_{m}^{r}$.

	We start by reducing both properties to the case $Y=X$.

	This is clear for the existence of a cover, since $Y\to X$
	is log étale and representable by log DM stacks, which have
	strict étale covers by log schemes.

	For the representability of the diagonal, note that since
	$Y\to X$ is a monomorphism, we have a fiber diagram:
	$$\begin{tikzcd}
		Y\arrow[r]\arrow[d]&X\arrow[d]\\
		Y\times_S Y\arrow[r]&X\times_S X,
	\end{tikzcd}$$
	where the vertical maps are the diagonal morphisms. Hence
	by base-change it suffices to show that $X\to X\times_S X$ is
	representable by log algebraic spaces of finite presentation.

	We reduce to showing the claims in the special case $Y=X$.
	We show first that $X$ has a log étale cover by log schemes.
	By \cite[Theorem 10.4]{LogarithmicAbeKajiwa2015}, we know that $A$ has a log étale cover by log algebraic spaces. Since
	the extension of $A$ by $\mathbb{G}_{m,\log}^{r}$ comes from an extension
	of $A$ by $\mathbb{G}_{m}^{r}$ and $\mathbb{G}_{m}^{r}$-torsors trivialize
	strict étale locally, we are reduced to showing that $\mathbb{G}_{m,\log}^{r}$
	has a log étale cover by log algebraic spaces. Clearly, it suffices to consider the case
	$r=1$. In this case, by \cite[Lemma 2.2.7.3 (2)]{TheLogarithmicMolcho2022},
	we know that $\mathbb{G}_{m,\log}$ has a log étale cover by
	$\mathbb{P}^{1}$.

	For the representability of the diagonal, we must show that the diagonal
	$X\to X\times_S X$ is representable by log algebraic spaces of finite presentation.
	For this, consider the following fiber diagram:
	$$\begin{tikzcd}
		X\times_A X\arrow[d]\arrow[r]&A\arrow[d]\\
		X\times_S X\arrow[r]&A\times_S A.
	\end{tikzcd}$$
	By \cite[Proposition 10.5]{LogarithmicAbeKajiwa2015}, the
	right-hand vertical morphism is representable by log algebraic
	spaces. The map $X\to X\times_S X$ factors as
	$X\to X\times_A X\to X\times_S X$. Hence it suffices to show
	that $X\to X\times_A X$ is representable by log algebraic spaces.
	For this, let $X'$ denote the $\mathbb{G}_{m}^{r}$-extension
	of $A$ such that
	$X=X'\times_{\mathbb{G}_{m}^{r}}\mathbb{G}_{m,\log}^{r}$.
	Let $X'\to X$ be the natural map. The multiplication map
	on $X$ induces a canonical morphism $\mathbb{G}_{m,\log}^{r}\times_S X'\to X$.
	Let $\mathbb{G}_{m,\log}^{r}\times_S X'\to X\times_A X'$
	be induced by the above map and the identity on the second factor.
	This is a map of $\mathbb{G}_{m,\log}^{r}$-torsors over $X'$
	and hence an isomorphism. Since $X'\to A$ is an fppf cover
	of $A$ and representability of the diagonal is fppf local by \cite[Tag 0ADV]{stacks-project},
	it suffices to show that the diagonal
	$\mathbb{G}_{m,\log}^{r}\to \mathbb{G}_{m,\log}^{r}\times_S \mathbb{G}_{m,\log}^{r}$
	is representable. Clearly, we may assume $r=1$. The diagonal
	is the base-change of the zero-section $0\to \mathbb{G}_{m,\log}$
	under the morphism $\mathbb{G}_{m,\log}\times \mathbb{G}_{m,\log}\to \mathbb{G}_{m,\log}$
	sending $(x,y)$ to $xy^{-1}$. Factor $0\to \mathbb{G}_{m,\log}$
	as $0\to \mathbb{G}_{m}\to \mathbb{G}_{m,\log}$. The first morphism
	is clearly representable and hence it suffices to show
	that $\mathbb{G}_{m}\to \mathbb{G}_{m,\log}$ is representable
	by morphisms of finite presentation.
	This is the base-change of the inclusion $0\to \mathbb{G}_{m,\trop}=[\mathbb{G}_{m,\log}/\mathbb{G}_{m}]$
	which is representable by morphisms of finite presentation \cite[Proposition 2.2.7.5]{TheLogarithmicMolcho2022}\footnote{The proposition
		statement only claims that it is representable by morphisms of finite type.
	However, the method of proof shows that it is in fact of finite presentation}.
\end{proof}

\subsection{Minimal log compactifications}

In this section we introduce so-called generalized partial log
compactifications. The motivation for doing so is that the minimal objects
in the birational equivalence classes of partial log compactifications will be
of this form (see Proposition \ref{prop:exists_unique_minimal_log_cpt} below
for details).

\begin{definition}
	\label{def:generalized_log_cpt}
	A generalized partial log compactification of $G_X\to S$ is a
	$G_X$-equivariant monomorphism $Y\to X$ such that there exists a partial
	log compactification $Y'\to X$ and a universally surjective
	proper log étale morphism $Y'\to Y$ over $X$
	representable by log algebraic spaces of finite type.
	Morphisms and birational
	equivalences of generalized partial log compactifications are defined as
	for partial log compactifications.
\end{definition}

\begin{definition}
	We say that a morphism of log schemes is universally surjective if it satisfies the
	condition in \cite[Proposition 2.2.4.2]{TheLogarithmicMolcho2022}. Moreover, we say
	that a morphism $X\to Y$ of log DM stacks is universally surjective if there exists a
	morphism $U\to X$ from a log scheme $U$ such that the base-change of $U\to X\to Y$ over
	any fs log scheme is universally surjective in the previous sense.
	Finally, we say that a morphism $X\to Y$ of stacks over $\LogSch/S$ representable
	by log DM stacks is universally surjective if its base-change
	over any fs log scheme is universally surjective in the previous sense.
\end{definition}

\begin{definition}
	\label{def:minimal_cpt}
	A generalized partial log compactification $Y\to X$ is called minimal if it satisfies descent with respect to log modifications.
	That is, if $T'\to T$ is a log modification of a log algebraic space
	$T$ (i.e., a universally surjective proper log étale monomorphism),
	then every descent datum for a morphism
	to $Y$ relative to $T'\to T$ is effective. Since log modifications are surjective monomorphisms, this equivalently means that for any morphism $T'\to Y$ there exists a unique morphism $T\to Y$ such that the diagram:
    $$\begin{tikzcd}
        T'\arrow[d]\arrow[r]&Y\\
        T\arrow[ur]&
    \end{tikzcd}$$
    commutes.
\end{definition}

To lighten the notation we will drop the word ``generalized''
from the term ``minimal generalized partial log compactification''.

\begin{proposition}
	\label{prop:exists_unique_minimal_log_cpt}
	Every birational equivalence class of generalized partial log
	compactifications contains a unique minimal partial log compactification,
	which is the terminal object in the birational equivalence
	class.
\end{proposition}

For the proof we need the following lemma:

\begin{lemma}
	Any proper morphism of generalized partial log compactifications
	representable by log DM stacks is universally surjective.
\end{lemma}

\begin{proof}
	Let $Y_1\to Y_2$ be a proper morphism of generalized partial
	log compactifications representable by log DM stacks.
	For $i\in \{1,2\}$, consider universally surjective proper morphisms
	$Y_{i}'\to Y_i$ representable by log algebraic spaces,
	where $Y_{i}'$ is a partial log compactification. Since universal
	surjectivity descends along arbitrary morphisms,
	it suffices to show that $Y_1'\to Y_2$ is universally surjective.
	Since $Y_2'\to Y_2$ is universally surjective and universal
	surjectivity is stable under composition, it suffices to
	show that the base-change $Y_1'\times_{Y_2}Y_2'=Y_1'\times_{X}Y_2'\to Y_2'$ is
	universally surjective. Replacing $Y_2$ with $Y_2'$ and
	$Y_1$ with $Y_1'\times_{X}Y_2'$, we hence reduce to the case
	when $Y_1$ and $Y_2$ are partial log compactifications.

	In this case, let $Z\to Y_2$ be a log étale morphism from a
	log scheme $Z$. Note that
	By Lemma \ref{lem:bir_is_bir} there is a dense open subset
	on which $Z\times_{Y_2}Y_1\to Z$ is an
	isomorphism. Hence by properness it is surjective.
\end{proof}

\begin{proof}[Proof of Proposition \ref{prop:exists_unique_minimal_log_cpt}]
	For uniqueness, let $Y_1,Y_2\to X$ be two minimal partial log compactifications and assume
	that
	$$\begin{tikzcd}
	 &Y_3\arrow[dr]\arrow[dl]&\\
		Y_1&&Y_2
	\end{tikzcd}$$
	is a birational equivalence.
	By the preceding lemma, the morphisms $Y_{3}\to Y_i$ for $i\in \{1,2\}$
	are universally surjective. They are clearly monomorphisms,
	since the compositions with the maps to $X$ are. Finally,
	they are log étale since the $Y_i$ are log étale over $X$ (see Lemma \ref{lem:let_when_both}). Hence by the minimality of
	$Y_1$ and $Y_2$ we get unique morphisms $Y_1\to Y_2$
	and $Y_2\to Y_1$ descending these. By uniqueness of the descended
	morphism, one sees that the compositions of these morphisms
	are the identity and hence $Y_1=Y_2$.

	For existence, let $Y$ be a generalized partial log compactification.
	Since $Y$ is birationally equivalent to a partial log compactification by
	the definition of a generalized partial log compactification, we may assume
	that $Y$ is a partial log compactification.
	Let $M_Y$ be the stack over $\LogSch/X$ which to
	$Z\to X$ associates the singleton if $Z\times_X Y\to Z$
	is a proper, universally surjective morphism representable by log algebraic spaces and
	the empty set otherwise. To see that this is indeed a stack,
	note that universal surjectivity, properness and representability are local in the
	strict fppf topology \cite[Tags 02L1, 02KV and 0ADV]{stacks-project}.
	We claim that $M_Y\to X$ is a generalized partial
	log compactification. First note that the map $M_Y\to X$ is clearly
	a monomorphism. Let $Z\to M_Y$ be a morphism. Then
	$Z\times_{M_Y}Y=Z\times_{X}Y\to Z$ is a proper, universally surjective morphism representable
	by log algebraic spaces. Since $Y\to X$ is log étale, so is $Z\times_X Y\to Z$,
	so $Y\to M_Y$ is a proper, universally surjective log étale morphism representable
	by log algebraic spaces, showing that $M_Y$ is a generalized partial log compactification.

	We now show that $M_Y$ is minimal. For this,
	let $Z'\to Z$ be a universally surjective, proper log étale monomorphism
	and assume we are given a morphism $Z'\to M_Y$ over $Z\to X$. Then
	$Z'\times_X Y\to Z'\to Z$ is universally surjective and since $Z'\times_X Y\to Z\times_X Y$
	is universally surjective and log étale, it follows from \cite{On_log_flat_des_Illusi_2013}
	that $Z\times_X Y\to Z$ is log étale. Since $Z\times_X Y\to Z$
	is of finite presentation, hence quasi-compact, it follows
	from \cite[Proposition 5.7 and Proposition 2.6]{Grothendieck_to_Hu_Xi_2025}
	that $Z\times_X Y\to Z$ is proper, so we obtain a morphism
	$Z\to M_Y$ descending $Z'\to M_Y$.

	Finally, it remains to show that $M_Y$ is the terminal object
	in the birational equivalence class. Equivalently, if $Y'$ is a generalized partial log compactification
	birationally equivalent to $Y$ then we must show that there exists a unique morphism $Y'\to M_Y$ over $X$.
	Uniqueness is clear since $M_Y\to X$ is a monomorphism. For existence, consider a birational equivalence
	$Y'\leftarrow Y''\rightarrow Y$. The map $Y''\to Y'$ is a proper
	morphism of generalized partial log compactifications representable by log
	algebraic spaces. By the preceding lemma, it is hence universally surjective. Since $M_Y$
	is minimal, the morphism $Y''\to Y\to M_Y$ descends uniquely to a morphism $Y'\to M_Y$, whence the claim.
\end{proof}

\section{Tropical correspondence theorem}
\label{sec:tropical_corresp}

In this section we continue our goal of converting the birational
classification problem for partial log compactifications into a combinatorial
problem. The tropical counterparts of log semiabelian varieties are
called split families of tropical semiabelian varieties and are defined below
(Definition \ref{def:trop_semiab}) after some initial setup:

\begin{definition}
	\label{def:trop_torus}
	Let $N$ be a lattice. We write $T_N^{\trop}$ for the stack
	over $\RPC$ which to $(\sigma,N_{\sigma})\in \RPC$ assigns
	the set of homomorphisms $N_{\sigma}\to N$. We also write $\mathbb{G}_{m,\trop}:=T_{\mathbb{Z}}^{\trop}$.
	The group structure on $N$ induces a group structure on $T_N^{\trop}$.
	A morphism $X\to Y$ of stacks over $\RPC$ is called a family of tropical
	tori if $X$ is a group object over $Y$ such that for every morphism $(\sigma,N_{\sigma})\to Y$ from an RPC
	to $Y$, the base-change $X\times_Y \sigma$ is isomorphic to $\sigma\times T_N^{\trop}$
	for some lattice $N$.
\end{definition}

The definition of a tropical abelian variety is somewhat more involved and requires
further notation:

\begin{definition}
	Let $(\sigma,N_{\sigma})\in \RPC$ and $a,b\in N_{\sigma}$. We say $a\le b$ if $b-a\in \sigma$. Similarly,
	if $a,b\in M_{\sigma}$, we say $a\le b$ if $b-a\in \sigma^{\vee}$.
\end{definition}

\begin{definition}
	Let $(\sigma_0,N_0)\in \RPC$ and let $N$ be a lattice. Fix a symmetric bilinear map $Q:M\times M\to M_0$.

	\begin{enumerate}
		\item $Q$ is called positive-definite if for every $m\in M\setminus \{0\}$ we have $Q(m,m)\in \sigma_0^{\vee}\setminus \{0\}$.
		\item If $\tau\to \sigma_0$ is a morphism of RPCs, we denote the composition of $Q$ with $M_0\to M_{\tau}$ by $Q_{\tau}$.
		\item We write $Q_{\hom}:M\to \Hom(M,M_0)$ for the morphism defined by $m\mapsto Q(m,\ast)$.
		\item Observe that $\Hom(M,M_0)\cong \Hom(N_0,N)$ by dualizing. We write $Q_{\hom,N}:M\to \Hom(N_0,N)$ for the morphism obtained
			by composing $Q_{\hom}$ with this isomorphism.
		\item Write $\underline{M}$ for the constant sheaf of groups on $\RPC$ with fiber $M$. We denote by
			$\Hom(\underline{M},\mathbb{G}_{m,\trop})^{(M)}$ the subsheaf of $\Hom(\underline{M},\mathbb{G}_{m,\trop})$ on $\RPC/\sigma_0$
			which to $\tau\to \sigma_0$ assigns the set of $\phi\in \Hom(M,M_{\tau})$ such that for every $m\in M$ there exist
			$m_1,m_2\in M$ with:
			$$Q_{\tau}(m_1,m)\le \phi(m)\le Q_{\tau}(m_2,m).$$
	\end{enumerate}
\end{definition}

The following property of positive-definite symmetric bilinear forms will be key to many properties of tropical abelian varieties:

\begin{lemma}
	\label{lem:pos_definite_vanishing}
	Fix a symmetric positive-definite bilinear form $Q$ over $\sigma\in \RPC$.
	Let $m\in M$ and $n\in N_{\sigma}\cap \sigma$. If $n(Q(m,m))=0$ then $n(Q(m,m'))=0$ for all $m'\in M$.
\end{lemma}

\begin{proof}
	Assume that there exists $m'\in M$ such that $n(Q(m,m'))\neq 0$. Replacing $m'$ with $-m'$ if necessary, we can assume
	that $n(Q(m,m'))<0$. By the positive-definiteness of $Q$, we have for each $\lambda>0$ that $n(Q(\lambda m+m',\lambda m+m'))\ge 0$.
	Expanding this expression, we get:
	$$0\le 2\lambda n(Q(m,m'))+n(Q(m',m')).$$
	Letting $\lambda\to \infty $ yields a contradiction.
\end{proof}

\begin{definition}
	Let $(\sigma_0,N_{\sigma_0})\in \RPC$ and $Q:M\times M\to M_{\sigma_0}$ a symmetric positive-definite bilinear form. We
	call the sheaf
	$$\Hom(\underline{M},\mathbb{G}_{m,\trop})^{(M)}/Q_{\hom}(\underline{M})$$
	on $\RPC/\sigma_0$
	a family of principally polarized tropical abelian varieties. The group structure on $\mathbb{G}_{m,\trop}$ induces the structure of a group object
	over $\sigma_0$ on $\Hom(\underline{M},\mathbb{G}_{m,\trop})^{(M)}/Q_{\hom}(\underline{M})$.

	A morphism $X\to Y$ of stacks over $\RPC$ is called a family of
	principally polarized tropical abelian varieties if $X$ is a group object over $Y$ and
	for every $\tau\in \RPC$ we have that $X\times_Y \tau\to \tau$ is a family of tropical
	abelian varieties in the previous sense.
\end{definition}

\begin{remark}
	All families of tropical abelian varieties in this article will
	be principally polarized. We will therefore drop the prefix
	``principally polarized'' in what follows. The results carry over similarly to more general
	non-principally polarized tropical abelian varieties. In this case, the lattice $M$ and its dual $N$
	can no longer be canonically identified and more care is required in the definition of a tropical abelian
	variety (see \cite{LogarithmicAbeKajiwa2008} for details).
	We restrict to the principally polarized case to avoid these additional complications.
\end{remark}

The above definition of a tropical abelian variety matches the definition in \cite{LogarithmicAbeKajiwa2008}.
Below we give an alternative description of the $\tau$-points of a tropical abelian variety over an RPC which
is more similar to the way the functor of points of a cone complex is defined. This will be used later when
we study the combinatorial structure of partial log compactifications.

\begin{definition}
	Let $\sigma_0\in \RPC$ and $Q:M\times M\to M_{\sigma_0}$ be a symmetric positive-definite bilinear form.
	We define $(N_{\sigma_0}\times N)^{(M)}\subseteq N_{\sigma_0}\times N$ to be the subset of all tuples $(n,n')$
	such that $n\in \sigma_0$ and $n'$ lies in the intersection of $N$ and the $\mathbb{Q}$-span of $Q_{\hom,N}(M)(n)\subseteq N$.
\end{definition}

\begin{lemma}
	In the above situation, $m\in M$ acts on $(N_{\sigma_0}\times N)^{(M)}$ by
	$$m\cdot (n,n'):=(n,n'+Q_{\hom,N}(m)(n)).$$
	We denote the action of $m\in M$ on $(N_{\sigma_0}\times N)^{(M)}$ by $T_m$.
\end{lemma}

\begin{proof}
	It suffices to show that if $(n,n')\in (N_{\sigma_0}\times N)^{(M)}$ then also $(n,n'+Q_{\hom,N}(m)(n))\in (N_{\sigma_0}\times N)^{(M)}$.
	This is clear since $Q_{\hom,N}(m)(n)$ obviously lies in the span of $Q_{\hom,N}(M)(n)$.
\end{proof}

\begin{proposition}
	\label{prop:alternative_descr_trop_torus}
	Let $\tau\to \sigma_0$ be an RPC over $\sigma_0$ in the situation above. Write $\Hom(\tau,(N_{\sigma_0}\times N)^{(M)})$
	for the set of homomorphisms $N_{\tau}\to N_{\sigma_0}\times N$ over $N_{\tau}\to N_{\sigma_0}$ such that the image
	of $N_{\tau}\cap \tau$ lies in $(N_{\sigma_0}\times N)^{(M)}$. By the preceding lemma, there is an action of $M$
	on $\Hom(\tau,(N_{\sigma_0}\times N)^{(M)})$ by translation. With this action, we have:
	$$\Hom(\tau,\Hom(\underline{M},\mathbb{G}_{m,\trop})^{(M)}/Q_{\hom}(\underline{M}))\cong \Hom(\tau,(N_{\sigma_0}\times N)^{(M)})/M.$$
\end{proposition}

\begin{proof}
	We first show that a morphism $\tau\to \Hom(\underline{M},\mathbb{G}_{m,\trop})/Q_{\hom}(\underline{M})$ is equivalent to the
	datum of a morphism $N_{\tau}\to (N_{\sigma_0}\times N)$ modulo the action of $M$. By definition, a morphism
	$\tau\to \Hom(\underline{M},\mathbb{G}_{m,\trop})/Q_{\hom}(\underline{M})$ is the datum of an equivalence class
	$[\phi]\in \Hom(M,M_{\tau})/Q_{\tau,\hom}(\underline{M})$. We obtain a morphism $\phi^{\vee}:N_{\tau}\to N$ and taking
	the product with $f:N_{\tau}\to N_{\sigma_0}$ gives a morphism $N_{\tau}\to N_{\sigma_0}\times N$. Clearly, this construction
	is a bijection. Note that $\phi'\in [\phi]$ if and only if
	$\phi'-\phi=Q_{\tau,\hom}(m)$ for some $m\in M$. This is equivalent to ${\phi'}^{\vee}-\phi^{\vee}=Q_{\tau,\hom,N}(m)$,
	so the quotients coincide.

	Finally, it remains to show that $\phi\in \Hom(M,M_{\tau})^{(M)}$ if and only if $f\times\phi^{\vee}$ sends $\tau\cap N_{\tau}$
	to $(N_{\sigma_0}\times N)^{(M)}$. First assume that $\phi\in \Hom(M,M_{\tau})^{(M)}$. Assume for a contradiction
	that there exists $n\in \tau\cap N_{\tau}$
	such that $\phi^{\vee}(n)$ does not lie in the $\mathbb{Q}$-span of $Q_{\tau,\hom}(M)(n)$. Then there exists $m\in M$ such that
	$m(\phi^{\vee}(n))>0$ and $m|_{Q_{\tau,\hom,N}(M)(n)}=0$, viewing $m$ as a map $N\to \mathbb{Z}$.
	Let $m_1\in M$ be arbitrary. Then $n(Q_{\tau}(m_1,m))=0<n(\phi(m))$. It follows
	that $Q_{\tau}(m_1,m)-\phi(m)\notin \tau^{\vee}$ which contradicts the definition of $\Hom(M,M_{\tau})^{(M)}$.

	The converse is more involved: Assume that for each $n\in N_{\tau}\cap \tau$, we have that
	$\phi^{\vee}(n)$ lies in the $\mathbb{Q}$-span of $Q_{\tau,\hom}(M)(n)$
	and let $m\in M$ be arbitrary. We must show that there exist $m_1,m_2\in M$ such that
	$$Q_{\tau}(m_1,m)\le \phi(m)\le Q_{\tau}(m_2,m).$$
	Note that $\phi(m)\ge Q_{\tau}(m_1,m)\Leftrightarrow Q_{\tau}(-m_1,m)\ge \phi(-m)$. So it suffices to construct $m_2$
	such that $\phi(m)\le Q_{\tau}(m_2,m)$. This inequality is equivalent to $m(\phi^{\vee}(n))\le m(Q_{\tau,\hom,N}(m_2)(n))$
	for all $n\in \tau$. Note that the locus of $n$ such that this inequality holds is convex, so it suffices to find $m_2$
	such that the inequality holds on each ray $\rho$ of $\tau$. Since $\phi^{\vee}(n)$ lies in the $\mathbb{Q}$-span of $Q_{\tau,\hom,N}(M)(n)$,
	we can find for each $\rho$ some $m_{2,\rho}\in M$ such that $m(\phi^{\vee}(n_{\rho}))\le m(Q_{\tau,\hom,N}(m_{2,\rho})(n_{\rho}))$
	for any $n_{\rho}\in \rho$. Setting $m_2:=\sum_{\rho}^{}m_{2,\rho}$, it suffices to show that we may choose $m_{2,\rho}$ such that $m(Q_{\tau,\hom,N}(m_{2,\rho})(n))\ge 0$
	for any $n\in \tau$. In other words, we must show that we may choose $m_{2,\rho}\in Q_{\tau}(m,\ast)^{-1}(\tau^{\vee})$. For
	this, note that since $Q_{\tau}(m,m)\ge 0$ by assumption, we may replace $m_{2,\rho}$ by $m_{2,\rho}+\lambda m$ for any $\lambda\in \mathbb{N}$.
	So we reduce to showing that $Q_{\tau}(m,m_{2,\rho}+\lambda m)\in \tau^{\vee}$ for $\lambda$ sufficiently large. Let $n\in \tau$
	be a primitive element of a ray of $\tau$. Then:
	$$n(Q_{\tau}(m,m_{2,\rho}+\lambda m))=n(Q_{\tau}(m,m_{2,\rho}))+\lambda n(Q_{\tau}(m,m)).$$
	If $n(Q_{\tau}(m,m))\neq 0$ then it must be positive by positive semi-definiteness of $Q_{\tau}$, so choosing $\lambda$
	large enough, we may assume that $n(Q_{\tau}(m,m_{2,\rho}+\lambda m))\ge 0$. Otherwise, by Lemma \ref{lem:pos_definite_vanishing},
	we know that $n(Q_{\tau}(m,m_{2,\rho}+\lambda m))=0$. Since there are only finitely many rays, we may choose $\lambda$ large
	enough to work for all $n$. Since any element of $\tau$ is a positive $\mathbb{R}$-linear combination of primitive elements of rays in $\tau$
	it then follows that $n(Q_{\tau}(m,m_{2,\rho}+\lambda m))\ge 0$ for all $n\in \tau$, so $Q_{\tau}(m,m_{2,\rho}+\lambda m)\in \tau^{\vee}$
	as required.
\end{proof}

\begin{definition}
	\label{def:trop_semiab}
	A split family of tropical semiabelian varieties is a product of a family of tropical abelian varieties and a family
	of tropical tori.
\end{definition}

\begin{remark}
	In general, one could consider more general extensions of tropical abelian varieties by tropical tori.
	In this case one would need to add additional restrictions similar to the bounded monodromy condition
	in \cite{TheLogarithmicMolcho2022} to ensure that the extensions of log abelian varieties by log tori
	tropicalize. We will not need this, since any extension of a log abelian variety by a log torus
	which comes from an extension by an algebraic torus tropicalizes to the trivial extension.
\end{remark}

Although we will not make use of this perspective, we remark that split families of tropical semiabelian varieties
are piecewise linear spaces (or piecewise linear complexes) in the sense of \cite{kennedyhunt2025logarithmicquotspacefoundations}.

\begin{definition}
	\label{def:partial_compact}
	Let $X\to Z$ be a split family of tropical semiabelian varieties. A partial tropical compactification of $X\to Z$ is a
	monomorphism $Y\to X$ representable by finite cone stacks such that there exists a section
	$Z\to Y$ over the zero-section $Z\to X$. Morphisms of partial tropical compactifications are morphisms over $X$.
	We call $Y\to X$ a tropical compactification of $X\to Z$ if $Y\to X$ is surjective, i.e., if the base-change
	over any $\sigma\in \RPC$ is surjective.
\end{definition}

We explain what it means to be a finite cone stack:

\begin{definition}
	A cone stack $C$ is called finite if there are only finitely many objects in $C$ up to isomorphism and
	if there are only finitely many morphisms in $C$ up to 2-isomorphism.
\end{definition}

\begin{lemma}
	\label{lem:morph_part_comp_repr}
	Let $X\to Z$ be a split family of tropical semiabelian varieties.
	Any morphism $Y_1\to Y_2$ of partial tropical compactifications of $X\to Z$ is a monomorphism of stacks
	representable by finite cone stacks.
\end{lemma}

\begin{proof}
	Since the composition $Y_1\to Y_2\to X$ is a monomorphism by assumption, it follows that $Y_1\to Y_2$
	is also a monomorphism. Let $\sigma\to Y_2$ be a morphism from an RPC $\sigma$. Then
	$Y_1\times_{Y_2}\sigma\cong Y_1\times_{X}\sigma$ since $Y_2\to X$ is a monomorphism.
	It follows that $Y_1\times_{Y_2}\sigma$ is a finite cone stack, since $Y_1\to X$ is representable
	by finite cone stacks.
\end{proof}

\subsection{The tropical realization functor}

To obtain a tropical correspondence theorem we must translate split families
of tropical semiabelian varieties into stacks over the category of
fs log schemes. As explained in
\cite{AModuliStackCavali2017}, there is a standard
procedure for doing this. For convenience, we recall the key points of this procedure
below.

\begin{definition}
	Let $\mathcal{C}$ be a stack over $\RPC$ and let $M\in \ShpMon^{\op}$ be an
	arbitrary sharp saturated and integral monoid. We define:
	$$\Hom(M,\overline{\mathcal{C}}):=\colim_{Q\to M}\Hom_{\RPC}(Q^{\vee},\mathcal{C})$$
	where the 2-colimit is over the category of morphisms
	$Q\to M$ in $\ShpMon$ for an arbitrary toric monoid $Q$.
\end{definition}

\begin{remark}
	In the above definition we are using the equivalence of categories
	between $\RPC$ and the opposite category of the category of
	toric monoids. It is clear that the above definition defines a stack for the chaotic topology (i.e., the Grothendieck topology generated
	by the class of isomorphisms) on $\ShpMon^{\op}$.
\end{remark}

\begin{definition}
	Let $\mathcal{C}$ be a stack over $\RPC$. We define a
	prestack $\mathcal{A}_{\mathcal{C}}'$ on $\LogSch/\Spec(K)$ by:
	$$\Hom(X,\mathcal{A}_{\mathcal{C}}')=\Hom(\overline{M}_X(X),\overline{\mathcal{C}}),$$
	where $K$ denotes the ground field (which we have assumed to
	be an algebraically closed field of characteristic zero).
	Let $\mathcal{A}_{\mathcal{C}}$ be the strict étale stackification of
	$\mathcal{A}_{\mathcal{C}}'$.
	We call $\mathcal{A}_{\mathcal{C}}$ the geometric realization
	of $\mathcal{C}$. The functor sending $\mathcal{C}$
	to $\mathcal{A}_{\mathcal{C}}$ is called the geometric realization
	functor.
\end{definition}

\begin{example}
	If $\mathcal{C}=\sigma$ is an RPC, then $\mathcal{A}_{\sigma}=[X_{\sigma}/T]$, where $X_{\sigma}$
	is the affine toric variety with cone $\sigma$ and dense open torus $T$. The canonical quotient
	map $X_{\sigma}\to [X_{\sigma}/T]$ is an Artin fan for $X_{\sigma}$, i.e., $\mathcal{A}_{\sigma}=\mathcal{A}_{X_{\sigma}}$.
	We call $\mathcal{A}_{\sigma}$ an \emph{Artin cone}.
	More generally, if $\Sigma$ is a fan in some lattice, then $\mathcal{A}_{\Sigma}=[X_{\Sigma}/T]$ for the toric
	variety $X_{\Sigma}$ with fan $\Sigma$ and dense open torus $T$.
\end{example}

\begin{lemma}
	\label{lem:rpc_repr}
	Let $\sigma\in \RPC$. Then $\mathcal{A}_{\sigma}=\mathcal{A}_{\sigma}'$
	represents the functor $X\mapsto \Hom_{\ShpMon}(\sigma^{\vee}\cap M_{\sigma},\overline{M}_X(X)).$
\end{lemma}

\begin{proof}
	First observe that since $X\mapsto \overline{M}_X(X)$ is a sheaf, the
	functor $X\mapsto \Hom_{\ShpMon}(\sigma^{\vee}\cap M_{\sigma},\overline{M}_X(X))$
	is indeed a stack.
	Hence it suffices to show that $\mathcal{A}_{\sigma}'$ represents
	this functor. This is a direct computation:
\begin{align*}
	\Hom(X,\mathcal{A}_{\sigma}')&=\colim_{Q\to \overline{M}_X(X)}\Hom_{\ShpMon}(Q^{\vee},\sigma\cap N_{\sigma})\\
	&=\colim_{Q\to \overline{M}_X(X)}\Hom_{\ShpMon}(\sigma^{\vee}\cap M_{\sigma},Q)\\
	&=\Hom_{\ShpMon}(\sigma^{\vee}\cap M_{\sigma},\overline{M}_X(X)).
\end{align*}
\end{proof}

\begin{lemma}
	\label{lem:smooth_descent}
	With notation as above, the functor $\mathcal{A}_{\mathcal{C}}$
	is a stack for the strict smooth topology.
\end{lemma}

\begin{proof}
	Let $Y$ be an fs log scheme and $f:X\to Y$ a strict smooth cover. Let $\xi\in \Hom(X,\mathcal{A}_{\mathcal{C}})$ be equipped
	with a descent datum
	$$\alpha:\pi_1^*\xi\to \pi_2^*\xi$$
	on $X\times_Y X$. Since $\mathcal{A}_{\mathcal{C}}$ is a stack in the strict étale topology, effectiveness
	of this descent datum is strict étale local on $Y$. Since the underlying morphism of schemes of $f$
	is smooth, étale locally on $Y$ it admits a section. Hence after replacing $Y$ by a strict étale cover,
	we may assume that $f$ admits a section $s:Y\to X$. Since $f$ is strict, the section $s$ is strict as well.

	Set $\eta:=s^*\xi\in \Hom(Y,\mathcal{A}_{\mathcal{C}})$. We claim that $\eta$ descends $\xi$.
	Let
	$$\Delta=(s\circ f,id_X):X\to X\times_Y X.$$
	Pulling back $\alpha$ along $\Delta$, we obtain an isomorphism
	$$\beta:=\Delta^*(\alpha):f^*\eta=(s\circ f)^*\xi\to \xi.$$
	It remains to check that $\beta$ is compatible with the descent datum on $\xi$.
	Let
	$$\delta=(s\circ f\circ \pi_1,\pi_1,\pi_2):X\times_Y X\to X\times_Y X\times_Y X.$$
	Pulling back the cocycle condition
	$$\pi_{23}^*(\alpha)\circ \pi_{12}^*(\alpha)=\pi_{13}^*(\alpha)$$
	along $\delta$, we get
	$$\alpha\circ \pi_1^*(\beta)=\pi_2^*(\beta),$$
	which is exactly the compatibility condition. Thus $(\eta,\beta)$ is a descent of $(\xi,\alpha)$,
	showing that objects descend.

	It remains to show that morphisms descend. Let $\eta_1,\eta_2\in \Hom(Y,\mathcal{A}_{\mathcal{C}})$ and let
	$$\psi:f^*\eta_1\to f^*\eta_2$$
	be a morphism such that
	$$\pi_2^*(\psi)=\pi_1^*(\psi)$$
	as morphisms $\pi_1^*f^*\eta_1=\pi_2^*f^*\eta_1\to \pi_1^*f^*\eta_2=\pi_2^*f^*\eta_2$ on $X\times_Y X$.
	Set
	$$\overline{\psi}:=s^*(\psi):\eta_1\to \eta_2.$$
	We claim that $f^*(\overline{\psi})=\psi$. Let
	$$\Delta=(s\circ f,id_X):X\to X\times_Y X.$$
	Pulling back the equality $\pi_2^*(\psi)=\pi_1^*(\psi)$ along $\Delta$, we get
	$$\psi=f^*s^*(\psi)=f^*(\overline{\psi}),$$
	as required. This proves existence of descended morphisms.

	For uniqueness, let $\theta,\theta':\eta_1\to \eta_2$ be morphisms such that
	$f^*(\theta)=f^*(\theta')$. Pulling back along the section $s$, we obtain
	$$\theta=s^*f^*(\theta)=s^*f^*(\theta')=\theta'.$$
	Hence morphisms descend uniquely. Therefore $\mathcal{A}_{\mathcal{C}}$
	is a stack for the strict smooth topology.
\end{proof}

For later use, we record the following useful lemma.

\begin{lemma}
	\label{lem:loc_factorization}
	Let $Z$ be an fs log scheme and $\mathcal{C}$ a stack over $\RPC$. Let $Z\to \mathcal{A}_{\mathcal{C}}$
	be a morphism. Then strict étale locally on $Z$
	there exists a strict morphism $Z\to \mathcal{A}_{\sigma}$ and a morphism $\sigma\to \mathcal{C}$
	such that $Z\to \mathcal{A}_{\mathcal{C}}$ factors as $Z\to \mathcal{A}_{\sigma}\to \mathcal{A}_{\mathcal{C}}$.
\end{lemma}

\begin{proof}
	Since the statement is strict étale local, we may assume that $Z\to \mathcal{A}_{\mathcal{C}}$ is represented by
	a morphism $Z\to \mathcal{A}_{\mathcal{C}}'$ of prestacks.
	By \cite[Proposition 2.2.2.5]{TheLogarithmicMolcho2022}, we may
	assume that $Z$ is atomic. Let $z\in Z$ lie in the closed stratum. The stalk $\overline{M}_{Z,z}$ is a sharp fs monoid, hence toric
	and $\overline{M}_Z(Z)\to \overline{M}_{Z,z}$ is an isomorphism. Let $\sigma\in \RPC$ be the cone such that
	$\sigma^{\vee}\cap M_{\sigma}\cong \overline{M}_{Z,z}$. The morphism $Z\to \mathcal{A}_{\mathcal{C}}'$
	corresponds to a morphism $\overline{M}_Z(Z)\to \mathcal{C}$, hence a morphism $\sigma\to \mathcal{C}$. The isomorphism
	$\overline{M}_Z(Z)\cong \sigma^{\vee}\cap M_{\sigma}$ induces a morphism $Z\to \mathcal{A}_{\sigma}$ and the composition
	$Z\to \mathcal{A}_{\sigma}\to \mathcal{A}_{\mathcal{C}}$ is equal to the given morphism by construction.
\end{proof}

\begin{definition}
	We define $\mathcal{A}^{1}:=(\mathbb{R}_{\ge 0},\mathbb{Z})$. Note
	that $\mathcal{A}^{1}$ is a monoid object in $\RPC$ via the map
	$(\mathbb{R}_{\ge 0},\mathbb{Z})\times(\mathbb{R}_{\ge 0},\mathbb{Z})\to
	(\mathbb{R}_{\ge 0},\mathbb{Z})$ sending $(x,y)$ to $x+y$.
\end{definition}

\begin{lemma}
	Let $X$ be a stack on $\LogSch/\Spec(K)$. Then
	$\Hom_{\mathrm{St}}(X,\mathcal{A}_{\mathcal{A}^{1}})$ is a sharp integral saturated monoid,
	where $\Hom_{\mathrm{St}}$ denotes morphisms of stacks.
\end{lemma}

\begin{proof}
	Note that since $\mathcal{A}_{\mathcal{A}^{1}}$ is fibered in
	setoids by Lemma \ref{lem:rpc_repr}, the Hom-object is indeed
	a set. We equip it with the monoid structure induced by
	the monoid structure on $\mathcal{A}^{1}$. Note that the submonoid
	of units of $(\mathbb{R}_{\ge 0},\mathbb{Z})$ is represented by
	the zero-cone $0$. Since $\mathcal{A}_{0}$ is the terminal object
	in the category of stacks, the monoid is sharp. To see that
	it is saturated, note that for any $Z\in \LogSch/\Spec(K)$,
	we have $\Hom(Z,\mathcal{A}_{\mathcal{A}^{1}})=\overline{M}_Z(Z)$
	by Lemma \ref{lem:rpc_repr}. Hence there is a natural monomorphism
	$j:\mathcal{A}_{\mathcal{A}^{1}}\to \mathbb{G}_{m,\trop}$ sending 
	a section of $\overline{M}_Z(Z)$ to its image in
	$\overline{M}^{\gp}_Z(Z)$. Let
	$\phi:X\to \mathbb{G}_{m,\trop}$ be a morphism and assume
	there exists $\psi:X\to \mathcal{A}_{\mathcal{A}^{1}}$ such
	that the composition $j\circ \psi$ is equal to $n\cdot \phi$.
	Let $f:Z\to X$ be an arbitrary morphism from a log scheme $Z$.
	Then $\phi\circ f$ corresponds to a section
	$s$ of $\overline{M}_Z^{\gp}(Z)$ and the statement that $j\circ \psi=n\cdot \phi$
	is equivalent to the statement that $n\cdot s\in \overline{M}_Z(Z)$.
	Since $\overline{M}_Z(Z)$ is saturated, it follows that
	$s\in \overline{M}_Z(Z)$ and hence we obtain a lift
	$Z\to \mathcal{A}_{\mathcal{A}^{1}}$. This construction is natural
	since the lift is unique, as $j$ is a monomorphism. We hence
	obtain a lift $X\to \mathcal{A}_{\mathcal{A}^{1}}$ of $\phi$.

	Finally, integrality follows from the fact that
	$\Hom_{\mathrm{St}}(X,\mathcal{A}_{\mathcal{A}^{1}})\to \Hom_{\mathrm{St}}(X,\mathbb{G}_{m,\trop})$
	is an injection into a group.
\end{proof}

We extend $\mathcal{A}_{\mathcal{C}}$ to a stack on the category of stacks over $\Spec(K)$ as follows:

\begin{definition}
	Let $\mathcal{C}$ be as above.
	We define
	a prestack $\mathcal{A}_{\mathcal{C}}''$ on the category of
	stacks on $\LogSch/\Spec(K)$ by
	$$\Hom(F,\mathcal{A}_{\mathcal{C}}''):=\Hom(\Hom_{\mathrm{St}}(F,\mathcal{A}_{\mathcal{A}^{1}}),\overline{\mathcal{C}}).$$
	Let $\mathcal{A}_{\mathcal{C}}$ be stackification of $\mathcal{A}_{\mathcal{C}}''$ with respect to the topology given by
	strict étale morphisms representable by log schemes.
\end{definition}

\begin{lemma}
	The new definition of $\mathcal{A}_{\mathcal{C}}$ restricts to the previous definition on the subcategory of representable stacks.
\end{lemma}

\begin{proof}
	Let $Z\in \LogSch/\Spec(K)$. Then $\Hom_{\mathrm{St}}(Z,\mathcal{A}_{\mathcal{A}^{1}})\cong \overline{M}_Z(Z)$
	by Lemma \ref{lem:rpc_repr}. Hence $\mathcal{A}_{\mathcal{C}}''$ restricts to $\mathcal{A}_{\mathcal{C}}'$ on $\LogSch/\Spec(K)$
	and their stackifications hence agree.
\end{proof}

\begin{definition}
	Let $\mathcal{X}$ be a stack on $\LogSch/\Spec(K)$. We define
	a stack $\Sigma_{\mathcal{X}}$ over $\RPC$ by setting:
	$$\Hom(\sigma,\Sigma_{\mathcal{X}}):=\Hom(\mathcal{A}_{\sigma},\mathcal{X}).$$
\end{definition}

\begin{theorem}
	\label{thm:left_inv_geom}
	The functor $\mathcal{X}\mapsto \Sigma_{\mathcal{X}}$ is a
	left-inverse of the geometric realization functor.
\end{theorem}

\begin{proof}
	Let $\mathcal{C}$ be a stack over $\RPC$ and $\sigma\in \RPC$. By \cite[Corollary 2.2.8]{Birational_inva_Abramo_2013},
	every strict étale cover of $\mathcal{A}_{\sigma}$ has a section. It follows that
	$$\Hom(\mathcal{A}_{\sigma},\mathcal{A}_{\mathcal{C}})\cong \Hom(\mathcal{A}_{\sigma},\mathcal{A}_{\mathcal{C}}'')\cong \Hom(\Hom_{\mathrm{St}}(\mathcal{A}_{\sigma},\mathcal{A}_{\mathcal{A}^{1}}),\overline{\mathcal{C}}).$$
	It remains to show that $\Hom(\mathcal{A}_{\sigma},\mathcal{A}_{\mathcal{A}^{1}})\cong \sigma^{\vee}\cap M_{\sigma}$.
	Let $X_{\sigma}$ be the affine toric variety with cone $\sigma$ and $X_{\sigma}\to \mathcal{A}_{\sigma}=[X_{\sigma}/T_{\sigma}]$
	the quotient map. By Lemma \ref{lem:smooth_descent}, we know that $\Hom(\mathcal{A}_{\sigma},\mathcal{A}_{\mathcal{A}^{1}})$
	is canonically isomorphic to the set of $T_{\sigma}$-equivariant morphisms $X_{\sigma}\to \mathcal{A}_{\mathcal{A}^{1}}$
	via pullback. By Lemma \ref{lem:rpc_repr}, we know that $\Hom(X_{\sigma},\mathcal{A}_{\mathcal{A}^{1}})$
	is isomorphic to $\Hom(\mathbb{Z}_{\ge 0}, \overline{M}_{X_{\sigma}}(X_{\sigma}))=\overline{M}_{X_{\sigma}}(X_{\sigma})\cong \sigma^{\vee}\cap M_{\sigma}$.
	Since $\overline{M}_{X_{\sigma}}$ is $T_{\sigma}$-invariant, these morphisms are automatically $T_{\sigma}$-invariant and hence
	the claim follows.
\end{proof}

\begin{corollary}
	\label{cor:fib_prods_preserved}
	The functor $\mathcal{C}\mapsto \mathcal{A}_{\mathcal{C}}$ preserves fiber products.
\end{corollary}

\begin{proof}
	Let
	$$\begin{tikzcd}
		\mathcal{C}_1\arrow[dr]&&\mathcal{C}_2\arrow[dl]\\
					 &\mathcal{C}_3&
	\end{tikzcd}$$
	be morphisms of stacks over $\RPC$.
	Note that we have a map $\mathcal{A}_{\mathcal{C}_1\times_{\mathcal{C}_3}\mathcal{C}_2}\to \mathcal{A}_{\mathcal{C}_1}\times_{\mathcal{A}_{\mathcal{C}_3}}\mathcal{A}_{\mathcal{C}_2}$
	by the universal property of the fiber product. It suffices to show that the base-change of this
	map along any morphism $Z\to \mathcal{A}_{\mathcal{C}_1}\times_{\mathcal{A}_{\mathcal{C}_3}}\mathcal{A}_{\mathcal{C}_2}$
	is an isomorphism, where $Z$ is an fs log scheme. This is strict étale local on $Z$.
	Hence by Lemma \ref{lem:loc_factorization}, we may assume that
	$Z\to \mathcal{A}_{\mathcal{C}_1}\times_{\mathcal{A}_{\mathcal{C}_3}}\mathcal{A}_{\mathcal{C}_2}$ factors as
	$Z\to \mathcal{A}_{\sigma}\to \mathcal{A}_{\mathcal{C}_1}\times_{\mathcal{A}_{\mathcal{C}_3}}\mathcal{A}_{\mathcal{C}_2}$
	for some RPC $\sigma$.
	The datum of a morphism $\mathcal{A}_{\sigma}\to \mathcal{A}_{\mathcal{C}_1}\times_{\mathcal{A}_{\mathcal{C}_3}}\mathcal{A}_{\mathcal{C}_2}$
	is equivalent to the datum of morphisms $\mathcal{A}_{\sigma}\to \mathcal{A}_{\mathcal{C}_1}$ and
	$\mathcal{A}_{\sigma}\to \mathcal{A}_{\mathcal{C}_2}$ such that the compositions with the maps $\mathcal{A}_{\mathcal{C}_i}\to \mathcal{A}_{\mathcal{C}_3}$
	agree. By Theorem \ref{thm:left_inv_geom}, this is equivalent to the datum of morphisms $\sigma\to \mathcal{C}_i$ for $i\in \{1,2\}$
	so that the compositions with the maps $\mathcal{C}_i\to \mathcal{C}_3$ for $i\in \{1,2\}$ agree. This in turn is equivalent
	to the datum of a map $\sigma\to \mathcal{C}_1\times_{\mathcal{C}_3}\mathcal{C}_2$ and hence the base-change
	of the map $\mathcal{A}_{\mathcal{C}_1\times_{\mathcal{C}_3}\mathcal{C}_2}\to \mathcal{A}_{\mathcal{C}_1}\times_{\mathcal{A}_{\mathcal{C}_3}}\mathcal{A}_{\mathcal{C}_2}$
	over $\mathcal{A}_{\sigma}$ is an isomorphism. Pulling back to $Z$ finishes the proof.
\end{proof}

\begin{definition}
	Given a property $P$ of morphisms of stacks over $\RPC$, we say that $\mathcal{A}_{\mathcal{C}}\to \mathcal{A}_{\mathcal{D}}$ satisfies $P$ if
	$\mathcal{C}\to \mathcal{D}$ satisfies $P$. For example, if $\mathcal{C}\to \mathcal{D}$ is a split family of tropical semiabelian varieties,
	we call $\mathcal{A}_{\mathcal{C}}\to \mathcal{A}_{\mathcal{D}}$ a split family of tropical semiabelian varieties.
	This notation is justified by Theorem \ref{thm:left_inv_geom}, which says that $\mathcal{C}\to \mathcal{D}$ is uniquely determined
	by $\mathcal{A}_{\mathcal{C}}\to \mathcal{A}_{\mathcal{D}}$.
\end{definition}

Using the tropical realization functor, we can translate partial tropical compactifications of split families of tropical semiabelian varieties
into the geometric setting:

\begin{definition}
	\label{def:combinat_prt_cpt}
	Let $X\to S$ be a log semiabelian variety. Assume that there exists an Artin fan $S\to \mathcal{A}_S$ for $S$ and
	a split family of tropical semiabelian varieties $\mathcal{X}\to \Sigma_{\mathcal{A}_S}$ such that
	$[X/G_X]\cong S\times_{\mathcal{A}_S}\mathcal{A}_{\mathcal{X}}$ (see Theorem \ref{thm:trop_corresp}). Let $\mathcal{Y}\to \mathcal{X}$
	be a partial tropical compactification of $\mathcal{X}\to \Sigma_{\mathcal{A}_S}$. We call the morphism
	$\mathcal{A}_{\mathcal{Y}}\times_{\mathcal{A}_{S}}S\to [X/G_X]$ a combinatorial partial log compactification of $[X/G_X]$.
	Similarly, we call the morphism $\mathcal{A}_{\mathcal{Y}}\times_{\mathcal{A}_{\mathcal{X}}}X\to X$
	a combinatorial partial log compactification. A combinatorial partial log compactification (of $[X/G_X]$) is said to be a combinatorial log compactification, if it is proper over $X$ (resp. $[X/G_X]$).
\end{definition}

\begin{proposition}
	\label{prop:combinat_cpt_is_cpt}
	A combinatorial partial log compactification is a partial
	log compactification. It is a log compactification
	if and only if the associated tropical partial compactification
	$\mathcal{Y}\to \mathcal{X}$ is a tropical compactification.
\end{proposition}

The proof will be given after some preliminary results.

\begin{proposition}
	\label{prop:mono_base_change}
	Let $\mathcal{Y}\to \mathcal{X}$ be a monomorphism of stacks over $\RPC$ representable by finite cone stacks.
	Then $\mathcal{A}_{\mathcal{Y}}\to \mathcal{A}_{\mathcal{X}}$ is a log étale monomorphism representable by
	log DM stacks of finite presentation. Moreover:
	\begin{enumerate}
		\item The morphism $\mathcal{A}_{\mathcal{Y}}\to \mathcal{A}_{\mathcal{X}}$ is representable by log
			algebraic spaces if and only if for every morphism $\sigma\to \mathcal{X}$ and every $\tau\in \sigma\times_{\mathcal{X}}\mathcal{Y}$
			the induced morphism $N_{\tau}\to N_{\sigma}$ surjects onto $N_{\sigma}\cap\Span(\tau)\subseteq N_{\sigma}$.
		\item If the base-change of $\mathcal{Y}\to \mathcal{X}$ over any cone is non-empty, The morphism $\mathcal{A}_{\mathcal{Y}}\to \mathcal{A}_{\mathcal{X}}$ is proper if and only if
			$\mathcal{Y}\to \mathcal{X}$ is surjective (i.e., for every $\sigma\to \mathcal{X}$, the morphism $\mathcal{Y}\times_{\mathcal{X}}\sigma\to \sigma$
			is surjective).
	\end{enumerate}
\end{proposition}

For the proof we require the following definition and lemma:

\begin{definition}[{\cite[Definition 2.2.3]{ATheoryOfStaGillam2015}}]
	Let $\sigma$ be an RPC. A stacky fan in $\sigma$
	is a finite set of RPCs $\Delta$ contained in $\sigma$ and finite-index
	sublattices $N_{\tau}\subseteq N_{\sigma}\cap \Span(\tau)$ satisfying the following properties:
	\begin{enumerate}
		\item $\tau_1\cap \tau_2$ is a common face of $\tau_1,\tau_2\in \Delta$.
		\item If $\tau\in \Delta$ and $\tau'$ is a face of $\tau$, then $\tau'\in \Delta$.
		\item If $\tau_1\to \tau_2$ is a face inclusion in $\Delta$ then $N_{\tau_2}\cap \Span \tau_1=N_{\tau_1}$.
	\end{enumerate}
	A stacky fan defines a cone complex (also denoted by $\Delta$)
	obtained by gluing the cones $(\tau,N_{\tau})$ along the face inclusions in $\Delta$.
\end{definition}

\begin{lemma}
	\label{lem:mono_impl_fan}
	Let $\mathcal{Y}\to \sigma$ be a monomorphism from a finite cone stack $\mathcal{Y}$ to an RPC $\sigma$. Then $\mathcal{Y}$
	is isomorphic to a stacky fan in $\sigma$.
\end{lemma}

\begin{proof}
	First observe that since $\sigma$ is fibered in setoids and $\phi:\mathcal{Y}\to \sigma$ is a monomorphism, the stack $\mathcal{Y}$
	is fibered in setoids, i.e., $\mathcal{Y}$ is a cone space. Let $\Delta:=\{\phi(\tau)|\tau\in \mathcal{Y}\}$. The set $\Delta$
	is clearly finite and closed under taking faces. Since
	$\mathcal{Y}\to \sigma$ is a monomorphism, for each $\tau\in \Delta$ there is a unique cone $\phi^{-1}(\tau)\in \mathcal{Y}$
	such that $\phi^{-1}(\tau)\to \tau$ is an isomorphism.
	Let $N_{\tau}$ be the image of $N_{\phi^{-1}(\tau)}$ under $\phi$. Since $\phi$ is a morphism of cone spaces, we have
	$N_{\tau}\subseteq N_{\sigma}\cap \Span(\tau)$. Moreover, $N_{\tau}$ spans $\Span(\tau)$ as an $\mathbb{R}$-vector space
	and hence $N_{\tau}$ is of finite index in $N_{\sigma}\cap \Span(\tau)$. We show that $\Delta$ satisfies the properties in
	the statement of the lemma. Let $\tau:=\tau_1\cap \tau_2$ for $\tau_1,\tau_2\in \Delta$ and let
	$$N_{\tau}':=N_{\tau_1}\cap N_{\tau_2}\cap \Span(\tau).$$
	The pair $(\tau,N_{\tau}')$ is an RPC and the inclusions $\tau\to \tau_i$ lift to morphisms
	$\phi_i:(\tau,N_{\tau}')\to \phi^{-1}(\tau_i)\to \mathcal{Y}$. Note that the compositions of the $\phi_i$ with
	the projection $\mathcal{Y}\to \sigma$ are equal. Hence since $\mathcal{Y}\to \sigma$ is a monomorphism, the two morphisms
	$\phi_i$ are equal. If the image of $\phi_1$ intersects the interior of $\phi^{-1}(\tau_1)$ this is only possible if there is a face
	morphism $\phi^{-1}(\tau_1)\to \phi^{-1}(\tau_2)$ in $\mathcal{Y}$ so $\tau=\tau_1$ and $\tau_1$ is a face of $\tau_2$. If not,
	let $\phi^{-1}(\tau_1')\subseteq \phi^{-1}(\tau_1)$ be a proper face such that the image of $\phi_1$ factors through $\phi^{-1}(\tau_1')$.
	Then $\tau=\tau_1'\cap \tau_2$ and by induction on the dimension of $\tau_1$ it follows that (1) is satisfied.

	For (2), notice that if $\tau'$ is a face of $\tau\in \Delta$, then $\tau'$ is the image of some face of $\phi^{-1}(\tau)$,
	so $\tau'\in \Delta$. It remains to show point (3): Let $\tau_1\to \tau_2$ be a face morphism in $\Delta$. The argument
	above for the case $\tau=\tau_1$ shows that this lifts to a face morphism $\phi^{-1}(\tau_1)\to \phi^{-1}(\tau_2)$. In particular,
	$N_{\phi^{-1}(\tau_1)}=N_{\phi^{-1}(\tau_2)}\cap \Span(\tau_1)$ and finally $N_{\tau_1}=N_{\tau_2}\cap \Span(\tau_1)$.

	We now show that $\mathcal{Y}\cong \Delta$. First note that $\mathcal{Y}\to \sigma$ lifts to a monomorphism $\mathcal{Y}\to \Delta$
	by construction of $\Delta$. Hence it suffices to show that every morphism $f:\tau\to \Delta$ for $\tau$ an RPC lifts uniquely to
	a morphism $\tau\to \mathcal{Y}$. Fix such a morphism. Then $f(\tau)\subseteq \tau'$ for some $\tau'\in \Delta$. Since $f(N_{\tau})\subseteq N_{\tau'}$,
	the morphism $f$ lifts to a morphism $\tau\to \phi^{-1}(\tau')$. Composing with $\phi^{-1}(\tau')\to \mathcal{Y}$ gives a lift
	of $f$ to $\mathcal{Y}$. Uniqueness of this lift is clear since $\mathcal{Y}\to \Delta$ is a monomorphism.
\end{proof}

\begin{proof}[Proof of Proposition \ref{prop:mono_base_change}]
	Step 1 (Show that $\psi:\mathcal{A}_{\mathcal{Y}}\to \mathcal{A}_{\mathcal{X}}$ is a monomorphism): This
	is equivalent to $\mathcal{A}_{\mathcal{Y}}\times_{\mathcal{A}_{\mathcal{X}}}\mathcal{A}_{\mathcal{Y}}\cong \mathcal{A}_{\mathcal{Y}}$.
	By Corollary \ref{cor:fib_prods_preserved} we have
	$$\mathcal{A}_{\mathcal{Y}}\times_{\mathcal{A}_{\mathcal{X}}}\mathcal{A}_{\mathcal{Y}}\cong\mathcal{A}_{\mathcal{Y}\times_{\mathcal{X}}\mathcal{Y}}\cong \mathcal{A}_{\mathcal{Y}}$$
	where the second equality holds since $\mathcal{Y}\to \mathcal{X}$ is a monomorphism.

	Step 2 (Reduce to the case of stacky fans): We must show that the properties in the statement hold after taking the base-change over any log scheme $Z\in \LogSch/S$.
	Note that each of these statements is strict étale local on the target and hence it suffices to show that they hold
	strict étale locally on $Z$ (see \cite[Theorem 1.1 (3)]{fortman2025descentalgebraicstacks} for the proof that representability is local).
	So we may assume that $Z$ has a global chart by a sharp toric monoid.
	That is, we obtain a strict morphism $Z\to X_{\sigma}$ for some affine toric variety $X_{\sigma}$.
	By Lemma \ref{lem:loc_factorization}, we may assume that $Z\to \mathcal{A}_{\mathcal{X}}$ factors
	as $Z\to X_{\sigma}\to \mathcal{A}_{\sigma}\to \mathcal{A}_{\mathcal{X}}$ and by
	Corollary \ref{cor:fib_prods_preserved} we have $\mathcal{A}_{\mathcal{Y}}\times_{\mathcal{A}_{\mathcal{X}}}\mathcal{A}_{\sigma}=\mathcal{A}_{\mathcal{Y}\times_{\mathcal{X}}\sigma}$.
	Therefore, we can reduce to the case where $\mathcal{X}=\sigma$ and $\mathcal{Y}$ is a finite cone stack. By the preceding lemma,
	$\mathcal{Y}$ is a stacky fan $\Delta$ in $\sigma$.

	Step 3 (Cover $\mathcal{A}_{\Delta}$ by Artin cones): We start by showing that
	$$\bigsqcup_{\tau\in \Delta}\mathcal{A}_{\tau}\to \mathcal{A}_{\Delta}$$
	is a strict (Zariski) open cover of $\mathcal{A}_{\Delta}$.
	This can be checked after base-change to any log scheme $W$. By Lemma \ref{lem:loc_factorization} and \cite[Tag 02L3]{stacks-project}
	we may assume that there is
	a factorization $W\to \mathcal{A}_{\sigma'}\to \mathcal{A}_{\Delta}$ for some $\sigma'\in \RPC$. Let the image of $\sigma'$
	under $\phi:\sigma'\to N_{\sigma,\mathbb{R}}$ be contained in $\tau\in \Delta$. Then $\tau\times_{\Delta}\sigma'\cong \sigma'$, so $\bigsqcup_{\tau\to \Delta}\mathcal{A}_{\tau\times_{\Delta}\sigma'}\to \mathcal{A}_{\sigma'}$
	is surjective. For a general $\tau\in \Delta$, we have $\tau\times_{\Delta}\sigma'=\phi^{-1}(\tau)$, equipped with the lattice
	obtained by restricting the lattice on $\sigma'$. Hence it suffices to show that if $\sigma''\to \sigma'$ is a face morphism then $\mathcal{A}_{\sigma''}\to \mathcal{A}_{\sigma'}$
	is represented by open immersions. Note that being an open immersion is fpqc-local on the base (see \cite[Tag 02L3]{stacks-project}).
	Since $\mathcal{A}_{\sigma'}=[X_{\sigma'}/T_{\sigma'}]$ with $T_{\sigma'}\subseteq X_{\sigma'}$ the dense torus, the morphism
	$X_{\sigma'}\to \mathcal{A}_{\sigma'}$ is an fpqc cover and it suffices to show that
	$\mathcal{A}_{\sigma''}\times_{\mathcal{A}_{\sigma'}}X_{\sigma'}\to X_{\sigma'}$ is an open immersion. Note that
	$\mathcal{A}_{\sigma''}\times_{\mathcal{A}_{\sigma'}}X_{\sigma'}$ represents the functor which sends a log scheme $Z$
	to the set of homomorphisms in
	$$\Hom(M_{\sigma'},M_Z^{\gp}(Z))$$
	such that strict étale locally on $Z$, the induced
	morphism $\overline{M}_Z(Z)^{\vee}\to N_{\sigma'}$ factors through $\sigma''$. This is the same as the functor represented
	by the open subset of $X_{\sigma'}$ corresponding to $\sigma''$, showing the claim.

	Step 4 (Verify log étaleness and representability conditions): Since representability, log étaleness and (1) are Zariski local on the source it suffices to assume that $\Delta=\tau$. I.e.,
	we must show that $\mathcal{A}_{\tau}\to \mathcal{A}_{\sigma}$ is log étale and representable by log DM stacks of finite presentation for any
	monomorphism $\tau\to \sigma$ of RPCs and that $\mathcal{A}_{\tau}\to \mathcal{A}_{\sigma}$ is representable by
	log algebraic spaces if and only if $N_{\tau}=N_{\sigma}\cap \Span(\tau)$. Recall that $Z\to \mathcal{A}_{\sigma}$ factors through $X_{\sigma}$
	and consider the base-change $\mathcal{A}_{\tau}\times_{\mathcal{A}_{\sigma}}X_{\sigma}\to X_{\sigma}$.

	Step 4.1 (Prove ``if'' direction of point (1)): We factor the map $\tau\to \sigma$ as $\tau\to \tau'\to \sigma$ where $\tau'$
	is the cone $\tau$ equipped with the lattice $\Span(\tau)\cap N_{\sigma}$. Note that $\mathcal{A}_{\tau'}\times_{\mathcal{A}_{\sigma}}X_{\sigma}$
	is the toric variety $Y_{\tau'}$ associated to the fan $\tau'$ in $N_{\sigma}$. Since any toric morphism of toric varieties
	which is an isomorphism on the dense open torus is log étale it follows that $Y_{\tau'}\to X_{\sigma}$ is log étale and representable
	by log schemes of finite presentation. Taking the base-change over $Z$, this shows the ``if'' part of point (1),
	since in this case $\tau=\tau'$.

	Step 4.2 (Compute $Y_{\tau}$ and $Y_{\tau'}$): Pick a sublattice $N\subseteq N_{\sigma}$ such that $N_{\sigma}\cong N_{\tau'}\oplus N$.
	Observe that $Y_{\tau'}$ represents the functor which sends a log scheme $W$ to the set of morphisms
	$f\in \Hom(M_{\sigma},M_W^{\gp}(W))$ such that strict étale locally on $W$ the induced
	morphism $\overline{M}_W(W)^{\vee}\to N_{\sigma}$
	factors through $\tau'$. Under the above splitting this becomes the set of pairs of morphisms
	$$(f,g)\in \Hom(M,M_W^{\gp}(W))\times\Hom(M_{\tau'},M_W^{\gp}(W))$$
	such that strict étale locally the induced morphisms $\overline{M}_W(W)^{\vee}\to N$ and $\overline{M}_W(W)^{\vee}\to N_{\tau'}$
	factor through $\{0\}$ and $\tau'$ respectively. Since the kernel of $M_W\to \overline{M}_{W}$ is isomorphic to $\mathcal{O}_W^{\times}$,
	$f$ factors uniquely through $\mathcal{O}_W^{\times}(W)$. So $Y_{\tau'}\cong \Hom(\underline{M},\mathbb{G}_{m})\times X_{\tau'}$ where $X_{\tau'}$
	is the toric variety associated to $\tau'$. Similarly,
	$$Y_{\tau}\cong \Hom(\underline{M},\mathbb{G}_{m})\times (\mathcal{A}_{\tau}\times_{\mathcal{A}_{\tau'}}X_{\tau'}).$$

	Step 4.3 (Prove that $Y_{\tau}$ is a log DM stack of finite presentation): Write $Z_{\tau}:=\mathcal{A}_{\tau}\times_{\mathcal{A}_{\tau'}}X_{\tau'}$.
	The stack $Z_{\tau}$ represents the functor which to $W\in \LogSch/S$ associates the set of morphisms $\phi\in \Hom(M_{\tau'},M_{W}^{\gp}(W))$
	such that strict étale locally on $W$, the induced map $\overline{M}_W(W)^{\vee}\to N_{\tau'}$ factors through $\tau\cap N_{\tau}$.
	Equivalently, $Z_{\tau}$ represents the functor which sends $W$ to the set of morphisms $W\to X_{\tau'}$ (given by morphisms
	$\phi\in \Hom(M_{\tau'},M_{W}^{\gp}(W))$ such that $\phi$ is the restriction of a morphism in $\Hom(M_{\tau},M_W^{\gp}(W))$ to $M_{\tau'}$).
	Hence we have a morphism $X_{\tau}\to Z_{\tau}$ which sends $\phi\in \Hom(M_{\tau},M_W^{\gp}(W))$ to its restriction to $M_{\tau'}$.
	Note that $\Hom(M_{\tau}/M_{\tau'},M_W^{\gp}(W))$ acts on $\Hom(M_{\tau},M_W^{\gp}(W))$ by multiplication. Since $\overline{M}_W^{\gp}(W)$
	is torsion-free, the composition map
	$$\Hom(M_{\tau}/M_{\tau'},M_{W}^{\gp}(W))\to \Hom(M_{\tau}/M_{\tau'},\overline{M}_W^{\gp}(W))$$
	is the zero-map, i.e., we have $\Hom(M_{\tau}/M_{\tau'},M_W^{\gp}(W))=\Hom(M_{\tau}/M_{\tau'},\mathcal{O}_{W}^{\times}(W))$.
	It follows that the action on $\Hom(M_{\tau},M_W^{\gp}(W))$ restricts to an action on $\Hom(W,X_{\tau})$ and
	$$\Hom(W,Z_{\tau})=\Hom(W,X_{\tau})/\Hom(M_{\tau}/M_{\tau'},\mathcal{O}_{W}^{\times}(W))$$
	is the quotient by this action.
	We conclude that $Z_{\tau}=[X_{\tau}/\Hom(\underline{M_{\tau}/M_{\tau'}},\mathbb{G}_m)]$ which is a log DM stack of finite presentation since
	$M_{\tau}/M_{\tau'}$ is finite. It follows that $Y_{\tau}\cong \Hom(\underline{M},\mathbb{G}_m)\times [X_{\tau}/\Hom(\underline{M_{\tau}/M_{\tau'}},\mathbb{G}_m)]$
	is representable by log DM stacks of finite presentation.

	Step 4.4 (Prove the ``only if'' direction of (1)): Since $\Hom(\underline{M},\mathbb{G}_m)\times X_{\tau}$ is a strict étale cover of $Y_{\tau}$
	such that $\Hom(\underline{M},\mathbb{G}_m)\times X_{\tau}\to Y_{\tau'}$ is log étale, it follows that $Y_{\tau}$ is log étale over $X_{\sigma}$. Note that the stabilizer
	of the closed point of $Y_{\tau}$ is $\Hom(\underline{M_{\tau}/M_{\tau'}},\mathbb{G}_m)$ which is non-trivial when $M_{\tau}\neq M_{\tau'}$. This
	shows the ``only if'' direction in (1).

	Step 5 (Show properness is equivalent to surjectivity): It remains to show (2). By the above, we reduce to the following situation: Let $\sigma$ be an RPC and $\Delta$ a non-empty stacky fan in $\sigma$.
	We must show that $\mathcal{A}_{\Delta}\times_{\mathcal{A}_{\sigma}}X_{\sigma}\to X_{\sigma}$ is proper if and only if
	$\Delta\to \sigma$ is surjective. Note that this is fppf local on $X_{\sigma}$. Let $N\subseteq N_{\sigma}$ be a finite-index sublattice
	such that for each $\tau\in \Delta$ we have $\Span(\tau)\cap N\subseteq N_{\tau}$. The map $X_{(\sigma,N)}\to X_{(\sigma,N_{\sigma})}$
	is an fppf cover. Hence it suffices to show that $\mathcal{A}_{\Delta}\times_{\mathcal{A}_{\sigma}}X_{(\sigma,N)}\to X_{(\sigma,N)}$
	is proper if and only if $\Delta\to \sigma$ is surjective. Observe that $X_{(\sigma,N)}\to \mathcal{A}_{(\sigma,N_{\sigma})}$ factors through
	$\mathcal{A}_{(\sigma,N)}$. By Corollary \ref{cor:fib_prods_preserved}, we have
	$\mathcal{A}_{(\sigma,N)}\times_{\mathcal{A}_{(\sigma,N_{\sigma})}}\mathcal{A}_{\Delta}\cong \mathcal{A}_{(\sigma,N)\times_{(\sigma,N_{\sigma})}\Delta}$.
	Note that $(\sigma,N)\times_{(\sigma,N_{\sigma})}\Delta$ is a stacky fan $\Delta'$ in $(\sigma,N)$ where all $\tau\in \Delta'$
	satisfy $N_{\tau}=N\cap \Span(\tau)$. Hence $X_{(\sigma,N)}\times_{\mathcal{A}_{(\sigma,N)}}\mathcal{A}_{\Delta}$ is the toric
	variety with fan $\Delta'$. It follows from the standard theory of toric varieties that the map to $X_{(\sigma,N)}$ is proper if and
	only if $\Delta'\to \sigma$ is surjective.
\end{proof}

Using this result, we can give a combinatorial description of when a
combinatorial partial log compactification is a log compactification of $G_X\to S$:

\begin{lemma}
	\label{lem:proper_iff_surj}
	Let notation be as in Definition \ref{def:combinat_prt_cpt}.
	A partial tropical compactification $\mathcal{Y}\to \mathcal{X}$ is a tropical compactification if
	and only if $S\times_{\mathcal{A}_{S}}\mathcal{A}_{\mathcal{Y}}\to [X/G_X]$ is proper.
\end{lemma}

\begin{proof}
	From Proposition \ref{prop:mono_base_change} (2) we know that $\mathcal{Y}\to \mathcal{X}$ is a tropical compactification if and only
	if $\mathcal{A}_{\mathcal{Y}}\to \mathcal{A}_{\mathcal{X}}$ is proper (non-emptyness of the base-change over any cone is clear by restricting to the origin, corresponding to the dense open subset of $S$ with trivial log structure). Note that since $S$ is log smooth, the morphism
	$S\to \mathcal{A}_S$ is a strict smooth cover. By \cite[Tag 02L1]{stacks-project} it follows that $\mathcal{A}_{\mathcal{Y}}\to \mathcal{A}_{\mathcal{X}}$
	is proper if and only if $S\times_{\mathcal{A}_S}\mathcal{A}_{\mathcal{Y}}\to [X/G_X]=S\times_{\mathcal{A}_S}\mathcal{A}_{\mathcal{X}}$
	is proper.
\end{proof}

\begin{proof}[Proof of Proposition \ref{prop:combinat_cpt_is_cpt}]
	Let $Y:=\mathcal{A}_{\mathcal{Y}}\times_{\mathcal{A}_{\mathcal{X}}}[X/G_X]\to [X/G_X]$
	be a combinatorial partial log compactification for some tropical
	partial compactification $\mathcal{Y}\to \mathcal{X}$.
	By Proposition \ref{prop:mono_base_change}, we know that
	$Y\to [X/G_X]$ is a log étale monomorphism representable by
	log DM stacks of finite presentation. The section
	$\Sigma_{\mathcal{A}_S}\to \mathcal{Y}$ over the unit section
	and the unit section $S\to [X/G_X]$ induce a section
	$S\to Y$ over the unit section $S\to [X/G_X]$. We must show that
	$Y$ is a stack for the strict fppf topology. Let $Z$ be an fs
	log scheme, $U\to Z$ a strict fppf cover of $Z$ and $\xi$
	a descent datum for a morphism $Z\to Y$ relative to
	$U\to Z$. Since $[X/G_X]$ is a stack for the strict fppf topology,
	the descent datum induced by $\xi$ for a morphism $Z\to [X/G_X]$
	relative to $U\to Z$ is effective, so we may assume that
	there exists a morphism $Z\to [X/G_X]$ and $\xi$ is a descent
	datum for a lift of this morphism. Hence it suffices to
	show that any descent datum for a lift
	$Z\to Z\times_{[X/G_X]}Y$ over the identity $Z\to Z$ is effective.
	This holds, since $Z\times_{[X/G_X]}Y$ is a log DM stack
	by Proposition \ref{prop:mono_base_change} and hence in
	particular satisfies strict fppf descent.

	By Lemma \ref{lem:proper_iff_surj}, $Y$ is a log compactification
	of $[X/G_X]$ if and only if $\mathcal{Y}$ is a tropical compactification.

	Pulling back along $X\to [X/G_X]$, the same results hold
	for (partial) log compactifications of $G_X$.
\end{proof}

\subsection{The correspondence theorem}

The goal of this section is to prove the following theorem:

\begin{theorem}
	\label{thm:trop_corresp}
	Let $X\to S$ be a log semiabelian variety over $S$. There exists an Artin fan $S\to \mathcal{A}_S$ for $S$
	and a split family of tropical semiabelian varieties $\mathcal{X}\to \Sigma_{\mathcal{A}_S}$ such that
	$[X/G_X]\cong S\times_{\mathcal{A}_S}\mathcal{A}_{\mathcal{X}}$. Moreover, the following statements hold:
	\begin{enumerate}
		\item Every partial log compactification of $[X/G_X]$ is isomorphic to an open subset of a combinatorial partial log compactification of $[X/G_X]$.
		\item The functor $\mathcal{Y}\mapsto S\times_{\mathcal{A}_S}\mathcal{A}_{\mathcal{Y}}$ is an equivalence of categories between
			the category of tropical compactifications of $\mathcal{X}\to \Sigma_{\mathcal{A}_S}$
			and the category of log compactifications of $[X/G_X]$.
		\item If $S$ is a point with trivial log structure then $\mathcal{Y}\mapsto S\times_{\mathcal{A}_S}\mathcal{A}_{\mathcal{Y}}$ is an
			equivalence of categories between the category of partial tropical compactifications of $\mathcal{X}\to \Sigma_{\mathcal{A}_S}$
			and the category of partial log compactifications of $[X/G_X]$.
	\end{enumerate}
\end{theorem}

\begin{corollary}
	\label{cor:trop_corresp_log}
	In the above situation there is an equivalence of categories between the category of tropical compactifications
	of $\mathcal{X}\to \Sigma_{\mathcal{A}_S}$ and the category of log compactifications
	of $G_X\to S$ which sends $\mathcal{Y}\to \mathcal{X}$ to $X\times_{\mathcal{A}_{\mathcal{X}}}\mathcal{A}_{\mathcal{Y}}$.
\end{corollary}

\begin{proof}
	Combine Theorem \ref{thm:trop_corresp}, point (2) and Lemma \ref{lem:equiv_cat}.
\end{proof}

The upshot of Theorem \ref{thm:trop_corresp} and Corollary \ref{cor:trop_corresp_log} is that log compactifications
of $G_X\to S$ can be described in a completely combinatorial manner. This will allow us to translate the birational
classification problem for log compactifications into a combinatorial statement.
Theorem \ref{thm:trop_corresp} will be deduced from the following more general statement:

\begin{theorem}
	\label{thm:trop_cpt_open_subsets}
	Let $S$ be a log smooth log scheme with Artin fan $\mathcal{A}_{S}$ and $\mathcal{X}$ be a stack over $\RPC/\Sigma_{\mathcal{A}_S}$.
	Write $X:=S\times_{\mathcal{A}_S}\mathcal{A}_{\mathcal{X}}$ and let
	$Y\to X$ be a log étale monomorphism representable by log DM stacks of finite presentation. Then there
	exists a unique initial monomorphism $\mathcal{Y}\to \mathcal{X}$ representable by finite cone stacks such that the map
	$Y\to X$ lifts to a strict open immersion $Y\to X\times_{\mathcal{A}_{\mathcal{X}}}\mathcal{A}_{\mathcal{Y}}$.
	Moreover, the following statements hold:
\begin{enumerate}
	\item If $S$ is a point with trivial log structure, then $Y\to X\times_{\mathcal{A}_{\mathcal{X}}}\mathcal{A}_{\mathcal{Y}}$ is an isomorphism.
	\item If $Y=X\times_{\mathcal{A}_{\mathcal{X}}}\mathcal{A}_{\mathcal{Y}'}$ for some log étale monomorphism $\mathcal{Y}'\to \mathcal{X}$ representable by finite cone stacks, then $\mathcal{Y}=\mathcal{Y}'$.
	\item If $\mathcal{X}\to \Sigma_{\mathcal{A}_S}$ admits a section $s:\Sigma_{\mathcal{A}_S}\to \mathcal{X}$ such
		that the induced section $S\to X$ lifts to a section $S\to Y$, then $s$ lifts to a (unique) section $\Sigma_{\mathcal{A}_S}\to \mathcal{Y}$.
\end{enumerate}
\end{theorem}

The remainder of this section is devoted primarily to the proofs of Theorem
\ref{thm:trop_corresp} and Theorem \ref{thm:trop_cpt_open_subsets}. We start
by showing that $\mathcal{A}_S$ as in the statement of Theorem
\ref{thm:trop_corresp} exists.

\begin{proposition}
	\label{prop:A_S_exists}
	Let $X\to S$ be a log semiabelian variety. Then there exists an Artin fan $S\to \mathcal{A}_S$ and a
	unique split family of tropical semiabelian varieties $\mathcal{X}\to \Sigma_{\mathcal{A}_S}$ such that $[X/G_X]\cong \mathcal{A}_{\mathcal{X}}\times_{\mathcal{A}_S}S$.
\end{proposition}

\begin{proof}
	Step 1 (Construct the Artin fan $\mathcal{A}_S$) Write $G_X$ as an extension of a semiabelian
	variety which is abelian over the generic point of $S$ by an algebraic torus $T$. Pick an étale cover $S'\to S$ on which $T$
	is trivialized, i.e., $T\times_{S}S'\cong S'\times \mathbb{G}_m^{r}$.
	Refining $S'$ if necessary, we can write $X$ as an extension of a log abelian variety $A$ by $T_{\log}$.
	Let $\overline{M},\overline{N}$ be the constructible sheaves
	such that
	$$[A/G_A]\cong \Hom(\underline{M},\mathbb{G}_{m,\trop})^{(N)}/\underline{N}$$
	(see \cite[4.4]{LogarithmicAbeKajiwa2008}). By \cite[4.1.2]{LogarithmicAbeKajiwa2008}
	we may refine $S'$ such that the restrictions $\overline{M}'$
	and $\overline{N}'$ of $\overline{M}$ and $\overline{N}$ to $S'$
	are generated by their global sections. After further refining
	$S'$, we can assume that each connected component of $S'$ has a
	global chart by a sharp fs monoid and that on each connected
	component $X\in \Ext(A,\mathbb{G}_{m,\log}^{r})$ lies in the image of
	$\Ext(A,\mathbb{G}_{m}^{r})$. We obtain a presentation
	$S'\times_S S'\rightrightarrows S'$ of $S$. Pick an Artin
	fan $S'\to \mathcal{A}_{S'}$ for $S'$. By
	\cite[Theorem 5.7]{SkeletonsAndFAbramo2015} we may choose
	an Artin fan $\mathcal{A}_{S'\times_S S'}$ such that there
	is a commutative diagram:
	$$\begin{tikzcd}
		S'\times_S S'\arrow[d]\arrow[r,shift left]\arrow[r,shift right]&S'\arrow[d]\\
		\mathcal{A}_{S'\times_S S'}\arrow[r,shift left]\arrow[r,shift right]&\mathcal{A}_{S'}.
	\end{tikzcd}$$
	Since the maps $S'\times_S S'\to S'$ are strict étale and being strict
	étale descends along strict smooth covers, the maps
	$\mathcal{A}_{S'\times_S S'}\to \mathcal{A}_{S'}$ are strict
	étale. Let $\mathcal{A}_{S}$ be the colimit of the diagram
	$\mathcal{A}_{S'\times_S S'}\to \mathcal{A}_{S'}$. Since the
	two maps $S'\times_S S'\rightrightarrows \mathcal{A}_S$ agree,
	the morphism $S'\to \mathcal{A}_S$ descends to a morphism
	$S\to \mathcal{A}_{S}$. This morphism is strict, since this property
	is strict étale local and holds for $S'\to \mathcal{A}_{S'}$.

	Step 2 (Reduce to the case in which $S$ has a global chart): For existence and uniqueness of $\mathcal{X}$ it suffices to show
	that strict étale locally on $\mathcal{A}_S$ the stack $\mathcal{X}$ exists
	and is unique up to unique isomorphism.
	Hence it suffices to consider the case
	when $S$ is an irreducible component of $S'$. In particular, we
	may assume that $S$ has the following additional properties:
	\begin{enumerate}
		\item $T\cong \mathbb{G}_{m}^{r}$ and $T_{\log}\cong \mathbb{G}_{m,\log}^{r}$.
		\item $\overline{M}$ and $\overline{N}$ are generated by their global sections which we denote by $M$ and $N$.
		\item $S$ has a global chart by an fs monoid, i.e., we have a strict smooth surjective morphism $S\to \mathcal{A}_{\sigma}$ for some RPC $\sigma$.
	\end{enumerate}

	Step 3 (Show that $\mathcal{A}_{\mathbb{G}_{m,\trop}}=\mathbb{G}_{m,\trop}$): Recall the sheaf $\mathbb{G}_{m,\trop}:=T_{\mathbb{Z}}^{\trop}$
	on $\RPC$ from Definition \ref{def:trop_torus}. We claim that $\mathcal{A}_{\mathbb{G}_{m,\trop}}=\mathbb{G}_{m,\trop}$. Indeed,
	the datum of a morphism $\sigma\to \mathbb{G}_{m,\trop}$ in $\RPC$ is equivalent to the datum of a
	section $m\in M_{\sigma}$. Hence for every $M\in \ShpMon^{\op}$, we have:
	$$\Hom(M,\overline{\mathbb{G}}_{m,\trop})\cong\colim_{Q\subseteq M}\Hom(Q,\overline{\mathbb{G}}_{m,\trop})\cong\colim_{Q\subseteq M}Q^{\gp}$$
	where the colimits extend over the set of toric monoids contained in $M$.
	Since $M\mapsto M^{\gp}$ is left adjoint to the inclusion functor
	from the category of groups into the category of monoids it
	preserves colimits:
	$$\colim_{Q\subseteq M}Q^{\gp}\cong (\colim_{Q\subseteq M}Q)^{\gp}\cong M^{\gp}.$$
	It follows that $\mathcal{A}_{\mathbb{G}_{m,\trop}}'$ is the
	functor which assigns to a log scheme $Z$ the group $\overline{M}_Z^{\gp}(Z)$.
	Since this is already a sheaf and is equal to $\mathbb{G}_{m,\trop}$
	we conclude that $\mathcal{A}_{\mathbb{G}_{m,\trop}}\cong \mathbb{G}_{m,\trop}$.

	Step 4 (Show that $[A/G_A]$ is the pullback of a tropical abelian variety):
	Let $\langle\ ,\ \rangle:M\times N\to \overline{M}_S^{\gp}(S)=M_{\sigma}$
	be the pairing from \cite[4.1.2 and 4.4]{LogarithmicAbeKajiwa2008}.
	The principal polarization on $A$ gives an isomorphism $M\cong N$
	such that under this identification
	the pairing $Q:M\times M\to M_{\sigma}$ induced by
	$\langle\ ,\ \rangle$ is symmetric and positive-definite
	(see \cite{LogarithmicAbeKajiwa2008}). The tuple $(M,Q)$ gives a family of tropical abelian varieties
	$\mathcal{A}_{\Hom(\underline{M},\mathbb{G}_{m,\trop})^{(M)}/Q_{\hom}(\underline{M})}$ over
	$\mathcal{A}_{\sigma}$. It follows from the definitions
	that $[A/G_A]$ is the pullback of this family
	along $S\to \mathcal{A}_{\sigma}$.

	Step 5 (Show the extension is trivial and the morphism is unique): Note that since $X\in \Ext(A,\mathbb{G}_{m,\log}^{r})$
	is the image of an element of $\Ext(A,\mathbb{G}_{m}^{r})$,
	the image $[X/G_X]$ in $\Ext([A/G_A],\mathbb{G}_{m,\trop}^{r})$
	is trivial, i.e., $[X/G_X]\to [A/G_A]$ is the trivial
	$\mathbb{G}_{m,\trop}^{r}$-extension of $[A/G_A]$. In fact,
	this trivialization is canonical, since the restriction
	of $[X/G_X]\to [A/G_A]$ over $0\in [A/G_A]$ is canonically trivial.
	The construction of $(M,Q)$ is canonical and independent
	of any choices and it follows that $\Hom(\underline{M},\mathbb{G}_{m,\trop})^{(M)}/Q_{\hom}(\underline{M})\to \sigma$ is determined
	up to canonical isomorphism. Since $[X/G_X]\to [A/G_A]$ is canonically trivial, the
	morphism
	$$[X/G_X]\to \mathcal{A}_{\mathbb{G}_{m,\trop}^{r}\times\Hom(\underline{M},\mathbb{G}_{m,\trop})^{(M)}/Q_{\hom}(\underline{M})}$$
	is also unique up to unique isomorphism.
\end{proof}

The following is a technical lemma which allows us to tropicalize log étale monomorphisms:

\begin{lemma}
	Let $X$ be a log smooth log DM stack with a global chart and $Y\to X$ a log étale monomorphism representable
	by log DM stacks of finite presentation. Write $X\to \mathcal{A}_{\tau}$ for the Artin fan of $X$ corresponding to the global chart.
	Then there exists a monomorphism $\Sigma\to \tau$ such that $\mathcal{A}_{\Sigma}$ is an Artin fan for $Y$
	and $Y\to X\times_{\mathcal{A}_{\tau}}\mathcal{A}_{\Sigma}$ is an open immersion.
\end{lemma}

\begin{proof}
	By \cite[Proposition 2.2]{Grothendieck_to_Hu_Xi_2025} we know
	that $Y$ has a Zariski log structure. Cover $Y$ by Zariski open subsets $\{U_i\}_i$ such
	that each $U_i$ is atomic with global chart corresponding to a morphism
	$U_i\to \mathcal{A}_{\tau_i}$. Note that in this case $U_i\cap U_j$ also has a global chart
	and we write $U_i\cap U_j\to \mathcal{A}_{\tau_{i,j}}$ for the corresponding Artin fan.
	We may then construct an Artin fan $\mathcal{A}_Y$ for $Y$ as the colimit:
	$$\bigsqcup_{i,j}\mathcal{A}_{\tau_{i,j}}\rightrightarrows \bigsqcup_{i}\mathcal{A}_{\tau_i}.$$
	Let $\Sigma$ be the colimit of the corresponding diagram:
	$$\bigsqcup_{i,j}\tau_{i,j}\rightrightarrows \bigsqcup_{i}\tau_i,$$
	so that $\mathcal{A}_Y=\mathcal{A}_{\Sigma}$.
	The map $U_i\to X$ (resp. $U_i\cap U_j\to X$) is log étale and hence the induced morphism $\tau_i\to \tau$
	(resp. $\tau_{i,j}\to \tau$) is injective by Kato's chart
	criterion. We identify $\tau_i$ (resp. $\tau_{i,j}$) with its image in $\tau$.
	To show that $\Sigma\to \tau$ is injective it suffices to show that for all $i$ and $j$,
	the morphism $\tau_{i,j}\to \tau_i\cap \tau_j$ is surjective. Note that $U_i\to X\times_{\mathcal{A}_{\tau}}\mathcal{A}_{\tau_i}$
	and $U_j\to X\times_{\mathcal{A}_{\tau}}\mathcal{A}_{\tau_j}$ are strict log étale monomorphisms, hence open immersions,
	by \cite[Tag 025G]{stacks-project}.
	Let $x\in \tau_i\cap \tau_j$ be a point in the lattice and consider the morphism $\tau':=(\mathbb{R}_{\ge 0},\mathbb{Z})\to \tau_i\cap \tau_j$
	given by sending $1\in \mathbb{R}_{\ge 0}$ to $x$. Then we get lifts
	$X\times_{\mathcal{A}_{\tau}}\mathcal{A}_{\tau'}\to X\times_{\mathcal{A}_{\tau}}\mathcal{A}_{\tau_i}$ and
	$X\times_{\mathcal{A}_{\tau}}\mathcal{A}_{\tau'}\to X\times_{\mathcal{A}_{\tau}}\mathcal{A}_{\tau_j}$ of the projection
	$X\times_{\mathcal{A}_{\tau}}\mathcal{A}_{\tau'}\to X$. Let
	$$V_i,V_j\subseteq X\times_{\mathcal{A}_{\tau}}\mathcal{A}_{\tau'}$$
	be the preimages of $U_i$ resp. $U_j$ under these lifts. Since $X$ is log smooth, the fiber of $X\times_{\mathcal{A}_{\tau}}\mathcal{A}_{\tau_i}\to \mathcal{A}_{\tau_i}$
	over the unique closed stratum is non-empty and since $U_i$ intersects the closed stratum of $X\times_{\mathcal{A}_{\tau}}\mathcal{A}_{\tau_i}$, the preimage of the closed stratum in
	$V_i$ is non-empty (similarly for $V_j$). Since the closed stratum is irreducible and $V_i,V_j$ are open, it follows that
	$V_i\cap V_j$ intersects the closed stratum. Note that $V_i\cap V_j\to U_i$ and $V_i\cap V_j\to U_j$ are two lifts of the projection
	$V_i\cap V_j\to X$. Since $Y\to X$ is a monomorphism, these two maps must be equal, i.e., the maps factor through a map
	$V_i\cap V_j\to U_i\cap U_j$. Since $V_i\cap V_j$ intersects the closed stratum of $X\times_{\mathcal{A}_{\tau}}\mathcal{A}_{\tau'}$
	it follows that $x$ lies in the image of $\mathcal{A}_{\tau_{i,j}}\to \mathcal{A}_{\tau}$ as required. Since
	$Y\to X\times_{\mathcal{A}_{\tau}}\mathcal{A}_{\Sigma}$ is a
	strict étale monomorphism, it is an open immersion.
\end{proof}

\begin{proof}[Proof of Theorem \ref{thm:trop_cpt_open_subsets}]
	Step 1 (Define the stack $\mathcal{Y}$ and list properties to show): We first assume that $S$ is a general log smooth log scheme.

	Let $Y\to X$ be a log étale monomorphism representable by log DM stacks of finite presentation.
	We define a stack $\mathcal{Y}$ over $\mathcal{X}$ as follows:
	Let $\tau\to \mathcal{X}$ be a morphism from an RPC $\tau$. We let the restriction of $\mathcal{Y}$ over $\tau$
	be the intersection of all subsheaves $F\subseteq \tau$ satisfying:
	\begin{itemize}
		\item $Y\times_{\mathcal{A}_{\mathcal{X}}}\mathcal{A}_{\tau}\to \mathcal{A}_{\tau}$ lifts to a morphism $Y\times_{\mathcal{A}_{\mathcal{X}}}\mathcal{A}_{\tau}\to \mathcal{A}_{F}$.
		\item $F$ is a cone stack.
	\end{itemize}
	Note that the condition that $Y\times_{\mathcal{A}_{\mathcal{X}}}\mathcal{A}_{\tau}\to \mathcal{A}_{\tau}$ lifts to a morphism $Y\times_{\mathcal{A}_{\mathcal{X}}}\mathcal{A}_{\tau}\to \mathcal{A}_{F}$
	is equivalent to the statement that $Y\times_{\mathcal{A}_{\mathcal{X}}}\mathcal{A}_F\to Y\times_{\mathcal{A}_{\mathcal{X}}}\mathcal{A}_{\tau}$
	is an isomorphism. In particular, the construction is invariant under replacing $Y$ with a strict étale cover.
	To prove the main statement of the theorem and point (2), we must show:
	\begin{enumerate}
		\item $\mathcal{Y}$ is a functor, i.e., if $\tau'\to \tau$ is a morphism of rational polyhedral cones, then
			every section of the subsheaf of $\tau$ defined above restricts to a section of the corresponding subsheaf over $\tau'$.
		\item $\mathcal{Y}\to \mathcal{X}$ is representable by finite cone stacks.
		\item Any section $s:\Sigma_{\mathcal{A}_S}\to \mathcal{X}$ as in the statement of the theorem lifts to $\mathcal{Y}$.
		\item $Y\to X\times_{\mathcal{A}_{\mathcal{X}}}\mathcal{A}_{\mathcal{Y}}$ is an open immersion.
		\item If $Y=X\times_{\mathcal{A}_{\mathcal{X}}}\mathcal{A}_{\mathcal{Y}'}$ for some log étale monomorphism $\mathcal{Y}'\to \mathcal{X}$,
			then $\mathcal{Y}=\mathcal{Y}'$.
	\end{enumerate}

	Step 2 (Show all properties can be checked after base-change along a map $\tau\to \mathcal{X}$ from an RPC $\tau$):
	For (1) and (2), this is clear by definition. For (3), note that
	this statement is equivalent to the statement that the map
	$\mathcal{Y}\times_{\mathcal{X}}\Sigma_{\mathcal{A}_S}\to \Sigma_{\mathcal{A}_S}$ is an
	isomorphism, where
	$\Sigma_{\mathcal{A}_S}\to \mathcal{X}$ is the given section. This can be checked over
	each cone in $\Sigma_{\mathcal{A}_S}$. For (4),
	we must show that for any morphism $Z\to X\times_{\mathcal{A}_{\mathcal{X}}}\mathcal{A}_{\mathcal{Y}}$,
	the base-change of $Y$ over $Z$ is an open immersion. This is strict étale local on the target by
	\cite[Tag 02L3]{stacks-project}, so by Lemma \ref{lem:loc_factorization} we may assume that $Z\to \mathcal{A}_{\mathcal{X}}$
	factors through $Z\to \mathcal{A}_{\tau}\to \mathcal{A}_{\mathcal{X}}$ for some RPC $\tau$. Hence by
	Corollary \ref{cor:fib_prods_preserved} it suffices to show (4) after base-change over any $\tau$.
	Finally, for (5), note that $\mathcal{Y}$ is a subsheaf of $\mathcal{Y}'$ by minimality of $\mathcal{Y}$. We get
	a morphism $\mathcal{Y}\to \mathcal{Y}'$ and we must show that it is an isomorphism, which can be checked
	after restricting to an arbitrary cone.

	Step 3 (Compute the restriction $\mathcal{Y}|_{\tau}$): Now note that:
	$$Y\times_{\mathcal{A}_{\mathcal{X}}}\mathcal{A}_{\tau}\cong Y\times_{X}(S\times_{\mathcal{A}_S}\mathcal{A}_{\tau}).$$
	$S\times_{\mathcal{A}_S}\mathcal{A}_{\tau}$ is a
	log smooth log DM stack
	and hence by representability of $Y\to X$, so is $Y\times_{\mathcal{A}_{\mathcal{X}}}\mathcal{A}_{\tau}$.
	Replacing $\mathcal{X}$ with $\tau$ and $Y$ with $Y\times_{\mathcal{A}_{\mathcal{X}}}\mathcal{A}_{\tau}$,
	we reduce to the case when $Y$ is a log smooth log DM stack and $\mathcal{X}=\tau$ is an RPC. In particular,
	$X=S\times_{\mathcal{A}_S}\mathcal{A}_{\tau}$ is a log smooth log DM stack. Replacing $X$
	with a strict étale cover, we may assume that $X$ has a global chart. Let $\mathcal{A}_{Y}$
	be the Artin fan for $Y$ as in the preceding lemma, i.e., $\Sigma_{\mathcal{A}_Y}\to \tau$ is a monomorphism.
	By Lemma \ref{lem:mono_impl_fan} we know that $\Sigma_{\mathcal{A}_Y}$ is a stacky fan in $\tau$.
	In particular, $\Sigma_{\mathcal{A}_Y}$ is a subsheaf of $\tau$. Since $Y\to \mathcal{A}_Y$ is strict and its image contains all
	points corresponding to maximal cones in $\mathcal{A}_Y$, we get that $\Sigma_{\mathcal{A}_Y}$ is minimal,
	i.e., $\mathcal{Y}|_{\tau}=\Sigma_{\mathcal{A}_Y}$.

	Step 4 (Prove properties (1), (2), and (4)): If $\tau'\to \tau$ is a morphism, then since $Y\to \mathcal{A}_Y$ is a strict smooth cover,
	it is clear that
	$$\Sigma_{\mathcal{A}_Y}\times_{\tau}\tau'=\Sigma_{\mathcal{A}_{Y\times_{\mathcal{A}_{\tau}}\mathcal{A}_{\tau'}}}$$
	is the minimal subsheaf $F$ such that $Y\times_{\mathcal{A}_{\tau}}\mathcal{A}_{\tau'}\to \mathcal{A}_{\tau'}$ factors through $F$,
	showing that $\mathcal{Y}$ is indeed a functor. Representability of $\mathcal{Y}$ follows from the fact that
	$\Sigma_{\mathcal{A}_Y}$ is a finite cone complex. By construction
	of $\mathcal{A}_Y$, the map $Y\to \mathcal{A}_Y$ is strict.
	Hence by base-change, so is $Y\to \mathcal{A}_Y\times_{\mathcal{A}_{\tau}}X$.
	It follows that this map is a strict étale monomorphism, i.e.,
	an open immersion by \cite[Tag 025G]{stacks-project}. This finishes the proof of (1), (2) and (4) of the above list.

	Step 5 (Prove property (5) - uniqueness): Assume now that $Y=X\times_{\mathcal{A}_{\mathcal{X}}}\mathcal{A}_{\mathcal{Y}'}$ for some monomorphism
	$\mathcal{Y}'\to \mathcal{X}=\tau$. By Lemma \ref{lem:mono_impl_fan} we know that $\mathcal{Y}'$ is isomorphic to a stacky fan in $\tau$.
	Since $S$ is log smooth, the map $S\times_{\mathcal{A}_{S}}\mathcal{A}_{\tau}\to \mathcal{A}_{\tau}$ is surjective. To see this,
	note that its image is open
	and it contains the unique closed point in $\mathcal{A}_{\tau}$, which lies in the closure of all other points. By base-change we get that
	$Y\to \mathcal{A}_{\mathcal{Y}'}$ is surjective. In particular, the map $\mathcal{A}_{\mathcal{Y}}\to \mathcal{A}_{\mathcal{Y}'}$
	must be surjective and hence an isomorphism, since both $\mathcal{Y}$ and $\mathcal{Y}'$ are stacky fans in $\tau$.

	Step 6 (Prove property (3) - existence of lifts): We replace $S$ with $S\times_{\mathcal{A}_S}\mathcal{A}_{\tau}$.
	Note that by assumption, the induced section $S\to X$ lifts to a section of $Y$, i.e.,
	$Y\times_{X}S\cong S$. In particular, since $S\to \mathcal{A}_S=\mathcal{A}_{\tau}$ is initial
	with respect to all morphisms from $S$ to an Artin fan, any
	substack of $\mathcal{X}$ with the property that $Y\times_{\mathcal{A}_{\mathcal{X}}}\mathcal{A}_{\tau}\to \mathcal{A}_{\mathcal{X}}$
	factors through it, must contain the image of $\tau\to \Sigma_{\mathcal{A}_S}\to \mathcal{X}$.

	Step 7 (Prove point (1) from the theorem statement): It remains to prove point (1) from the statement of the theorem, i.e.,
	that the map
	$$Y\to X\times_{\mathcal{A}_{\mathcal{X}}}\mathcal{A}_{\mathcal{Y}}=S\times_{\mathcal{A}_{S}}\mathcal{A}_{\mathcal{Y}}$$
	is an isomorphism in the case when $S$ is a point with trivial log structure.
	This can be checked after base-change to any
	log scheme $Z\to X\times_{\mathcal{A}_{\mathcal{X}}}\mathcal{A}_{\mathcal{Y}}$.
	By Lemma \ref{lem:loc_factorization} we may strict étale locally factor $Z\to \mathcal{A}_{\mathcal{X}}$ as
	$Z\to \mathcal{A}_{\tau}\to \mathcal{A}_{\mathcal{X}}$ for some RPC $\tau$.
	Note that the map $Z\to \mathcal{A}_{\tau}$ can strict étale locally be lifted to a factorization $Z\to X_{\tau}\to \mathcal{A}_{\tau}$
	where $X_{\tau}$ is the affine toric variety with cone $\tau$.
	Write $X_{\mathcal{Y}}:=X_{\tau}\times_{\mathcal{A}_{\mathcal{X}}}\mathcal{A}_{\mathcal{Y}}$. By the above, $\mathcal{Y}$
	is a stacky fan in $\tau$, so $X_{\mathcal{Y}}$ is a toric DM stack in the sense of \cite{ATheoryOfStaGillam2015}.
	$X_{\tau}\times_{X}Y$ is a torus-invariant open subset of $X_{\mathcal{Y}}$. Hence $X_{\tau}\times_{X}Y$
	is isomorphic to $X_{\Delta}$ for some stacky fan $\Delta\subseteq \mathcal{Y}$. But the image
	of $X_{\tau}\times_{X}Y$ in $X_{\mathcal{Y}}$ intersects every closed stratum by minimality of $\mathcal{Y}$, so $\Delta=\mathcal{Y}$.
\end{proof}

\begin{proof}[Proof of Theorem \ref{thm:trop_corresp}]
	Point (1) follows from Theorem \ref{thm:trop_cpt_open_subsets} and point (3) follows from point (1)
	of Theorem \ref{thm:trop_cpt_open_subsets} (take the space $X$ in Theorem \ref{thm:trop_cpt_open_subsets} to be $[X/G_X]$). From Lemma \ref{lem:proper_iff_surj} we know that a combinatorial
	partial log compactification $[X/G_X]\times_{\mathcal{A}_{\mathcal{X}}}\mathcal{A}_{\mathcal{Y}}$ is a log compactification of $[X/G_X]$
	if and only if $\mathcal{Y}\to \mathcal{X}$ is a tropical compactification. Let $Y$ now be an arbitrary log compactification of $[X/G_X]$.
	From Theorem \ref{thm:trop_cpt_open_subsets} we know that $Y$ is an open subset of a stack of the
	form $[X/G_X]\times_{\mathcal{A}_{\mathcal{X}}}\mathcal{A}_{\mathcal{Y}}$ for some partial tropical compactification
	$\mathcal{Y}\to \mathcal{X}$. We must show that $Y\to [X/G_X]\times_{\mathcal{A}_{\mathcal{X}}}\mathcal{A}_{\mathcal{Y}}$
	is an isomorphism. Equivalently, we show that
	$$X\times_{[X/G_X]}Y\to X\times_{\mathcal{A}_{\mathcal{X}}}\mathcal{A}_{\mathcal{Y}}$$
	is an isomorphism. Since this is a morphism over $X$ and $X$ has a log étale cover by log schemes by
	Lemma \ref{lem:part_cptifications_are_nice}, it suffices to check that this morphism is an isomorphism
	after base-change along some log étale morphism $Z\to X$. Recall that $G_X\subseteq X\times_{[X/G_X]}Y$ and
	$G_X\subseteq X\times_{\mathcal{A}_{\mathcal{X}}}\mathcal{A}_{\mathcal{Y}}$ is dense by Lemma
	\ref{lem:bir_is_bir}, so $Z\times_{[X/G_X]}Y\to Z\times_{\mathcal{A}_{\mathcal{X}}}\mathcal{A}_{\mathcal{Y}}$
	is an isomorphism on a dense open subset. So
	$$Z\times_{[X/G_X]}Y\to Z\times_{\mathcal{A}_{\mathcal{X}}}\mathcal{A}_{\mathcal{Y}}$$
	is a proper, birational, open immersion. In particular, it is surjective, hence an isomorphism. It follows
	from point (2) in Theorem \ref{thm:trop_cpt_open_subsets} that $\mathcal{Y}$ is uniquely determined by $Y$,
	and we conclude.
\end{proof}

We end this section with the following two technical results which will be used later:

\begin{lemma}
	\label{lem:repr_when_cone_cx}
	Let $\mathcal{X}\to \Sigma_{\mathcal{A}_S}$ be as in the statement of Theorem \ref{thm:trop_corresp}.
	If $\mathcal{Y}\to \mathcal{X}$ is a partial tropical compactification such that $\mathcal{Y}$ is representable
	by a finite cone complex, then $\mathcal{A}_{\mathcal{Y}}\times_{\mathcal{A}_{\mathcal{X}}}X$ is representable by a log DM stack.
\end{lemma}

\begin{proof}
	Let $\tau\in \mathcal{Y}$ be a cone. We claim that $\mathcal{A}_{\tau}\to \mathcal{A}_{\mathcal{Y}}$ is representable by
	open immersions. For this, let $T\to \mathcal{A}_{\mathcal{Y}}$ be a morphism from a log scheme $T$. Note that
	being an open immersion is strict étale local on the target by \cite[Tag 02L3]{stacks-project}. So by Lemma \ref{lem:loc_factorization}
	we may assume that $T\to \mathcal{A}_{\mathcal{Y}}$ factors as
	$T\to \mathcal{A}_{\sigma}\to \mathcal{A}_{\mathcal{Y}}$ for some RPC $\sigma$ and strict map $T\to \mathcal{A}_{\sigma}$.
	Note that $\tau\to \mathcal{Y}$ is representable by face morphisms, so by Corollary \ref{cor:fib_prods_preserved}
	we may assume that $\mathcal{Y}=\sigma$ is an RPC and $\tau$ is a face of $\sigma$. In this case, by Lemma \ref{lem:rpc_repr},
	we have $\mathcal{A}_{\sigma}\cong [X_{\sigma}/T_{\sigma}]$ where $X_{\sigma}$ is the toric variety with cone $\sigma$
	and $T_{\sigma} \subseteq X_{\sigma}$ is the dense open torus. The base-change of $\mathcal{A}_{\tau}$ over $X_{\sigma}$
	is the open subset $U_{\tau}\subseteq X_{\sigma}$ corresponding to the face $\tau$ (by the standard theory of toric
	varieties) and hence by \cite[Tag 02L3]{stacks-project}, $\mathcal{A}_{\tau}\to \mathcal{A}_{\sigma}$ is representable
	by open immersions.

	Since representability is Zariski local, we reduce to the case $\mathcal{Y}=\tau$ by covering $\mathcal{Y}$ by disjoint
	unions of cones. By
	\cite[Theorem 1.1 (3)]{fortman2025descentalgebraicstacks},
	representability is strict étale local on $S$, so we may
	assume that $X\cong A\times \mathbb{G}_{m,\log}^{r}$
	for some $r\ge 0$. Let $\tau'\subseteq \mathbb{R}^{r}$ be the image of $\tau$
	under the projection $\mathcal{X}\to T_{\mathbb{Z}^{r}}^{\trop}$
	induced by $A\times \mathbb{G}_{m,\log}^{r}\to \mathbb{G}_{m,\log}^{r}$.
	Equip $\tau'$ with the integral structure coming from
	$\mathbb{Z}^{r}$. The composition
	$$X\times_{\mathcal{A}_{\mathcal{X}}}\mathcal{A}_{\tau}\to X\to \mathbb{G}_{m,\log}^{r}$$
	factors through $\mathbb{G}_{m,\log}^{r}\times_{\mathbb{G}_{m,\trop}^{r}}\mathcal{A}_{\tau'}\cong Y_{\tau'}$,
	where $Y_{\tau'}$ is the affine toric variety with cone $\tau'$
	and lattice $\mathbb{Z}^{r}$.
	Hence replacing $S$ with $S\times Y_{\tau'}$ we may assume that
	$X$ is a log abelian variety. In this case, the morphism
	$\tau\to \mathcal{X}$ defines a subcone of $C$ in the
	notation of \cite[2.2]{LogarithmicAbeKajiwa2015}.
	Let $\tau''$ be the subcone of $C$ equipped with the integral
	structure pulled back from $C$. Then by \cite[Theorem 8.1]{LogarithmicAbeKajiwa2015},
	the fiber product $X\times_{\mathcal{A}_{\mathcal{X}}}\mathcal{A}_{\tau''}$ is
	representable by a log algebraic space. It remains to show
	that $\mathcal{A}_{\tau}\to \mathcal{A}_{\tau''}$ is
	representable by log DM stacks. This follows as in the proof
	of Proposition \ref{prop:mono_base_change}.
\end{proof}

\begin{proposition}
	\label{prop:bir_iff_subdiv}
	Let $\mathcal{X}\to \mathcal{Z}$ be a split family of tropical semiabelian varieties and
	let $\mathcal{Y}_1\to \mathcal{Y}_2$ be a morphism of partial tropical compactifications of $\mathcal{X}\to \mathcal{Z}$.
	Then $\mathcal{A}_{\mathcal{Y}_1}\to \mathcal{A}_{\mathcal{Y}_2}$ is proper and representable by log algebraic spaces
	if and only if for every morphism $\sigma\to \mathcal{Y}_2$ from an RPC $\sigma$, the map
	$\mathcal{Y}_1\times_{\mathcal{Y}_2}\sigma\to \sigma$ is a subdivision of $\sigma$.
\end{proposition}

\begin{proof}
	Assume first that the second condition is verified and consider a morphism $W\to \mathcal{A}_{\mathcal{Y}_2}$ from
	a log scheme $W$. By Lemma \ref{lem:loc_factorization} and the definition of a chart of a log scheme, strict étale locally on $W$ we have a factorization $W\to X_{\sigma}\to \mathcal{A}_{\sigma}\to \mathcal{A}_{\mathcal{Y}_2}$
	for some RPC $\sigma$. By Corollary \ref{cor:fib_prods_preserved},
	it holds that $\mathcal{A}_{\sigma}\times_{\mathcal{A}_{\mathcal{Y}_2}}\mathcal{A}_{\mathcal{Y}_1}\cong \mathcal{A}_{\sigma\times_{\mathcal{Y}_2}\mathcal{Y}_1}$ and
	by assumption, $\sigma\times_{\mathcal{Y}_2}\mathcal{Y}_1\to \sigma$ is a subdivision. So
	$$X_{\sigma}\times_{\mathcal{A}_{\mathcal{Y}_2}}\mathcal{A}_{\mathcal{Y}_1}\cong X_{\sigma\times_{\mathcal{Y}_2}\mathcal{Y}_1}\to X_{\sigma}$$
	is a proper morphism of toric varieties. In particular, the pullback to $W$ is a proper morphism of log schemes.

	Assume now that $\mathcal{A}_{\mathcal{Y}_1}\to \mathcal{A}_{\mathcal{Y}_2}$ is a proper morphism representable
	by log algebraic spaces. By Lemma \ref{lem:morph_part_comp_repr} we know that $\mathcal{Y}_1\to \mathcal{Y}_2$
	is a monomorphism representable by finite cone stacks. It then follows from points 1 and 2
	of Proposition \ref{prop:mono_base_change} that $\mathcal{Y}_1\times_{\mathcal{Y}_2}\sigma$ is a subdivision of $\sigma$.
\end{proof}

\section{Solution of the combinatorial problem}
\label{sec:combinat_problem}

\begin{definition}
	Let $X\to Y$ be a morphism of stacks over $\RPC$.
	\begin{enumerate}
		\item We say that $X\to Y$ is a subdivision if
			for every morphism $\sigma\to Y$ from an RPC $\sigma$, the base-change $X\times_{Y}\sigma\to \sigma$ is
			a subdivision.
		\item We say that $X\to Y$ is a root construction if for every morphism $\sigma\to Y$ from an RPC $\sigma$,
			the base-change $X\times_Y\sigma\to(\sigma,N_{\sigma})$ is isomorphic (over $\sigma$) to $(\sigma,N')$ for some finite-index
			sublattice $N'\subseteq N_{\sigma}$.
	\end{enumerate}
\end{definition}

The results of the previous sections reduce the problem of classifying log compactifications
of semiabelian varieties to the problem of classifying tropical compactifications of families of split tropical
semiabelian varieties. We do so by enlarging the category of tropical compactifications and then showing
that in this category every equivalence class contains an explicit unique minimal object.
In the present section we fix a split family $X\to Z$ of tropical semiabelian varieties.

\begin{definition}
	A generalized partial tropical compactification of $X\to Z$ is a monomorphism $Y\to X$ such that there exists a
	subdivision $Y'\to Y$ such that $Y'\to X$ is a partial tropical compactification of $X\to Z$.
	We call a generalized partial tropical compactification a generalized tropical compactification if
	we may choose $Y'$ to be a tropical compactification.
\end{definition}

\begin{definition}
	Let $Y_1,Y_2\to X$ be generalized partial tropical compactifications of $X\to Z$. We say that $Y_1$ and $Y_2$ are birationally
	equivalent if there exists a generalized partial tropical compactification $Y_3$ of $X\to Z$ and a diagram
	$$\begin{tikzcd}
		&Y_3\arrow[dr]\arrow[dl]&\\
		Y_1&&Y_2
	\end{tikzcd}$$
	of morphisms of generalized partial tropical compactifications in which both morphisms are subdivisions. We call the class of all generalized partial tropical compactifications
	of $X\to Z$ which are birationally equivalent to $Y_1$ the birational equivalence class of $Y_1$.
\end{definition}

\begin{remark}
	It is clear that birational equivalence is an equivalence relation since subdivisions
	are stable under base-change and composition.
\end{remark}

\begin{lemma}
	\label{lem:cpt_bir_inv}
	Let $Y_1$ and $Y_2$ be birationally equivalent generalized partial tropical compactifications. Then $Y_1$ is a generalized tropical compactification
	if and only if $Y_2$ is a generalized tropical compactification. In other words, being a tropical compactification is a birational invariant.
\end{lemma}

\begin{proof}
	Since being a generalized tropical compactification can be checked on any subdivision, it suffices to assume that $Y_1$
	and $Y_2$ are partial tropical compactifications.
	By symmetry, it suffices to assume that there exists a subdivision $Y_2\to Y_1$. Let
	$\sigma\to X$ be an arbitrary morphism from an RPC $\sigma$. We claim that $\sigma\times_X Y_2\to \sigma\times_X Y_1$
	is a subdivision. Indeed, if $\tau\to \sigma\times_X Y_1$ is an arbitrary morphism from an RPC, then
	$$\tau\times_{\sigma\times_X Y_1}\sigma\times_X Y_2=\tau\times_{Y_1}Y_2$$
	since $Y_1\to X$ is a monomorphism, so the claim follows from the fact that $Y_2\to Y_1$ is a subdivision.

	Since subdivisions are surjective, the images of $\sigma\times_X Y_2$ and $\sigma\times_X Y_1$ in $\sigma$ are equal,
	showing that $Y_1\to X$ is surjective if and only if $Y_2\to X$ is surjective.
\end{proof}

\begin{definition}
	A generalized partial tropical compactification $Y\to X$ of $X\to Z$ is called minimal if $Y$ satisfies descent
	with respect to subdivisions. More precisely, the following statement holds: Let $Y_1\to X$ be a morphism
	with $Y_1$ an arbitrary stack over $\RPC$ and let $Y_2\to Y_1$ be a subdivision. Assume we are given a morphism $f:Y_2\to Y$ over $X$.
	Then the morphism $f$ descends uniquely to a morphism $Y_1\to Y$ (over $X$), i.e., $f$ is the precomposition of this morphism
	with $Y_2\to Y_1$.
\end{definition}

The name ``minimal'' is justified by the following lemma:

\begin{lemma}
	Let $Y\to X$ be a minimal generalized partial tropical compactification of $X\to Z$ and $Y_1\to X$ a generalized partial tropical compactification
	birationally equivalent to $Y$. Then there exists a unique morphism of generalized partial tropical compactifications $Y_1\to Y$.
\end{lemma}

\begin{proof}
	Uniqueness is clear, since $Y_1,Y\to X$ are monomorphisms.
	Let $Y_2\to X$ be a generalized partial tropical compactification such that there exist subdivisions
	$Y_2\to Y_1$ and $Y_2\to Y$ as in the definition of birational equivalence.
	Since $Y$ is minimal, the morphism $Y_2\to Y$ descends uniquely to a morphism $Y_1\to Y$.
\end{proof}

\begin{definition}
	Let $Y\to X$ be a generalized partial tropical compactification of $X\to Z$. We write $M_Y$ for the stack over $\RPC$
	constructed as follows: A morphism $\sigma\to M_Y$ is the datum of a morphism $\sigma\to X$ such that
	there exists a subdivision $\Sigma\to \sigma$ and a map $\Sigma\to Y$ lifting $\Sigma\to \sigma\to X$.
\end{definition}

\begin{lemma}
	The stack $M_Y$ is a generalized partial tropical compactification of $X\to Z$. More precisely, $Y\to M_Y$ is represented
	by subdivisions.
\end{lemma}

\begin{proof}
	By construction, the canonical morphism $M_Y\to X$ (given by sending $\sigma\to M_Y$ to the underlying
	morphism $\sigma\to X$) is a monomorphism. It hence remains to show that $Y\to M_Y$ is represented by
	subdivisions. First, observe that the construction of $M_Y$ is invariant under taking subdivisions of
	$Y$, so we may assume that $Y$ is a partial tropical compactification.
	Let $\sigma\to M_Y$ be a morphism and
	$$Y_{\sigma}:=Y\times_{X}\sigma\cong Y\times_{M_Y}\sigma$$
	(the second isomorphism is due to the fact that $M_Y\to X$ is a monomorphism). By assumption on $Y\to X$, the stack $Y_{\sigma}$
	is a finite cone stack. Consider a subdivision $\Sigma\to \sigma$ such that $\Sigma\to X$ lifts to $\Sigma\to Y$.
	We get a lift $\Sigma\to Y_{\sigma}$ of $\Sigma\to \sigma$. Let $\Delta$ be the cone complex corresponding to
	$Y_{\sigma}$ by Lemma \ref{lem:mono_impl_fan}. Note that since $\Sigma$ is a subdivision of $\sigma$ and $\Sigma\to \sigma$ lifts
	to a morphism $\Sigma\to \Delta$, we must have that for every $\tau\in \Delta$, we have $N_{\tau}=N_{\sigma}\cap \Span(\tau)$
	and that for every $x\in \sigma$ there must exist a $\tau\in \Delta$ containing $x$. In other words, $\Delta\to \sigma$ is
	a subdivision, showing the claim.
\end{proof}

\begin{lemma}
	Let $Y$ be a generalized partial tropical compactification of $X\to Z$. Then $M_Y$ is a minimal generalized partial tropical compactification.
\end{lemma}

\begin{proof}
	Let $Y_1\to X$ be a morphism and let $Y_2\to Y_1$ be a subdivision
	with a morphism $Y_2\to M_Y$ over $X$.
	We must construct a functorial morphism $\Hom(\sigma,Y_1)\to \Hom(\sigma,M_Y)$
	for any $\sigma\in \RPC$. For this, let $\sigma\to Y_1$ be an arbitrary morphism. Then $\Sigma:=\sigma\times_{Y_1}Y_2$
	is a subdivision of $\sigma$ with a morphism $\Sigma\to Y_2$. By the preceding lemma, $\Sigma\times_{M_Y}Y$
	is a subdivision of $\Sigma$ with a map to $Y$ lifting $\Sigma\to X$. We thus obtain a well-defined morphism sending
	$\sigma\to Y_1$ to the composition $\sigma\to Y_1\to X$, viewed
	as a morphism $\sigma\to M_Y$. This construction is functorial and we conclude.
\end{proof}

\begin{theorem}
	\label{thm:pre_part_comp_descr}
	Every birational equivalence class of generalized partial tropical compactifications of $X\to Z$
	contains a unique minimal element. For any partial tropical compactification $Y\to X$,
	the minimal element in the birational equivalence class of $Y$ is $M_Y$.
\end{theorem}

\begin{proof}
	Combining the two preceding lemmas, we know that $M_Y$ is minimal and that $Y\to M_Y$ is represented
	by subdivisions. For uniqueness, assume that $Y_1$ and $Y_2$ are minimal and that there exists a birational
	equivalence $Y_1\leftarrow Y_3\rightarrow Y_2$. In particular, $Y_3\to Y_i$ is a subdivision for $i\in \{1,2\}$.
	By minimality, the maps $Y_3\to Y_i$ descend to unique maps $Y_1\to Y_2$ and $Y_2\to Y_1$. By uniqueness, the compositions
	of these morphisms are the identity maps, so $Y_1\cong Y_2$.
\end{proof}

From the theorem we conclude that to obtain a complete birational classification of partial tropical compactifications
it suffices to classify minimal generalized partial tropical compactifications of the form $M_Y$. For this,
we consider the combinatorial objects below. To lighten the notation, we introduce the following conventions:

\begin{definition}
	Given an RPC $\sigma$ and a lattice $N$, we write $\sigma\subseteq N$ if $\sigma\subseteq N_{\mathbb{R}}$.
	If $A:=\Hom(\underline{M},\mathbb{G}_{m,\trop})^{(M)}/Q_{\hom}(\underline{M})$ is a tropical abelian variety over an RPC $\sigma$,
	we write $\tau\subseteq (N_{\sigma}\times N)^{(M)}$ if $\tau$ is contained in the convex hull of
	$(N_{\sigma}\times N)^{(M)}\subseteq N_{\sigma,\mathbb{R}}\times N_{\mathbb{R}}$.
\end{definition}

\begin{definition}
	\label{def:stacky_fan}
	Let $A:=\Hom(\underline{M},\mathbb{G}_{m,\trop})^{(M)}/Q_{\hom}(\underline{M})$ be a tropical abelian variety
	over an RPC $\sigma$. Assume moreover we are given a second lattice $M'$.

	A stacky fan is an (in general infinite) set $\Delta$ of rational polyhedral cones $\tau\subseteq (N_{\sigma}\times N)^{(M)}\times N'$,
	equipped with finite-index sublattices $\tilde{N}_{\tau}\subseteq ((N_{\sigma}\times N)\times N')\cap \Span(\tau)$
	satisfying the following conditions:
		\begin{enumerate}
			\item Any two elements of $\Delta$ intersect along a common face.
			\item For any $\tau,\tau'\in \Delta$ we have $\tilde{N}_{\tau}\cap \Span(\tau\cap \tau')=\tilde{N}_{\tau'}\cap \Span(\tau\cap \tau')$.
			\item For any $\tau\in \Delta$ and any face $\tau'$ of $\tau$ we have $\tau'\in \Delta$.
			\item For any $\tau\in \Delta$ and $m\in M$, we have $T_m(\tau)\in \Delta$ equipped with the lattices $T_m(\tilde{N}_{\tau})$ (where $T_m$ acts trivially on the $N'$ factor).
			\item For any $\tau\in \Delta$ and $m\in M$, we have that $\tau\cap T_m(\tau)$ is (pointwise) fixed by $T_m$.
			\item There are only finitely many equivalence classes of cones in $\Delta$, where $\tau_1,\tau_2\in \Delta$ are equivalent if and only if $\tau_2=T_m(\tau_1)$ for some $m\in M$.
			\item The cone $\sigma\times\{0\}\times\{0\}\subseteq N_{\sigma}\times N\times N'$ with lattice
				$\tilde{N}_{\sigma\times \{0\}\times\{0\}}:=(N_{\sigma}\times N\times N')\cap \Span(\sigma\times \{0\}\times\{0\})$ is an element of $\Delta$.
		\end{enumerate}
		We say that $\Delta$ is complete if $(N_{\sigma}\times N)^{(M)}\times N'\subseteq \bigcup_{\tau\in \Delta}^{}\tau$.
\end{definition}

\begin{definition}
	Given a stacky fan $\Delta$, we define a sheaf $\Sigma_{\Delta}$ on $\RPC/\sigma$ by letting $\Hom(\tau,\Sigma_{\Delta})$
	be the set of morphisms $\tau\to A\times T_{N'}^{\trop}$
	represented by $[\phi]\in \Hom(\tau,(N_{\sigma}\times N)^{(M)}\times N')/M$ (see Proposition \ref{prop:alternative_descr_trop_torus}, the action on $N'$ is trivial)
	such that any lift $\phi\in \Hom(\tau,(N_{\sigma}\times N)^{(M)}\times N')$ factors through some $(\tau',\tilde{N}_{\tau'})\in \Delta$.
\end{definition}

\begin{lemma}
	\label{lem:complete_iff_cpt}
	$\Delta$ is complete if and only if $\Sigma_{\Delta}\to A\times T_{N'}^{\trop}$ is surjective.
\end{lemma}

\begin{proof}
	Assume first that $\Delta$ is not complete. Then there exists some $x\in (N_{\sigma}\times N)^{(M)}\times N'$
	which does not lie in any cone of $\Delta$. Consider the map $\mathcal{A}^{1}:=(\mathbb{R}_{\ge 0},\mathbb{Z})\to A\times T_{N'}^{\trop}$
	induced by the morphism which sends $1\in \mathbb{R}_{\ge 0}$ to $x$ via the isomorphism in Proposition \ref{prop:alternative_descr_trop_torus}.
	Then $\Sigma_{\Delta}\times_{A\times T_{N'}^{\trop}}\mathcal{A}^{1}$ is the inclusion of the zero-cone and hence not surjective.
	Conversely, if $\Delta$ is complete, then observe that by
	Proposition \ref{prop:alternative_descr_trop_torus} any morphism $\tau\to A\times T_{N'}^{\trop}$ whose composition with the projection onto the first factor has a lift
	$\tau\to (N_{\tau}\times N)\times N'$ has image in $(N_{\tau}\times N)^{(M)}\times N'$. Hence its image is covered by
	cones pulled back from $\Delta$, showing that $\tau\times_{A\times T_{N'}^{\trop}}\Sigma_{\Delta}\to \tau$ is surjective.
\end{proof}

\begin{proposition}
	\label{prop:cone_cx_iff_fan}
	In the above situation, the following statements hold:
\begin{enumerate}
	\item $\Sigma_{\Delta}$ is representable by a finite cone complex.
	\item If $\tilde{N}_{\tau'}=\Span(\tau')\cap (N_{\sigma}\times N \times N')$ for all $\tau'\in \Delta$ and $\Delta$ is complete, then $\Sigma_{\Delta}\to A\times T_{N'}^{\trop}$ is a subdivision.
	\item The morphism $\Sigma_{\Delta}\to A\times T_{N'}^{\trop}$ is a partial tropical compactification of $A\times T_{N'}^{\trop}\to \sigma$.
	\item Every partial tropical compactification $\Sigma\to A\times T_{N'}^{\trop}$ with $\Sigma$ a finite cone complex is isomorphic to $\Sigma_{\Delta}$ for a unique stacky fan $\Delta$.
\end{enumerate}
\end{proposition}

For the proof, we need the following lemma:

\begin{lemma}
	\label{lem:diag_repr}
	The diagonals $T_{N'}^{\trop}\to T_{N'}^{\trop}\times T_{N'}^{\trop}$ and
	$A\to A\times_{\sigma}A$ are representable by finite cone complexes.
\end{lemma}

\begin{proof}
	Step 1 (Show representability of the diagonal of $T_{N'}^{\trop}$): Given two morphisms
	$\phi_1,\phi_2:\tau\to T_{N'}^{\trop}$, the fiber product $T_{N'}^{\trop}\times_{T_{N'}^{\trop}\times T_{N'}^{\trop}}\tau$
	is represented by the locus $\{x\in \tau|\phi_1(x)=\phi_2(x)\}$. This
	is given as the intersection of the images of $\tau$ under the maps
	$\tau\to \tau\times N_{\mathbb{R}}'$, where the map on the first factor is the identity
	and the map on the second factor is $\phi_i$. Since an intersection of rational
	polyhedral cones is a rational polyhedral cone, we conclude.

	Step 2 (Show representability of the diagonal of $A$): Let $0:\sigma\to A$ be the zero-section and consider the morphism
	$A\times_{\sigma}A\to A$ which sends $(x,y)$ to $x-y$. The diagonal
	morphism is the base-change of $0:\sigma\to A$ along this morphism, so
	it suffices to show that $\sigma\to A$ is representable by finite
	cone complexes. Let $\tau\to A$ be a morphism from an RPC $\tau$
	over $\sigma$. Fix a lift $\psi:\tau\to (N_{\sigma}\times N)^{(M)}$
	of $\tau\to A$ (see Proposition \ref{prop:alternative_descr_trop_torus}).

	Step 2.1 (Describe the fiber product $\tau\times_{A}\sigma$):
	A morphism $\tau'\to \tau\times_A \sigma$ is equivalent to the
	datum of a morphism $\tau'\to \tau$ such that the composition $\tau'\to \tau\xrightarrow{\psi} (N_{\sigma}\times N)^{(M)}$
	is defined by $\tau'\to N_{\sigma,\mathbb{R}}\times N_{\mathbb{R}},\ x\mapsto (f(x),Q_{\tau',\hom,N}(m)(x))$
	for some $m\in M$,
	where $f:N_{\tau',\mathbb{R}}\to N_{\sigma,\mathbb{R}}$ is the
	structure morphism corresponding to $\tau'\to \sigma$.
	Hence $\tau\times_A \sigma$ is a union of sets
	of the form:
	$$V(m):=\{x\in \tau|\psi_2(x)=Q_{\tau,\hom,N}(m)(x)\},$$
	with $V(m)$ and $V(m')$ glued along the common face given by the set of $x\in \tau$
	with $\psi_2(x)=Q_{\tau,\hom,N}(m)(x)=Q_{\tau,\hom,N}(m')(x)$ (this is common face by positive semi-definiteness
	of $Q_{\tau}$, as it is the locus where $Q_{\tau,\hom,N}(m-m')$ vanishes).
	Here, $\psi_2$ denotes the second component of $\psi$. Note that
	$$V(m)=(id\times\psi_2)^{-1}((id\times\psi_2)(\tau)\cap (id\times Q_{\tau,\hom,N}(m))(\tau))$$
	where $id\times \psi_2:{\tau}\to N_{\tau,\mathbb{R}}\times N_{\mathbb{R}}$
	is the identity on the first component and $\psi_2$ on the second
	(and similarly for $Q_{\tau,\hom,N}(m)$). Hence it is the preimage
	of an intersection of rational polyhedral cones and therefore
	again a rational polyhedral cone. Define an equivalence relation $\sim$ on $M$ by
	$m_1\sim m_2$ if $V(m_1)=V(m_2)$. The quotient $M/\sim$ is partially
	ordered by $m_1\le m_2$ if and only if $V(m_1)\subseteq V(m_2)$.
	We let $R$ be the set of maximal elements of $M/\sim$. Then:
	$$\tau\times_A \sigma\cong \bigcup_{[m]\in R}V(m)$$
	glued along common faces as above.

	It remains to show that $R$ is finite. This follows from
	\cite[Lemma 7.12 (2)]{LogarithmicAbeKajiwa2008} and \cite[5.2.7]{LogarithmicAbeKajiwa2015},
	see also \cite[9.4]{LogarithmicAbeKajiwa2008}. For the convenience
	of the reader, we spell out the argument below. We show the
	stronger claim that $M/{\sim}$ is finite.

	Step 2.2 (Reduce to the case where $Q_{\tau,\hom,N}$ is injective): Given
	$m\in M$, write $\tau_m$ for the smallest face of $\tau$ containing
	$V(m)$. Since $\tau$ has only finitely many faces, it suffices
	to prove finiteness of $\{[m]\in M/\sim|\tau_m=\tau'\}$ for some fixed
	face $\tau'$ of $\tau$. Replacing $\tau$ with $\tau'$, we
	may assume that $\tau=\tau'$. Note that if $m'\in \ker(Q_{\tau,\hom,N})$,
	then $V(m+m')=V(m)$ for all $m\in M$. Moreover, if $m\in \ker(Q_{\tau,\hom,N})$,
	and $m'\in M$, then
	$$Q_{\tau}(m,m')(n)=m'(Q_{\tau,\hom,N}(m)(n))=0$$
	for all
	$n\in N_{\tau}$, so $Q_{\tau}$ factors through $M/\ker(Q_{\tau,\hom,N})\times M/\ker(Q_{\tau,\hom,N})$.
	Therefore, we may replace
	$M$ with $M/\ker(Q_{\tau,\hom,N})$ and assume that $Q_{\tau,\hom,N}$
	is injective. We must show that there are only finitely many
	$m$ such that $V(m)$ intersects the relative interior of
	$\tau$.

	Step 2.3 (Reformulating the problem): Let $\tau_1$ be the image of $\tau$
	under $id\times \psi_2$ and $\tau_2$ be the image of
	$\tau$ under $id\times 0$. We must show that
	for all but finitely many $m$, the set $\tau_1\cap T_m\tau_2$
	does not intersect the preimage of $\tau^{\circ}$ under the
	projection $N_{\tau,\mathbb{R}}\times N_{\mathbb{R}}\to N_{\tau,\mathbb{R}}$ onto the first factor.
	Note that by definition of $(N_{\tau}\times N)^{(M)}$, both
	$\tau_1$ and $\tau_2$ lie in the image of
	$F:N_{\tau,\mathbb{R}}\times M_{\mathbb{R}}\to N_{\tau,\mathbb{R}}\times N_{\mathbb{R}}$
	defined by
	$$F(n,m):=(n,Q_{\tau,\hom,N}(m)(n))$$
	where we identify $Q_{\tau}$ with its $\mathbb{R}$-linear extension. If
	$n\in \tau^{\circ}$, then $m(Q_{\tau,\hom,N}(m)(n))=Q_{\tau}(m,m)(n)>0$
	by positive-definiteness, so $F|_{\tau^{\circ}\times M_{\mathbb{R}}}$
	is injective. We show that the intersection of the preimages of $\tau_1$ and $T_m\tau_2$ under
	$F$ does not intersect $\tau^{\circ}\times M_{\mathbb{R}}$ for all but finitely many $m$.
	Let $M$ act on $N_{\tau,\mathbb{R}}\times M_{\mathbb{R}}$
	by $m\cdot (x,y)=(x,y+m)$. Then
	$$T_m(F(x,y))=F(m\cdot (x,y)).$$
	The preimage of $\tau_2$ under $F$ is given as $\{(n,0)|n\in \tau\}$ up to further components which do not intersect the fibers over $\tau^{\circ}$.
	Similarly, the preimage of $T_m\tau_2$ is $\{(n,m)|n\in \tau\}$ up to components not intersecting the fibers over $\tau^{\circ}$.
	In particular, the projections of $F^{-1}(T_m \tau_2)\cap (\tau^{\circ}\times M_{\mathbb{R}})$ onto $M_{\mathbb{R}}$
	are bounded. It therefore suffices to show that the image of $F^{-1}(\tau_1)\cap (\tau^{\circ}\times M_{\mathbb{R}})$
	under the projection onto $M_{\mathbb{R}}$ is bounded.

	Step 2.4 (Summary of the strategy): We embed $\tau_1\cap (\tau^{\circ}\times N_{\mathbb{R}})$ into a compact topological space $C$,
	which itself is contained in a topological space $T$ which is homeomorphic to $[0,1]^{n-1}\times M_{\mathbb{R}}$
	in a way such that the restriction of this homeomorphism to the preimage of $\tau_1\cap (\tau^{\circ}\times N_{\mathbb{R}})$ under $F$ is
	equal to the restriction of $F$. The map $\pi_2\circ F^{-1}$ on $\tau_1\cap (\tau^{\circ}\times N_{\mathbb{R}})$ can thus be extended
	to the compact space and hence its image must be bounded. This relationship is visualized in the following diagram with the explicit
	spaces $C$ and $T$, as well as the maps between them explained later in the proof:
	$$\begin{tikzcd}[row sep=4.6em, column sep=3.8em, ampersand replacement=\&]
		F^{-1}\!\bigl(\tau_1 \cap (\tau^\circ \times N_{\mathbb{R}})\bigr)
		\arrow[rr, hook,"(f^{-1}\circ F)\times \pi_2"]
		\arrow[d, "F"']
		\&
		\& {[}0,1{]}^{n-1} \times M_{\mathbb{R}}
		\arrow[d, "G", "\cong"']
		\\
		\tau_1 \cap (\tau^\circ \times N_{\mathbb{R}})
		\arrow[r, hook, "f^{-1}"]
		\& C=[0,1]^{n-1}
		\arrow[r, hook,"\phi"]
		\& T=\im(G).
	\end{tikzcd}$$

	Step 2.5 (Reparametrize $\tau_1$ and construct the extension $\phi_{m'}$): We reparametrize $\tau_1\cap (\tau^{\circ}\times N_{\mathbb{R}})$
	as follows: Let $\{(x_i,n_i)\}_{i=1}^{n}\subseteq N_{\tau,\mathbb{R}}\times N_{\mathbb{R}}$
	be the set of ray generators of $\tau_1$. In particular, $x_i$ lies over a ray of $\tau$. Each element of $\tau_1$ is of the form:
	$$\left(\sum_{i=1}^{n}a_ix_i,\sum_{i=1}^{n}a_in_i\right)$$
	with $a_i\ge 0$.
	It suffices to prove boundedness after restricting to finitely many subcones of $\tau_1$ that cover $\tau_1$. After reordering the $x_i$,
	we may therefore assume that $a_1\ge a_2\ge\dots\ge a_n$. Moreover, after further subdivision, we may assume that $\tau$ is simplicial. Since the second
	factor of $F^{-1}$ is invariant under simultaneously scaling $n$ and $m$, we may restrict to the locus where
	$a_1=1$. Over $\tau^{\circ}$, we have $a_i\neq 0$ for all $i$.
	Therefore, we may reparametrize by $f:(0,1]^{n-1}\to \tau_1\cap (\tau^{\circ}\times N_{\mathbb{R}})$, defined by:
	$$f(a):=\left(\sum_{i=1}^{n}\left(\prod_{j=1}^{i-1}a_j\right)x_i,\sum_{i=1}^{n}\left(\prod_{j=1}^{i-1}a_j\right)n_i\right)=:(x_a,n_a).$$

	Fix $m'\in M_{\mathbb{R}}\setminus \{0\}$. We construct a morphism $\phi_{m'}':\tau_1\cap (\tau^{\circ}\times N_{\mathbb{R}})\to \mathbb{R}$
	which we will show extends to $C:=[0,1]^{n-1}$ under the embedding $f^{-1}$. These functions will be the coordinates of our reembedding
	into $T$. For each face $u$ of $\tau^{\vee}$, fix once and for all an element $s_{u}$ in the relative interior of $u$.
	Let $u_{m'}\subseteq \tau^{\vee}$ be the smallest face of $\tau^{\vee}$ containing $Q_{\tau}(m',m')$. Define
	$$\phi_{m'}'(x,n):=\frac{m'(n)}{s_{u_{m'}}(x)}.$$
	This is well-defined since $x$ is in the relative interior of $\tau$, so $s_{u_{m'}}(x)\neq 0$. We claim that under the
	embedding $f^{-1}$ of $\tau_1\cap (\tau^{\circ}\times N_{\mathbb{R}})$ this extends to a continuous map $\phi_{m'}:[0,1]^{n-1}\to \mathbb{R}$.
	For this, we let $a\in (0,1]^{n-1}$. Let $k$ be minimal such that $s_{u_{m'}}(x_k)\neq 0$. We claim that for all $i<k$ we
	have $m'(n_i)=0$. Indeed, note that since $\tau_1\subseteq (N_{\tau,\mathbb{R}}\times N_{\mathbb{R}})^{(M)}$, we can write
	$n_i=Q_{\tau,\hom,N}(m)(x_i)$ for some $m\in M\otimes \mathbb{R}$. Therefore, $m'(n_i)=Q_{\tau}(m,m')(x_i)$. Since $Q_{\tau}(m',m')(x_i)=0$ by assumption,
	Lemma \ref{lem:pos_definite_vanishing} shows that $m'(n_i)=0$. From this it follows that:
	$$\phi_{m'}'(f(a))=\frac{m'\left(\sum_{i=k}^{n}\left(\prod_{j=k}^{i-1}a_j\right)n_i\right)}{s_{u_{m'}}\left(\sum_{i=k}^{n}\left(\prod_{j=k}^{i-1}a_j\right)x_i\right)}.$$
	Note that:
	$$s_{u_{m'}}\left(\sum_{i=k}^{n}\left(\prod_{j=k}^{i-1}a_j\right)x_i\right)=s_{u_{m'}}(x_k)+s_{u_{m'}}\left(\sum_{i=k+1}^{n}\left(\prod_{j=k}^{i-1}a_j\right)x_i\right)\ge s_{u_{m'}}(x_k)>0,$$
	so the denominator is non-zero for all $a\in [0,1]^{n-1}$, giving an extension $\phi_{m'}$ as required.

	Step 2.6 (Construct the homeomorphism $G$): We define
	$\phi:[0,1]^{n-1}\to [0,1]^{n-1}\times\mathbb{R}^{M_{\mathbb{R}}\setminus \{0\}}$ by the identity on the first component and $\phi_{m'}$
	on the $m'$-th factor of the second component. It remains to construct a morphism $G:[0,1]^{n-1}\times M_{\mathbb{R}}\to [0,1]^{n-1}\times \mathbb{R}^{M_{\mathbb{R}}\setminus \{0\}}$
	which is a homeomorphism onto its image $T:=\im(G)$ and such that $G\circ ((f^{-1}\circ F)\times \pi_2)=\phi\circ f^{-1}\circ F$,
	where $\pi_2$ is the projection $N_{\tau,\mathbb{R}}\times M_{\mathbb{R}}\to M_{\mathbb{R}}$ onto the second factor.
	We define $G$ to be the identity on the first component and on the $m'$-th factor of the second component to be given by:
	$$G_{2,m'}(a,m):=\frac{Q_{\tau}(m',m)(x_a)}{s_{u_{m'}}(x_a)}.$$
	The same argument as above shows that $G_{2,m'}$ indeed extends to all of $[0,1]^{n-1}$ even in the cases when $s_{u_{m'}}(x_a)=0$.
	Unraveling the definitions shows that $G\circ ((f^{-1}\circ F)\times \pi_2)=\phi\circ f^{-1}\circ F$ as required. It remains to show that $G$ is a homeomorphism
	onto its image. We start by showing that it is injective. For this, observe that the second factor $G_2$ of $G$ is linear in $m$.
	So it suffices to show that if $m\in M_{\mathbb{R}}\setminus \{0\}$, then $G_2(a,m)\neq 0$. Assume that $m\neq 0$. Let $k$ be minimal such that
	$Q_{\tau}(m,m)(x_k)\neq 0$. Then
	$$G_{2,m}(a,m)=\frac{Q_{\tau}(m,m)\left(\sum_{i=k}^{n}\left(\prod_{j=k}^{i-1}a_j\right)x_i\right)}{s_{u_m}\left(\sum_{i=k}^{n}\left(\prod_{j=k}^{i-1}a_j\right)x_i\right)}.$$
	Note that $Q_{\tau}(m,m)\in \tau^{\vee}$ by assumption, so the numerator is greater or equal to $Q_{\tau}(m,m)(x_k)>0$. It
	follows that $G_{2,m}(a,m)\neq 0$. Finally, note that the image of $G$ is a vector bundle over $[0,1]^{n-1}$. Indeed, by the above, we know that for any basis $B$ of $M_{\mathbb{R}}$, at every point $a$ there
	exists a finite set $S$ of coordinates of $\mathbb{R}^{M_{\mathbb{R}}\setminus \{0\}}$ such that the images of the elements $m\in B$ in these coordinates are linearly independent. Since this is an open condition, the same holds for any point in some neighbourhood of $a$
	and hence we obtain a local chart of the image of $G$. Note that
	$G$ induces a map of vector bundles. Since every bijective map of vector bundles is a homeomorphism, it follows that $G$ is a homeomorphism onto
	its image.

	Step 2.7 (Conclusion): To summarize, we have constructed an extension of $\pi_2\circ F^{-1}|_{\tau_1\cap (\tau^{\circ}\times N_{\mathbb{R}})}\circ f$ to $[0,1]^{n-1}$
	given by $p_2\circ G^{-1}\circ \phi$, where $p_2:[0,1]^{n-1} \times M_{\mathbb{R}}\to M_{\mathbb{R}}$ is the projection
	onto the second factor. This is a continuous function and since $[0,1]^{n-1}$ is compact, the image
	must be compact, hence bounded and it follows that the image of
	$\pi_2\circ F^{-1}|_{\tau_1\cap (\tau^{\circ}\times N_{\mathbb{R}})}$ is bounded as required.
\end{proof}

\begin{proof}[Proof of Proposition \ref{prop:cone_cx_iff_fan}]
	Step 1 (Show that $\Sigma_{\Delta}$ is a finite cone complex): To lighten the notation, we assume that $M'=0$ in this proof. The proof goes through verbatim for the general case.

	Pick a stacky fan $\Delta$. We must show that $\Sigma_{\Delta}$ is representable by a finite cone complex.
	Let $\Sigma$ be the set of equivalence classes of cones in $\Delta$ modulo the equivalence relation in Definition
	\ref{def:stacky_fan}. For each $[\tau]\in \Sigma$ pick a representative $\tau$ and equip $\tau$ with the lattice $\tilde{N}_{\tau}$.
	We define face morphisms $[\tau]\to [\tau']$ to be the datum
	of an RPC $T_m\tau\in [\tau]$ and a morphism
	$$\tau\to T_m\tau\to \tau'$$
	in $\Delta$, where the first
	morphism is the canonical isomorphism and the second is a face inclusion (the element $m$ is not part of the data). This endows
	$\Sigma$ with the structure of a category fibered in setoids over $\RPC$ and it remains to show that $\Sigma$ is a finite cone complex with
	$\Sigma\cong \Sigma_{\Delta}$. The fact that $\Sigma$ has finitely many cones is clear by construction.
	To see that it is a cone complex, first observe that
	if $\tau\in \Sigma$, then for every face $\tau'$ of $\tau$ we have $[\tau']\in \Sigma$ by condition (3) in the definition of a stacky fan. Hence it remains
	to show that for $[\tau_1],[\tau_2]\in \Sigma$, there exists at most one morphism $[\tau_1]\to [\tau_2]$ in $\Sigma$. Assume we are
	given two morphisms $[\tau_1]\to [\tau_2]$ defined by $\tau_1\to T_{m_1}\tau_1\subseteq\tau_2$ and $\tau_1\to T_{m_2}\tau_1\subseteq\tau_2$.
	Then
	$$T_{m_1}\tau_1=T_{m_1-m_2}T_{m_2}\tau_1\subseteq T_{m_1-m_2}\tau_2.$$
	Hence $T_{m_1}\tau_1\subseteq \tau_2\cap T_{m_1-m_2}\tau_2$.
	By condition (5) in the definition of a stacky fan, the element $m_1-m_2\in M$ fixes $\tau_2\cap T_{m_1-m_2}\tau_2$ pointwise,
	so $T_{m_1}\tau_1=T_{m_2}\tau_1$ and it follows that the two morphisms are equal.

	Step 1.1 (Show $\Sigma_{\Delta}\cong \Sigma$): The datum of a morphism $\tau\to \Sigma_{\Delta}$ for $\tau\in \RPC$
	is equivalent to the datum of a morphism $\tau\to N_{\sigma,\mathbb{R}}\times N_{\mathbb{R}}$ such that the image of $\tau$ is contained
	in some $\tau_0\in \Delta$, modulo the action of $M$ on $\Hom(\tau,(N_{\sigma}\times N)^{(M)})$. Translating $\tau\to N_{\sigma}\times N$
	if necessary, we may assume that $\tau_0$ is the chosen representative of $[\tau_0]\in \Sigma$. This defines a morphism $\tau\to \Sigma$,
	independent of the choice of translation of $\tau_0$ by condition (5) in the definition of a stacky fan. Conversely, given $\tau\to \Sigma$ factoring through
	$\tau_0\in \Sigma$, composing with $\tau\to \tau_0\to N_{\sigma,\mathbb{R}}\times N_{\mathbb{R}}$ gives a morphism $\tau\to \Sigma_{\Delta}$. It is clear from
	construction that these two assignments are inverses to one another and we conclude.

	Step 2 (Show $\Sigma_{\Delta}\to A$ is a partial tropical compactification): From Lemma \ref{lem:diag_repr} and the representability of $\Sigma_{\Delta}$, it follows that $\Sigma_{\Delta}\to A$
	is representable by finite cone complexes. We have a section $\sigma\to \Sigma_{\Delta}$ given by the cone $\sigma\times\{0\}\in \Delta$,
	which lies over the zero-section $\sigma\to A$. It is clear from the construction of $\Sigma_{\Delta}$ that $\Sigma_{\Delta}\to A$
	is a monomorphism, hence $\Sigma_{\Delta}\to A$ is a partial tropical compactification of $A\to \sigma$.

	If each $\tilde{N}_{\tau}$ is given by $\Span(\tau)\cap (N_{\sigma}\times N)$ and $\Delta$ is complete, then it follows from the
	above construction that $\Sigma_{\Delta}\to A$ is a subdivision.

	Step 3 (Show that every representable partial tropical compactification is a stacky fan): We show that any partial tropical compactification $\Sigma\to A$ with $\Sigma$ representable by a finite cone complex is isomorphic
	to $\Sigma_{\Delta}$ for a unique stacky fan. For uniqueness, note that if $\Sigma\cong \Sigma_{\Delta}$, then $\Delta$
	can be recovered as the set of cones $\{T_m\tau|\tau\in \Sigma\}$ together with the finite-index sublattices
	$T_m N_{\tau}$. For existence, we let $\Sigma$ be arbitrary. Let $\tau\in \Sigma$. Then $\tau\to \Sigma\to A$ is a monomorphism,
	so choosing a lift, we may view $\tau\subseteq(N_{\sigma}\times N)^{(M)}$. Let $\Delta:=\{T_m\tau|\tau\in \Sigma\}$
	and for each $T_m\tau$ equip it with the sublattice $\tilde{N}_{T_m \tau}=T_m N_{\tau}$. This is well-defined by assumption, since any
	$m$ fixing $\tau$ acts trivially on the span of $\tau$ by assumption. We show that $\Delta$ is a stacky
	fan. First, it is clear that $T_m \tau\subseteq (N_{\sigma}\times N)^{(M)}$ since $\tau$ is. Moreover, if $\tilde{\tau}$ is a
	face of $T_m \tau$, then $\tilde{\tau}=T_m \tau'$ for some face $\tau'$ of $\tau$, so $\Delta$ is closed under taking faces. Next,
	since $\Sigma$ is finite, $\Delta$ has only finitely many cones up to equivalence. Let $\sigma\to \Sigma$ be a morphism lifting
	the zero section, then the composition $\sigma\to \Sigma\to A$ lies in the equivalence class of the map $\sigma\to \sigma\times \{0\}\subseteq (N_{\sigma}\times N)^{(M)}$,
	so $\sigma\times\{0\}\in \Delta$. Next, we show that $T_{m}\tau\cap \tau$ is pointwise fixed by $m$ for any $\tau\in \Delta$.
	For this, let
	$$x\in T_m\tau\cap\tau\cap \tilde{N}_{\tau}\cap \tilde{N}_{T_m\tau}$$
	and consider the morphism $\mathcal{A}^{1}:=(\mathbb{R}_{\ge 0},\mathbb{Z})\to \tau$
	which sends 1 to $x$. The morphisms $\mathcal{A}^{1}\to T_m\tau$ and $\mathcal{A}^{1}\to \tau$ define lifts of this morphism
	to $\Sigma$. Since $\Sigma\to A$ is a monomorphism, it follows that the two lifts agree, so $x$ must be fixed by the action of $T_m$.

	Step 3.1 (Show that cones of $\Delta$ intersect along common faces): We show that $\tau:=T_{m_1}\tau_1\cap T_{m_2}\tau_2$ is a common face of $T_{m_1}\tau_1$
	and $T_{m_2}\tau_2$ and that the sublattices agree on this common face. We prove this by induction on the maximum $d$ of $\dim(\tau_1)$
	and $\dim(\tau_2)$, the case $d=0$ being clear. Equip $\tau$ with the lattice $T_{m_1}N_{\tau_1}\cap T_{m_2}N_{\tau_2}$.
	Then the compositions $\tau\to T_{m_{i}}\tau_{i}\cong \tau_i$ define two lifts of $\tau\to A$ to $\tau\to \Sigma$. Since $\Sigma\to A$
	is a monomorphism, these two lifts must agree. It follows that the image of $\tau$ in $\tau_1$ is glued to the image of $\tau$ in $\tau_2$
	in $\Sigma$. Since the gluing morphisms in $\Sigma$ are face morphisms, the image of $\tau$ is contained in a common face
	$\tau_0\subseteq \tau_i$ of the cones $\tau_i$. It follows that $\tau=T_{m_1}\tau_0\cap T_{m_2}\tau_0$. If $\dim(\tau_0)<d$, we conclude
	by the induction hypothesis. Otherwise, the image of $\tau$ must intersect the interior of the cone of larger dimension which forces
	$\tau_0=\tau_1=\tau_2$. So we must show that $T_{m_1}\tau_0\cap T_{m_2}\tau_0$ is a face of $\tau_0$. Acting by $T_{-m_2}$,
	we reduce to showing this for $T_{m}\tau_0\cap \tau_0$. In this case, by the previous property, we know that $\tau$ is pointwise
	fixed by $m$. Let $(x,n)\in N_{\sigma}\times N$ be a point lying in $T_m\tau_0\cap \tau_0$.
	So $(x,n+Q_{\hom,N}(m)(x))=(x,n)$, i.e., $Q_{\hom,N}(m)(x)=0$. By Lemma \ref{lem:pos_definite_vanishing}, this is equivalent
	to $Q(m,m)(x)=0$. Since $Q(m,m)\in \sigma^{\vee}\subseteq \tau_0^{\vee}$, it follows that the locus $\{(x,n)|Q(m,m)(x)=0\}$
	is a face of $\tau_0$ (resp. $T_{m}\tau_0$). Hence
	$$T_m\tau_0\cap \tau_0=\{(x,n)\in \tau_0|Q(m,m)(x)=0\}$$
	is a common face
	of $\tau_0$ and $T_m\tau_0$. Moreover, we have $Q(m,m)(x)=0$ for all $x$ in the span of $\tau$, so it follows that
	$\Span(\tau)\cap \tilde{N}_{\tau_0}=\Span(\tau)\cap \tilde{N}_{T_m \tau_0}$ finishing the proof that $\Delta$ is a stacky fan.

	Step 3.2 (Show $\Sigma\cong \Sigma_{\Delta}$): By construction, we have a monomorphism $\Sigma\to \Sigma_{\Delta}$
	which sends $\tau\to \Sigma$ to $\tau\to \tau'\subseteq N_{\sigma,\mathbb{R}}\times N_{\mathbb{R}}$ for any cone $\tau'\in \Sigma$ containing the
	image of $\tau$. Conversely, given a morphism $\tau\to \Sigma_{\Delta}$, let $T_m\tau'\in \Delta$ be a cone such
	that $\tau\to N_{\sigma,\mathbb{R}}\times N_{\mathbb{R}}$ factors through $T_m\tau'$, with $\tau'\in \Sigma$. We obtain a morphism $\tau\to \tau'$
	by composing with the canonical isomorphism $T_m\tau'\cong \tau'$. The morphism $\tau\to \tau'\to \Sigma$ is independent of the
	choice of $m$ and $\tau'$ and defines an inverse $\Sigma_{\Delta}\to \Sigma$ to the above morphism.
\end{proof}

The above result is only useful if stacky fans exist:

\begin{proposition}
	\label{prop:stacky_fan_exists}
	Let $X\to \Sigma$ be a split family of tropical semiabelian varieties over a finite cone stack $\Sigma$.
	Then there exists a tropical compactification $Y\to X$ such that $Y\to X$ is representable
	by subdivisions.
\end{proposition}

\begin{proof}
	Write $X=A\times T$ for a tropical abelian variety $A$ and a tropical torus $T$.
	We construct a tropical compactification as a product of a tropical compactification of $A$
	and a tropical compactification of $T$.
	For $A$, the existence of a tropical compactification $Y_A\to A$ representable by subdivisions
	follows from \cite[Theorem 1.11]{kajiwara2018logarithmic}.

	For $T$, note that on each cone in $\Sigma$, we have
	$T\cong T_{\mathbb{Z}^{r}}^{\trop}$ for some $r\ge 0$ fixed. We inductively
	construct a complete fan (with sublattice equal to the ambient lattice) in $T|_{\Sigma_k}$, where $\Sigma_k\subseteq \Sigma$
	is the full substack spanned by the cones in $\Sigma$ of dimension
	at most $k$. Let $\sigma\in \Sigma_k$ be $k$-dimensional. By the induction hypothesis, we may
	assume that we have constructed a complete fan over all proper faces of $\sigma$ and we must
	extend this to an $\Aut_{\Sigma_k}(\sigma)$-invariant subdivision of $T_{\mathbb{Z}^{r}}\times \sigma$.
	If $k>0$, then we define a subdivision of $T_{\mathbb{Z}^{r}}\times \sigma$ by letting the cones in the
	subdivision be either:
	\begin{enumerate}
		\item Cones $\tau$ in the subdivision of $T_{\mathbb{Z}^{r}}\times \sigma'$ for some face $\sigma'$ of $\sigma$, or
		\item cones of the form $\Conv(\rho,\tau)$, where $\rho$ is the ray through the barycenter of $\sigma$ and $\tau$ is as in point (1).
	\end{enumerate}
	This is indeed a subdivision since for any point $x\in T_{\mathbb{Z}^{r}}\times \sigma$, the ray starting at some non-zero
	point in $\rho$ and passing through $x$ intersects the preimage of the boundary of $\sigma$ in a point $y$ and $x$ is hence contained
	in the cone $\Conv(\rho,\tau)$, where $\tau$ is a cone containing $y$. The subdivision is moreover $\Aut(\sigma)$-invariant,
	since $\Aut(\sigma)$ fixes the barycenter of $\sigma$ and the subdivisions on the boundary of $\sigma$ are $\Aut(\sigma)$-invariant
	by the inductive hypothesis (i.e., since they glue to subdivisions of $\Sigma_k$). It remains to consider the case when $\dim(\sigma)=0$,
	i.e., $\sigma$ is a point modulo a finite group $G$. In this case, let $S$ be the set consisting of
	$g\cdot e_i$ for $g\in G$ and $e_i$ the $i$-th standard basis vector in $\mathbb{Z}^{r}$. We define a subdivision of $T_{\mathbb{Z}^{r}}$
	by subdividing along the set of hyperplanes orthogonal to some $x\in S$. This is a $G$-invariant subdivision of $T_{\mathbb{Z}^{r}}$
	and a complete fan, hence we conclude.
\end{proof}

\begin{definition}
	\label{def:minimal_fan}
	Consider a tropical abelian variety $A:=\Hom(\underline{M},\mathbb{G}_{m,\trop})^{(M)}/Q_{\hom}(\underline{M})$
	over an RPC $\sigma$. Assume moreover we are given a second lattice $M'$.

	A minimal fan is a subset $S\subseteq N_{\sigma}\times N\times N'$
	such that there exists a stacky fan $\Delta$ such that $S=\bigcup_{\tau\in \Delta}\tau\cap \tilde{N}_{\tau}$.
	We say that $S$ is complete if for every $n\in ((N_{\sigma}\times N)^{(M)}\times N')\cap (\sigma\times N \times N')$ there exists $a\in \mathbb{N}_{>0}$
	such that $an\in S$.
\end{definition}

\begin{example}
	Let $\Delta$ be a stacky fan. Then $S=\bigcup_{\tau\in \Delta}\tau\cap \tilde{N}_{\tau}$
	is a minimal fan. $\Delta$ is complete if and only if $S$ is complete.
\end{example}

\begin{definition}
	Given a minimal fan, we define a stack $\Sigma_{S}$ over $\RPC$ by letting $\Hom(\tau,\Sigma_{S})$ be the
	set of morphisms $\tau\to A\times T_{N'}^{\trop}$ such that the image of $\tau\cap N_{\tau}$ in $N_{\sigma}\times N\times N'$
	lies in $S$ for every lift $\tau\to N_{\sigma,\mathbb{R}}\times N_{\mathbb{R}}\times N_{\mathbb{R}}'$
	of $\tau\to A\times T_{N'}^{\trop}$.
\end{definition}

\begin{proposition}
	Let $S$ be a minimal fan and $\Delta$ be as in the definition of a minimal fan. Then
	$\Sigma_S\cong M_{\Sigma_{\Delta}}$. In particular, $\Sigma_S$ is minimal.
\end{proposition}

\begin{proof}
	First note that for each $\tau\in \Delta$, we have $\tau\cap \tilde{N}_{\tau}\subseteq S$,
	so we get a morphism $\Sigma_{\Delta}\to \Sigma_S$. We first show that this morphism is represented
	by subdivisions: Let $\tau\to \Sigma_S$ be an arbitrary morphism. Since $\Sigma_S\to A\times T_{N'}^{\trop}$
	is a monomorphism, we have
	$$\tau\times_{\Sigma_S}\Sigma_{\Delta}\cong \tau\times_{A\times T_{N'}^{\trop}} \Sigma_{\Delta}.$$
	In particular, the latter is representable by finite cone complexes by our previous results.
	Pick a lift $\tau\to N_{\sigma}\times N\times N'$. By
	construction, this cone complex is the union of the preimages of $(\tau',\tilde{N}_{\tau'})$
	under this lift, for $\tau'\in \Delta$. Since $\tau\to A\times T_{N'}^{\trop}$ factors through
	$\Sigma_S$, the preimage of $\tau'\cap \tilde{N}_{\tau'}$ is equal to the restriction
	of $\tau\cap N_{\tau}$ to the preimage of $\tau'$. Since $\tau'$ spans $\tilde{N}_{\tau'}$, it follows that
	the preimage of $\tilde{N}_{\tau'}$ is the intersection of $N_{\tau}$ with the span of the preimage of
	$\tau'$. It remains to show that the preimages cover all of $\tau$. Indeed,
	if not, we find some $x\in \tau\cap N_{\tau}$ not contained in the preimage (since the preimage
	is a finite union of rational polyhedral cones). But $x$ must be mapped to some element of $S$ which is contained
	in some $\tau'\in \Delta$, a contradiction. It follows that $\tau\times_{\Sigma_S}\Sigma_{\Delta}$
	is a subdivision, which hence defines a map $\tau\to M_{\Sigma_{\Delta}}$. This construction is
	functorial, so we get a morphism $\Sigma_{S}\to M_{\Sigma_{\Delta}}$. To construct an inverse,
	assume we are given a morphism $\tau\to M_{\Sigma_{\Delta}}$ corresponding to a subdivision
	$\Sigma\to \tau$ and a morphism $\Sigma\to \Sigma_{\Delta}$. For each cone $\tau'\in\Sigma_{\Delta}$
	and any lift
	$$\tau'\to N_{\sigma,\mathbb{R}}\times N_{\mathbb{R}}\times N_{\mathbb{R}}',$$
	we know that $\tau'\cap N_{\tau'}\subseteq S$.
	Since the $\tau'$ cover all of $\tau$ and the lattices $N_{\tau'}$ agree with the restriction of $N_{\tau}$,
	it follows that $\tau\cap N_{\tau}$ maps to $S$ for any lift of $\tau\to A\times T_{N'}^{\trop}$.
	Hence we obtain a map $\tau\to \Sigma_S$. This construction is clearly inverse to the previous
	one and hence the claim follows.
\end{proof}

\begin{corollary}
	\label{cor:combinat_descr_min}
	In the above situation the following statements hold:
\begin{enumerate}
	\item Any minimal generalized partial tropical compactification is isomorphic to $\Sigma_{S}$ for some minimal fan $S$.
	\item If $\Sigma_{S}\cong \Sigma_{S'}$, then $S=S'$.
\end{enumerate}
\end{corollary}

\begin{proof}
	We start with point (1).
	By Theorem \ref{thm:pre_part_comp_descr}, we know that any minimal generalized partial tropical compactification is isomorphic
	to some $M_Y$ for $Y$ a generalized partial tropical compactification. Since the formation of $M_Y$ is invariant under subdivisions,
	we may assume that $Y$ is a partial tropical compactification. Let $\Delta$ be as in Proposition \ref{prop:stacky_fan_exists},
	so $\Sigma_{\Delta}\to A\times T_{N'}^{\trop}$ is a subdivision. Then $Y':=\Sigma_{\Delta}\times_A Y\to Y$
	is a subdivision. We claim that $Y'$ is a finite cone complex. Indeed, the projection $Y'\to \Sigma_{\Delta}$
	is representable by finite cone stacks, hence so is $Y'$. Since $Y'\to A$ is a monomorphism and $A$ is fibered in setoids, it follows
	that $Y'$ is in fact representable by a finite cone space. Let $\tau_1,\tau_2\in Y'$ be cones. We must show that there is
	at most one morphism $\tau_1\to \tau_2$. Assume for a contradiction that there are two morphisms $\phi_i:\tau_1\to \tau_2$. The
	compositions of $\phi_i$ with the composition $\tau_2\to Y'\to \Sigma_{\Delta}$ must be equal since $\Sigma_{\Delta}$ is a finite
	cone complex. But $\tau_2\to Y'\to \Sigma_{\Delta}$ is a monomorphism, so $\phi_1=\phi_2$ and hence $Y'$ is a cone complex.
	By Proposition \ref{prop:cone_cx_iff_fan}, we have $Y'\cong \Sigma_{\Delta}$ for some stacky fan $\Delta$ and we
	see that $M_Y=M_{Y'}=\Sigma_{S}$ for
	$$S=\bigcup_{\tau\in \Delta}\tau\cap \tilde{N}_{\tau}.$$

	For point (2), note that if $\Sigma_{\Delta}\to \Sigma_S$ is any subdivision, then
	$S=\bigcup_{\tau\in \Delta}\tau\cap \tilde{N}_{\tau}$ and hence $S$ is uniquely determined by $\Sigma_S$.
\end{proof}

The above corollary gives a concrete combinatorial description for the set of birational equivalence classes of generalized
partial tropical compactifications over a rational polyhedral cone $\sigma$. We complete the classification in general
by means of the following result:

\begin{proposition}
	\label{prop:minimal_iff_minimal_on_cones}
	Let $X\to Z$ be a split family of tropical semiabelian varieties and let $Y\to X$ be a generalized partial tropical compactification.
	Then $Y$ is minimal if and only if for every morphism $\sigma\to Z$ from an RPC $\sigma$, the base-change
	$Y\times_{Z}\sigma\to X\times_Z \sigma$ is minimal.
\end{proposition}

\begin{proof}
	Assume first that $Y$ is minimal and let $Y_1$ be birationally equivalent to $Y\times_Z \sigma$. Let $Y_2$ be
	a subdivision of $Y_1$ such that there exists a morphism $Y_2\to Y\times_Z \sigma$. Composing with the projection
	onto the first factor gives a morphism $Y_2\to Y$ which descends to a morphism $Y_1\to Y$ by minimality of $Y$.
	Together with the structure morphism $Y_1\to \sigma$, this defines a morphism $Y_1\to Y\times_Z \sigma$ showing
	that $Y\times_Z \sigma$ is minimal.

	Conversely, assume that
	$$Y\times_Z\sigma\to X\times_Z \sigma$$
	is minimal for each $\sigma\in \RPC$. Let
	$Y_1\to X$ be a morphism and $Y_2\to Y_1$ a subdivision such that there exists a morphism
	$Y_2\to Y$. We must construct functorial maps $\Hom(\sigma,Y_1)\to \Hom(\sigma,Y)$. Let $\sigma\to Y_1$
	be a morphism and write $\Sigma:=\sigma\times_{Y_1}Y_2$. This is a subdivision of $\sigma$ by assumption.
	The projection onto the first factor and $\Sigma\to Y_2\to Y$ define a morphism $\Sigma\to Y\times_{Z}\sigma$
	over $\sigma\to X\times_Z \sigma$. By minimality of $Y\times_Z\sigma$ we obtain a morphism $\sigma\to Y\times_Z \sigma$.
	Projecting onto the first factor we obtain a morphism $\sigma\to Y$ as required. Functoriality is clear,
	so we obtain a morphism $Y_1\to Y$. To see that this morphism is unique, note that any morphism $\sigma\to Y$
	can be uniquely factored as $\sigma\to Y\times_Z \sigma\to Y$ where the second morphism is the projection onto
	the first factor. The first morphism is unique by the minimality of $Y\times_Z \sigma$ and hence we conclude.
\end{proof}

This result means that for a split family of tropical semiabelian varieties over a cone complex, the datum of a minimal
generalized partial tropical compactification is equivalent to the datum of a minimal fan over each cone compatible with face inclusions.
We end this section with the following result for complete minimal fans.

\begin{proposition}
	\label{prop:complete_iff_bir_to_cpt}
	A minimal fan is complete if and only $\Sigma_S$ is birationally equivalent to a tropical compactification of $X\to Z$.
\end{proposition}

\begin{proof}
	Let $S$ be a minimal fan and write $S=\bigcup_{\tau\in \Delta}\tau\cap \tilde{N}_{\tau}$ for some stacky
	fan $\Delta$. If $\Sigma_S$ is birationally equivalent to a tropical compactification of $X\to Z$ then
	by Lemma \ref{lem:cpt_bir_inv} we have that $\Sigma_{\Delta}$ is a tropical compactification and hence $\Delta$
	is complete by Lemma \ref{lem:complete_iff_cpt}. In particular, $S$ is complete. Conversely, if $S$
	is complete, then $\Delta$ must be complete, so by Lemma \ref{lem:complete_iff_cpt} we have that
	$\Sigma_{\Delta}$ is a tropical compactification. Since
	$\Sigma_{\Delta}\to \Sigma_{S}$ is a subdivision, it follows
	that $\Sigma_S$ is birationally equivalent to the tropical
	compactification $\Sigma_{\Delta}$.
\end{proof}

\section{Translation to the geometric setting}
\label{sec:translate_geom}

In the following sections we will apply the results of the previous section
to solve the birational classification problem for specific classes
of log semiabelian varieties. For this, we must translate the combinatorial results to
the geometric setting. We fix a log semiabelian variety $X\to S$ over an irreducible log smooth
log scheme $S$ and write
$[X/G_X]=\mathcal{A}_{\mathcal{X}}\times_{\mathcal{A}_{S}}S$ as in Proposition \ref{prop:A_S_exists}.

\begin{lemma}
	Let $\mathcal{Y}\to \mathcal{X}$ be a generalized partial tropical compactification. Then $\mathcal{A}_{\mathcal{Y}}\times_{\mathcal{A}_{\mathcal{X}}}X\to X$
	is a generalized partial log compactification of $G_X\to S$.
\end{lemma}

\begin{proof}
	Pick a subdivision $\mathcal{Y}'\to \mathcal{Y}$ such that $\mathcal{Y}'$ is a partial tropical compactification.
	We need only show that $\mathcal{A}_{\mathcal{Y}'}\to \mathcal{A}_{\mathcal{Y}}$ is proper and representable by
	log algebraic spaces. Consider a morphism $W\to \mathcal{A}_{\mathcal{Y}}$ from
	a log scheme $W$. By Lemma \ref{lem:loc_factorization}, strict étale locally on $W$ we have a
	factorization $W\to X_{\sigma}\to \mathcal{A}_{\sigma}\to \mathcal{A}_{\mathcal{Y}}$
	for some RPC $\sigma$. By Corollary \ref{cor:fib_prods_preserved},
	it holds that $\mathcal{A}_{\sigma}\times_{\mathcal{A}_{\mathcal{Y}}}\mathcal{A}_{\mathcal{Y}'}\cong \mathcal{A}_{\sigma\times_{\mathcal{Y}}\mathcal{Y}'}$.
	By assumption, $\sigma\times_{\mathcal{Y}}\mathcal{Y}'\to \sigma$ is a subdivision. So $X_{\sigma}\times_{\mathcal{A}_{\mathcal{Y}}}\mathcal{A}_{\mathcal{Y}'}\cong X_{\sigma\times_{\mathcal{Y}}\mathcal{Y}'}\to X_{\sigma}$
	is a proper morphism of toric varieties. In particular, the pullback to $W$ is a proper morphism of log schemes.
\end{proof}

Our next goal is to prove that the inclusion of generalized partial log compactifications does not change the notion
of birational equivalence for partial log compactifications. For this, we require the following lemma:

\begin{lemma}
	\label{lem:always_bir_to_DM}
	Let $Y\to X$ be a
	partial log compactification of $G_X\to S$. Then there exists
	a proper morphism $Y'\to Y$ of partial log
	compactifications representable by log algebraic spaces
	such that $Y'$ is representable by a log DM stack.
\end{lemma}

\begin{proof}
	Pick a partial tropical compactification $\mathcal{Y}$, such that the corresponding combinatorial partial log compactification is as in Theorem \ref{thm:trop_corresp}, point (1).
	Since open immersions are representable by log DM stacks and properness is stable under base-change, it suffices to replace
	$Y$ with $X\times_{\mathcal{A}_{\mathcal{X}}}\mathcal{A}_{\mathcal{Y}}$.
	Let $\mathcal{Y}_0\to \mathcal{X}$ be a tropical compactification as in Proposition \ref{prop:stacky_fan_exists}.
	Then by Proposition \ref{prop:bir_iff_subdiv}, Corollary
	\ref{cor:fib_prods_preserved} and base-change, the morphism
	$$X\times_{\mathcal{A}_{\mathcal{X}}}\mathcal{A}_{\mathcal{Y}_0\times_{\mathcal{X}}\mathcal{Y}}\to X\times_{\mathcal{A}_{\mathcal{X}}}\mathcal{A}_{\mathcal{Y}}$$
	is proper and representable by log algebraic spaces.
	Let $\mathcal{Y}':=\mathcal{Y}_0\times_{\mathcal{X}}\mathcal{Y}$.
	Since $\mathcal{Y}\to \mathcal{X}$ is representable by finite cone stacks, so
	is $\mathcal{Y}'\to \mathcal{Y}_0$. In particular, $\mathcal{Y}'$ is a finite cone stack.
	The map $\mathcal{Y}'\to \mathcal{Y}_0$ is a monomorphism by base-change, hence since $\mathcal{Y}_0$
	is fibered in setoids, so is $\mathcal{Y}'$, i.e., $\mathcal{Y}'$ is a cone space. Finally, since
	$\mathcal{Y}'\to \mathcal{Y}_0$ is a monomorphism and $\mathcal{Y}_0$ is a cone complex, we see that
	there can exist at most one morphism between any pair of cones in $\mathcal{Y}'$, showing that $\mathcal{Y}'$
	is a cone complex.
	By Lemma \ref{lem:repr_when_cone_cx}, we know that
	$X\times_{\mathcal{A}_{\mathcal{X}}}\mathcal{A}_{\mathcal{Y}'}$ is representable
	by a log DM stack. Hence $Y':=X\times_{\mathcal{A}_{\mathcal{X}}}\mathcal{A}_{\mathcal{Y}'}$
	satisfies the conditions of the lemma.
\end{proof}

\begin{corollary}
	\label{cor:bir_can_assume_repr}
	Let $X\to S$ be a log semiabelian variety and $Y_1,Y_2\to X$ two
	generalized partial log compactifications of $G_X\to S$. Then $Y_1$ and $Y_2$ are birationally
	equivalent if and only if there exists a diagram:
	$$\begin{tikzcd}
		&Y_3\arrow[dr]\arrow[dl]&\\
		Y_1&&Y_2
	\end{tikzcd}$$
	of proper morphisms of partial log compactifications of $G_X\to S$
	representable by log algebraic spaces, such that $Y_3$ is representable
	by a log DM stack.
\end{corollary}

\begin{proof}
	``If'' is clear. For ``only if'', assume that $Y_1$ and $Y_2$ are birationally equivalent
	and let
	$$\begin{tikzcd}
		&Y_3\arrow[dr]\arrow[dl]&\\
		Y_1&&Y_2
	\end{tikzcd}$$
	be a diagram of proper morphisms of generalized partial log compactifications of $G_X\to S$
	representable by log algebraic spaces. By definition of a generalized partial log compactification, we may find a proper morphism
	$Y_3'\to Y_3$ of generalized partial log compactifications representable by log algebraic
	spaces with $Y_3'$ a partial log compactification. By Lemma \ref{lem:always_bir_to_DM}
	we can find a proper morphism $Y_3''\to Y_3'$ representable by log algebraic spaces
	such that $Y_3''$ is representable by log DM stacks. Hence replacing $Y_3$ with $Y_3''$, we conclude.
\end{proof}

\begin{proposition}
	Let $\mathcal{Y}\to \mathcal{X}$ be a minimal generalized partial tropical compactification. Then $\mathcal{A}_{\mathcal{Y}}\times_{\mathcal{A}_{\mathcal{X}}}X$
	is a minimal generalized partial log compactification.
\end{proposition}

\begin{proof}
	We must show that $\mathcal{A}_{\mathcal{Y}}\to \mathcal{A}_{\mathcal{X}}$ satisfies descent with respect to universally surjective,
	proper log étale monomorphisms.
	Let $T'\to T$ be such a morphism. Since $\mathcal{A}_{\mathcal{Y}}$
	satisfies strict étale descent, this is strict étale local on $T$.
	By \cite[Remark 2.7 (2)]{Grothendieck_to_Hu_Xi_2025}, we may strict
	étale locally find a log blowup $T''\to T'$ such that the composition
	$T''\to T'\to T$ is a log blowup. Hence it suffices to show that $\mathcal{A}_{\mathcal{Y}}\to \mathcal{A}_{\mathcal{X}}$
	satisfies descent with respect to log blowups. Since $\mathcal{A}_{\mathcal{Y}}\to \mathcal{A}_{\mathcal{X}}$ is
	a monomorphism, for any morphism $T\to \mathcal{A}_{\mathcal{X}}$, there exists at most one lift
	$T\to \mathcal{A}_{\mathcal{Y}}$. Hence we reduce to the following claim:

	Let $T'\to T$ be a log blowup and $f:T\to \mathcal{A}_{\mathcal{X}}$ a morphism. Assume that the composition
	$T'\to T\to \mathcal{A}_{\mathcal{X}}$ lifts to a morphism $T'\to \mathcal{A}_{\mathcal{Y}}$. Then
	$f$ lifts to a morphism $T\to \mathcal{A}_{\mathcal{Y}}$.

	As before, this statement is strict étale local on $T$, so we may assume that there exists a factorization
	$T\to \mathcal{A}_{\sigma}\to \mathcal{A}_{\mathcal{X}}$ for some strict morphism $T\to \mathcal{A}_{\sigma}$
	with $\sigma\in \RPC$. In this case, $T'$ is given as $T\times_{\mathcal{A}_{\sigma}}\mathcal{A}_{\Sigma}$ for
	some subdivision $\Sigma$ of $\sigma$. Replacing $\sigma$ with a face if necessary, we may assume that the image
	of $T\to \mathcal{A}_{\sigma}$ contains the unique closed point (corresponding to the maximal cone), i.e., that
	there exists $t\in T$ with $\overline{M}_{T,t}\cong \sigma^{\vee}\cap M_{\sigma}$. We claim that in this case
	the lift $T'\to \mathcal{A}_{\mathcal{Y}}$ factors as $T'\to \mathcal{A}_{\Sigma}\to \mathcal{A}_{\mathcal{Y}}$.
	Note that any such factorization is unique up to unique isomorphism, since the composition
	$\mathcal{A}_{\Sigma}\to \mathcal{A}_{\mathcal{Y}}\to \mathcal{A}_{\mathcal{X}}$ agrees with
	$\mathcal{A}_{\Sigma}\to \mathcal{A}_{\sigma}\to \mathcal{A}_{\mathcal{X}}$ and $\mathcal{A}_{\mathcal{Y}}\to \mathcal{A}_{\mathcal{X}}$
	is a monomorphism. Hence it suffices to show that the morphism $T'':=T'\times_{\mathcal{A}_{\Sigma}}\mathcal{A}_{\sigma'}\to \mathcal{A}_{\mathcal{Y}}$
	factors through a morphism $\mathcal{A}_{\sigma'}\to \mathcal{A}_{\mathcal{Y}}$ for any maximal cone $\sigma'$ in $\Sigma$.
	Note that the unique maximal point in $\mathcal{A}_{\sigma'}$ maps to the maximal point in $\mathcal{A}_{\sigma}$ and since
	the fiber of $T\to \mathcal{A}_{\sigma}$ over this point is non-empty, it follows that there exists $t'\in T'$ over $t$ with
	$\overline{M}_{T',t'}\cong \sigma'^{\vee}\cap M_{\sigma'}$. Let $U\to T'$ be a strict étale cover such that the morphism
	$T'\to \mathcal{A}_{\mathcal{Y}}$ is represented by a morphism $U\to \mathcal{A}_{\mathcal{Y}}'$. Let $u\in U$ be a point
	over $t'$. The composition of the morphism $\overline{M}_{U,u}\to \overline{M}_U(U)$ in $\ShpMon^{\op}$ induced by the restriction
	morphism with the map $\overline{M}_U(U)\to \mathcal{Y}$ gives a morphism $\alpha:\sigma'\to \mathcal{Y}$. The fact that $T'\to \mathcal{A}_{\mathcal{Y}}$
	is a lift of $T\to \mathcal{A}_{\mathcal{X}}$ implies that the composition of $\alpha$ with $\mathcal{Y}\to \mathcal{X}$ is equal
	to the composition of the morphism $\overline{M}_{U,u}\to \overline{M}_{T,t}$ in $\ShpMon^{\op}$ induced by pullback and the
	morphism $\overline{M}_{T,t}\to \mathcal{X}$. Hence $\alpha$ induces a lift $\mathcal{A}_{\sigma'}\to \mathcal{A}_{\mathcal{Y}}$
	as required. Gluing these maps for all maximal cones gives the factorization $\mathcal{A}_{\Sigma}\to \mathcal{A}_{\mathcal{Y}}$

	By Theorem \ref{thm:left_inv_geom}, the morphisms $\mathcal{A}_{\sigma}\to \mathcal{A}_{\mathcal{X}}$
	and $\mathcal{A}_{\Sigma}\to \mathcal{A}_{\mathcal{Y}}$ induce morphisms $\Sigma\to \mathcal{Y}$ and $\sigma\to \mathcal{X}$.
	Since $\mathcal{Y}$ is minimal, the morphism $\Sigma\to \mathcal{Y}$ descends to a lift $\sigma\to \mathcal{Y}$. Hence we get
	a lift $T\to \mathcal{A}_{\sigma}\to \mathcal{A}_{\mathcal{Y}}$ as required.
\end{proof}

Combined with the uniqueness of the minimal object in a birational equivalence class (see Proposition \ref{prop:exists_unique_minimal_log_cpt}), this shows:

\begin{corollary}
	\label{cor:unique_minimal_combinat}
	If a birational equivalence class of generalized log compactifications contains a combinatorial partial log compactification
	$X\times_{\mathcal{A}_{\mathcal{X}}}\mathcal{A}_{Y}$, then its unique minimal element
	is $X\times_{\mathcal{A}_{\mathcal{X}}}\mathcal{A}_{M_{Y}}$.
\end{corollary}

\section{Birational classification of compactified Jacobians}
\label{sec:bir_class_jac}

We use the results of the previous section to classify all compactified Jacobians over log smooth bases
up to birational equivalence. In the following, we fix an irreducible log smooth log scheme $S$
and a proper vertical log smooth family $C\to S$ of log curves over $S$ with a section $S\to C$. Denote by $U\subseteq S$ the dense open subset
on which $S$ has trivial log structure. In particular, $C_U$ is a family of smooth curves over $U$.

\begin{definition}
	A partially compactified Jacobian is a separated log smooth morphism $\overline{J}_C\to S$ of
	log DM stacks of finite presentation, together with an action $J_C\times_S \overline{J}_C\to \overline{J}_C$
	of the Jacobian $J_C$ of $C$, satisfying the following conditions:
	\begin{enumerate}
		\item $\overline{J}_C|_U\cong J_C|_U$.
		\item The unit section $U\to J_C|_U$ extends to a section $S\to \overline{J}_C$.
	\end{enumerate}
	A partially compactified Jacobian is called a compactified Jacobian if it is proper over $S$.
\end{definition}

\begin{definition}
	A morphism of partially compactified Jacobians is a
	$J_C$-equivariant morphism over $S$ preserving the extension of the unit section.
\end{definition}

\begin{remark}
	A semi-stable degeneration of $J_C|_U$ is an example of a proper
	vertical log smooth morphism.
\end{remark}

To be able to apply our results so far, we must construct a morphism from $\overline{J}_C$
to a log abelian variety. This is based on the following result by Holmes et al.:

\begin{theorem}[Theorem 1.7 in \cite{ModelsOfJacobHolmes2020} and Theorem 4.15.7 in \cite{TheLogarithmicMolcho2022}]
	\label{thm:neron_mapping}
	The logarithmic Jacobian $\LogJac(C/S)$ is a log abelian variety and
	satisfies the Néron mapping property for log smooth morphisms. That is, for
	every log smooth morphism $X\to S$, the restriction morphism:
	$$\Hom_S(X,\LogJac(C/S))\to \Hom_U(X|_U,J_C|_U)$$
	is an isomorphism.
\end{theorem}

\begin{corollary}
	\label{cor:let_over_logjac}
	Let $\overline{J}_C$ be a partially compactified Jacobian. Then $\overline{J}_C$
	admits a unique $J_C$-equivariant log étale monomorphism $\overline{J}_C\to \LogJac(C/S)$
	preserving the unit section.
\end{corollary}

For the proof we need the following lemma:

\begin{lemma}
	\label{lem:let_when_both}
	Let $A\to C$ and $B\to C$ be log étale morphisms.
	Assume that $A\to B$ is a morphism over $C$ locally of finite presentation. Then $A\to B$ is log étale.
\end{lemma}

\begin{proof}
	Since $A\to B$ is assumed to be locally of finite presentation, we need only show that
	$A\to B$ is formally log étale. For this,
	let $I\to I'$ be a strict square-zero extension and consider
	a commutative diagram of solid arrows:
	$$\begin{tikzcd}
		I\arrow[d]\arrow[r]&A\arrow[d]\\
		I'\arrow[r]\arrow[rd]\arrow[ru,dashed]&B\arrow[d]\\
									&C.
	\end{tikzcd}$$
	Since $A\to C$ is log étale there exists
	a unique dashed arrow $I'\to A$
	such that the diagram:
	$$\begin{tikzcd}
		I\arrow[d]\arrow[r]&A\arrow[d]\\
		I'\arrow[r]\arrow[ru,dashed]&C
	\end{tikzcd}$$
	commutes. It remains to show that the composition of the
	dashed arrow with $A\to B$
	agrees with $I'\to B$. For this, note that both
	of these define dashed morphisms making the following diagram
	commute:
	$$\begin{tikzcd}
		I\arrow[r]\arrow[d]&B\arrow[d]\\
		I'\arrow[r]\arrow[ru,dashed]&C.
	\end{tikzcd}$$
	Since $B\to C$ is log étale it follows that
	the two morphisms must be equal.
\end{proof}

\begin{proof}[Proof of Corollary \ref{cor:let_over_logjac}]
	Note that over $U$, the morphism is uniquely determined
	since $\LogJac(C/S)|_U\cong \overline{J}_C|_U\cong J_C|_U$ and the $J_C$-equivariant isomorphism is uniquely
	determined by requiring that it preserves the lifts of the unit section. By the Néron mapping property,
	we have
	$$\Hom(\overline{J}_C,\LogJac(C/S))\to \Hom(\overline{J}_C|_U,\LogJac(C/S)|_U)$$
	is an isomorphism, so $\overline{J}_C\to \LogJac(C/S)$ must be the preimage of $\overline{J}_C|_U\to \LogJac(C/S)|_U$
	under this isomorphism. It remains to show that this morphism has the properties given in the statement of the corollary.

	For the preservation of the unit section, note that the restriction of the composition $S\to \overline{J}_C\to \LogJac(C/S)$
	to $U$ is the unit section by assumption. Hence by the Néron mapping property, the composition is equal to the unit section.

	For equivariance, consider the
	two morphisms
	$$f_1,f_2:J_C\times \overline{J}_C\to \LogJac(C/S)$$
	defined as follows:
	$f_1$ is the composition of the group action morphism $J_C\times \overline{J}_C\to \overline{J}_C$
	with the morphism $\overline{J}_C\to \LogJac(C/S)$ constructed above and $f_2$ is the
	composition of the morphism $J_C\times \overline{J}_C\to J_C\times\LogJac(C/S)$ which is
	the identity on the first component together with the group action morphism $J_C\times \LogJac(C/S)\to \LogJac(C/S)$.
	Equivariance is equivalent to the statement that these two morphisms are equal. By the Néron mapping property,
	it suffices to check this after restricting to $U$. The isomorphism $\overline{J}_C|_U\to \LogJac(C/S)|_U$
	is $J_C|_U$-equivariant by construction, so $f_1|_U=f_2|_U$.

	It remains to show that
	$\overline{J}_C\to \LogJac(C/S)$ is a log étale monomorphism. For this, define
	$\overline{J}_C^{\trop}:=[\overline{J}_C/J_C]$. Since $\overline{J}_C$ and $J_C$ are
	log smooth over $S$ we have that $\overline{J}_C^{\trop}\to S$ is log smooth. Moreover,
	the restriction $\overline{J}_C^{\trop}|_U$ is isomorphic to $U$, so $\overline{J}_C^{\trop}\to S$
	is of relative dimension 0, hence log étale. Similarly, using \cite[Corollary 4.11.4]{TheLogarithmicMolcho2022},
	one shows that $\TroJac(C/S)\to S$ is log étale.
	We must show that
	$$\overline{J}_C^{\trop}\to \TroJac(C/S):=[\LogJac(C/S)/J_C]$$
	is a
	log étale monomorphism. The morphism is locally of finite
	presentation since $\overline{J}_C^{\trop}$
	and $\TroJac(C/S)$ are (see \cite[Corollary 4.2.3]{TheLogarithmicMolcho2022}) locally of finite presentation.
	Hence by Lemma \ref{lem:let_when_both}, the map $\overline{J}_C^{\trop}\to \TroJac(C/S)$
	is log étale.

	To see that $\overline{J}_C\to \LogJac(C/S)$
	is a monomorphism, observe that this is equivalent to the statement
	that the diagonal morphism
	$\overline{J}_C\to \overline{J}_C\times_{\LogJac(C/S)}\overline{J}_C$
	is an isomorphism. Observe that since $\overline{J}_C\to \LogJac(C/S)$ is log étale
	and $\overline{J}_C\to S$ is log smooth, the space $\overline{J}_C\times_{\LogJac(C/S)}\overline{J}_C$
	is log smooth over $S$ and the map $\overline{J}_C\times_{\LogJac(C/S)}\overline{J}_C\to \overline{J}_C$ given by projecting
	onto the first factor is log étale. $\overline{J}_C\times_{\LogJac(C/S)}\overline{J}_C\to \overline{J}_C$
	is a log scheme by \cite[Theorem 4.12.1]{TheLogarithmicMolcho2022}.
	Since $\overline{J}_C$ and $\overline{J}_C\times_{\LogJac(C/S)}\overline{J}_C\to \overline{J}_C$ are log étale over $\overline{J}_C$,
	by Lemma \ref{lem:let_when_both}, the morphism $\overline{J}_C\to \overline{J}_C\times_{\LogJac(C/S)}\overline{J}_C$
	is log étale. We show that it is also proper. Indeed, consider the composition:
	$$\overline{J}_C\to \overline{J}_C\times_{\LogJac(C/S)}\overline{J}_C\to \overline{J}_C\times_{S}\overline{J}_C.$$
	Since $\overline{J}_C$ is separated, the composition is a proper morphism. The second morphism
	$\overline{J}_C\times_{\LogJac(C/S)}\overline{J}_C\to \overline{J}_C\times_{S}\overline{J}_C$ fits into a fiber diagram:
	$$\begin{tikzcd}
		\overline{J}_C\times_{\LogJac(C/S)}\overline{J}_C\arrow[d]\arrow[r]&\LogJac(C/S)\arrow[d]\\
		\overline{J}_C\times_S \overline{J}_C\arrow[r]&\LogJac(C/S)\times_S \LogJac(C/S).
	\end{tikzcd}$$
	By \cite[Theorem 4.12.1]{TheLogarithmicMolcho2022}, the right-hand morphism is representable by finite (hence separated)
	morphisms of schemes, so by base-change
	$$\overline{J}_C\times_{\LogJac(C/S)}\overline{J}_C\to \overline{J}_C\times_{S}\overline{J}_C$$
	is separated. By \cite[Tag 01W6]{stacks-project}, it follows that $\overline{J}_C\to \overline{J}_C\times_{\LogJac(C/S)}\overline{J}_C$
	is proper. Since it is an isomorphism over $U$, it follows that it is a proper birational monomorphism and hence an isomorphism.
	To see that the log structures agree, note that since both spaces are log smooth, their log structures
	are induced by reduced (Weil) divisors. Hence we reduce to the following situation:

	Let $X$ be a scheme and $D_1,D_2\subseteq X$ be two reduced divisors such that there is a log étale morphism of log schemes
	$X_1\to X_2$ whose underlying map of schemes is the identity,
	where $X_i$ is the log scheme $X$ equipped with the log structure induced by $D_i$. We must show that $D_1=D_2$.
	For this, first observe that we must have $D_2\subseteq D_1$ in order for a morphism $X_1\to X_2$ to exist. Moreover,
	on $X\setminus D_2$, the log scheme $X_2$ has trivial log structure, so $X_1\setminus D_2\to X_2\setminus D_2$ is strict étale.
	In particular, $X_1\setminus D_2$ has trivial log structure, so $D_1\subseteq D_2$, finishing the proof.
\end{proof}

\begin{corollary}
	\label{cor:cpt_iff_cpt_jac}
	In the notation of Corollary \ref{cor:let_over_logjac}, a partially compactified Jacobian
	is a compactified Jacobian if and only if $\overline{J}_C\to \LogJac(C/S)$ is a log compactification
	of $J_C\to S$.
\end{corollary}

\begin{proof}
	Assume first that $\overline{J}_C\to \LogJac(C/S)$ is a log compactification of $J_C\to S$, i.e., that the
	map is proper. Since $\LogJac(C/S)\to S$ satisfies the valuative
	criterion for properness by \cite[Theorem 4.10.1]{TheLogarithmicMolcho2022}
	and this property is stable under compositions, we have that
	$\overline{J}_C\to S$ satisfies the valuative criterion for
	properness. Since $\overline{J}_C$ is of finite presentation,
	it follows that $\overline{J}_C$ is proper.

	Conversely, assume that $\overline{J}_C$ is proper. Since
	the diagonal of $\LogJac(C/S)$ is representable by
	finite (hence in particular proper) morphisms by \cite[Theorem 4.12.1]{TheLogarithmicMolcho2022},
	it follows that $\overline{J}_C\to \LogJac(C/S)$ is proper,
	since any morphism with proper source and separated target
	is proper.
\end{proof}

\begin{corollary}
	\label{cor:equiv_cat_cptfd_jacs}
	The map which sends $\overline{J}_C$ to the morphism
	$\overline{J}_C\to \LogJac(C/S)$ is an equivalence of categories
	between the category of compactified Jacobians
	and the category of log compactifications of $J_C\to S$
	representable by log DM stacks.
\end{corollary}

\begin{proof}
	Assume first that $\overline{J}_C$ is a compactified Jacobian.
	By Corollary \ref{cor:cpt_iff_cpt_jac}, the morphism $\overline{J}_C\to \LogJac(C/S)$
	is a log compactification of $J_C\to S$. Conversely,
	given a log compactification $\overline{J}_C\to \LogJac(C/S)$,
	note that since $\LogJac(C/S)\to S$ is log smooth by \cite[Corollary 4.11.4]{TheLogarithmicMolcho2022}, the same
	is true for $\overline{J}_C\to S$. Since $\overline{J}_C\to \LogJac(C/S)$
	is a monomorphism, it is separated and since the same is true
	for $\LogJac(C/S)\to S$ by \cite[Theorem 4.12.1]{TheLogarithmicMolcho2022},
	it follows that $\overline{J}_C\to S$ is separated. Over
	$U$, the map
	$$\overline{J}_C |_U\to \LogJac(C/S)|_U=J_C|_U$$
	is a strict log étale monomorphism hence a strict open immersion. Since $J_C|_U$
	acts transitively on itself, it follows that $\overline{J}_C\to J_C|_U$ is surjective
	by equivariance. Hence $\overline{J}_C\to J_C|_U$ is an isomorphism.
	The existence of a lift of the unit section and equivariance
	is clear. Finally, since $\overline{J}_C\to \LogJac(C/S)$ is locally
	of finite presentation and $\LogJac(C/S)\to S$ is locally of
	finite presentation by \cite[Corollary 4.2.3]{TheLogarithmicMolcho2022},
	it follows that $\overline{J}_C\to S$ is locally of
	finite presentation. The stack $\overline{J}_C$ is also
	quasi-compact, since this statement is log étale local and
	$\LogJac(C/S)$ has log étale covers that are proper over $S$
	\cite{LogarithmicAbeKajiwa2015}. Hence $\overline{J}_C$
	is of finite presentation over $S$.

	This shows that the map is a bijection on objects.
	To show that the same is true for morphisms, it remains to
	show that any morphism $\overline{J}_C\to \overline{J}_C'$ of
	compactified Jacobians lies over the
	morphisms to $\LogJac(C/S)$. I.e., we must show that the morphism
	$$\overline{J}_C\to \overline{J}_C'\to \LogJac(C/S)$$
	agrees with $\overline{J}_C\to \LogJac(C/S)$. Note that both of these
	morphisms are $J_C$-equivariant morphisms $\overline{J}_C\to \LogJac(C/S)$
	which preserve the unit section. Hence by the uniqueness part
	of Corollary \ref{cor:let_over_logjac}, we conclude that the two morphisms agree.
\end{proof}

\begin{definition}
	Two partially compactified Jacobians $\overline{J}_C$ and $\overline{J}_C'$ are said
	to be birationally equivalent if there exists a third partially
	compactified Jacobian $\overline{J}_C''$ and proper morphisms $\overline{J}_C''\to \overline{J}_C$ and $\overline{J}_C''\to \overline{J}_C'$
	of partially compactified Jacobians representable by log algebraic spaces.
\end{definition}

We now fix an Artin fan $S\to \mathcal{A}_S$ such
that $C\to S$ descends to a tropical curve $\mathcal{C}\to \mathcal{A}_S$. By
\cite[Theorem 4.9.4]{TheLogarithmicMolcho2022}, we have
$$[\LogJac(C/S)/J_C]\cong \TroJac(C/S),$$
where the tropical
Jacobian is defined in
\cite[Definition 3.6.1]{TheLogarithmicMolcho2022} for the case
in which $S$ has a global chart. From the construction on
page 1510 of \cite{TheLogarithmicMolcho2022} of $\TroJac(C/S)$
for general $S$, it is clear that $\TroJac(C/S)$ is the
pull-back of $\TroJac(\mathcal{C}/\mathcal{A}_S)$. The morphism
$\TroJac(\mathcal{C}/\mathcal{A}_S)\to \mathcal{A}_S$ is a family of tropical
abelian varieties.

\begin{figure}
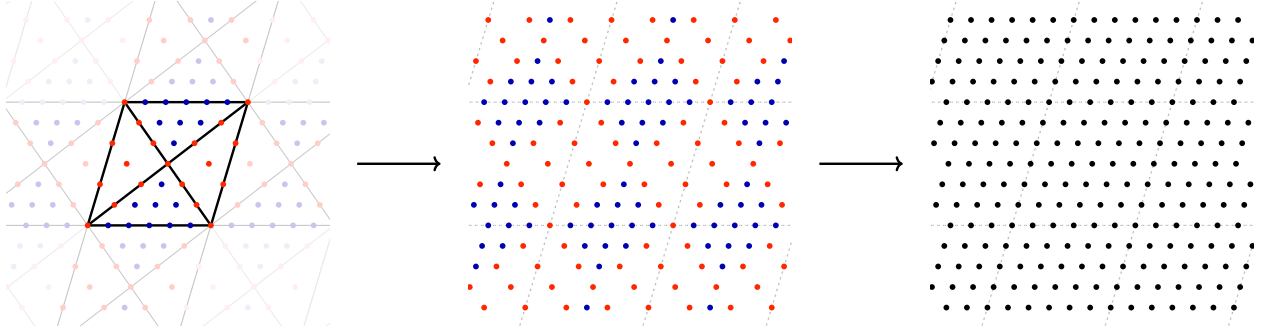

	\resizebox{\textwidth}{!}{
}
	\caption{An example of a tropicalization of a compactified Jacobian (left), the minimal generalized tropical compactification it is birational to (center), and the tropical Jacobian (right). This is an equivariant analogue of the figure in the introduction.}
\end{figure}

\begin{corollary}
	\label{cor:main_thm_jac}
	There is a bijection between the set of birational equivalence classes of compactified Jacobians
	and the set of minimal generalized tropical compactifications of $\TroJac(\Sigma_{\mathcal{C}}/\Sigma_{\mathcal{A}_{S}})$.
	This bijection is obtained by sending an equivalence class,
	viewed as a birational equivalence class of generalized log compactifications of $J_C\to S$,
	to the tropicalization of its unique minimal object.
\end{corollary}

\begin{proof}
	Everything in the statement of the corollary has been proven in the previous sections for
	general $G_X\to S$. We explain how to combine the results to obtain the statement of the corollary:
\begin{itemize}
	\item By Corollary \ref{cor:bir_can_assume_repr} and Corollary \ref{cor:equiv_cat_cptfd_jacs}
		two compactified Jacobians are birationally equivalent if and only if they are birationally equivalent as
		generalized log compactifications of $J_C\to S$. Moreover, by Lemma \ref{lem:always_bir_to_DM}, every birational equivalence
		class of generalized log compactifications of $J_C\to S$ contains a compactified Jacobian.
	\item By Theorem \ref{thm:trop_corresp}, any log compactification is of the form $\LogJac(C/S)\times_{\mathcal{A}_S}\mathcal{A}_{\mathcal{Y}}$
		for a uniquely determined tropical compactification $\mathcal{Y}$ of $\TroJac(\Sigma_{\mathcal{C}}/\Sigma_{\mathcal{A}_{S}})\to \mathcal{A}_S$.
	\item Since $S$ is log smooth, the morphism $S\to \mathcal{A}_S$ is a log smooth cover, so by Proposition \ref{prop:bir_iff_subdiv}, a morphism of log compactifications is proper and representable by log algebraic spaces if and only if the corresponding tropical morphism is a subdivision. In particular, the birational equivalence class of the log compactification uniquely determines the birational equivalence class of its tropicalization.
	\item By Corollary \ref{cor:combinat_descr_min}, each birational equivalence class of generalized log compactifications contains a unique minimal element of the form $\LogJac(C/S)\times_{\mathcal{A}_{\TroJac(\Sigma_{\mathcal{C}})}}\mathcal{A}_{\mathcal{M}}$ for some minimal generalized tropical compactification $\mathcal{M}\to \TroJac(\Sigma_{\mathcal{C}})$.
	\item By point (3), $\mathcal{M}$ is uniquely determined by the birational equivalence class.
\end{itemize}
\end{proof}

By Proposition \ref{prop:minimal_iff_minimal_on_cones} and Corollary \ref{cor:combinat_descr_min}, minimal generalized tropical compactifications
of $\TroJac(\Sigma_{\mathcal{C}}/\Sigma_{\mathcal{A}_{S}})$ have a concrete combinatorial description. We spell this out in detail
in the case when $\mathcal{A}_S=\mathcal{A}_{\sigma}$ is a cone.

\begin{corollary}
	Assume in the above situation that $\mathcal{A}_{S}=\mathcal{A}_{\sigma}$ for some RPC $\sigma$ (for example,
	if $S$ is a DVR with its divisorial log structure) and that $C$ has a tropicalization $\mathcal{C}\to \sigma$
	over $\sigma$. Then there is a bijection between the set of birational equivalence
	classes of compactified Jacobians and the set of subsets $Y\subseteq \TroJac(\mathcal{C}/\sigma)$
	of the form:
	$$Y=\bigcup_{\tau\in \Delta}\tilde{N}_{\tau}\cap \tau$$
	for some complete stacky fan $\Delta$ in $\TroJac(\mathcal{C}/\sigma)$.
\end{corollary}

\begin{proof}
	This follows from the previous corollary, Corollary \ref{cor:combinat_descr_min} and Lemma \ref{lem:complete_iff_cpt}.
\end{proof}

\section{Birational classification of toric varieties}

In this section, we assume that $S$ is smooth and equipped with the trivial log structure. We consider
a torus bundle $T\to S$. Write
$$T_{\log}:=T\times^{\mathbb{G}_{m}^{n}}\mathbb{G}_{m,\log}^{n}.$$
This is a log semiabelian variety over $S$, so the theory of the preceding sections applies. We start with
the following definition:

\begin{definition}
	A toric orbifold bundle over $S$ is a family $X\to S$ with an action $T\times X\to X$
	of $T$ on $X$ and a $T$-equivariant open immersion $T\to X$ with dense image, such
	that, étale locally on $S$, we have $T\cong S\times T_0$ for some torus $T_0$
	and $X\cong S\times X_{\Delta}$ via a $T$-equivariant isomorphism restricting to the identity on the dense open tori
	isomorphic to $T$, for some toric DM stack
	$X_{\Delta}$ with dense open torus $T_0$ in the sense of
	\cite[Definition 3.10.1]{ATheoryOfStaGillam2015}. Here, $\Delta$ denotes the stacky fan of $X_{\Delta}$.
	We call $X$ a complete toric orbifold bundle if $X\to S$ is proper.
\end{definition}

\begin{remark}
	In \cite{ATheoryOfStaGillam2015}, stacky fans are called lattice KM fans.
\end{remark}

\begin{definition}
	A morphism of toric orbifold bundles over $S$ is a $T$-equivariant
	morphism which commutes with the open embedding of $T$.
\end{definition}

\begin{remark}
	The entire theory of toric varieties carries over essentially
	verbatim to the category of toric DM stacks. In particular,
	there is a combinatorial classification of toric DM stacks
	in terms of stacky fans, see
	\cite[Theorem 3.10.7]{ATheoryOfStaGillam2015}.
\end{remark}

\begin{lemma}
	Let $M_X\subseteq\mathcal{O}_X$ be the (étale) subsheaf of sections $s\in \mathcal{O}_X$
	such that for every $U$ with $T|_U\cong U\times T_0$ and $X|_U\cong U\times X_{\Delta}$, we have $s\in M_{U\times X_{\Delta}}$,
	where $M_{U\times X_{\Delta}}$ is the logarithmic structure pulled back from the canonical logarithmic structure on $X_{\Delta}$.
	Then $M_X$ is a logarithmic structure on $X$ such that $X\to S$ is log smooth.
\end{lemma}

\begin{proof}
	Let $U$ be as in the statement of the lemma. We claim that $M_X|_U\cong M_{U\times X_{\Delta}}$. Indeed,
	let $V\to U$ be étale and non-empty and assume that $X|_V\cong V\times X_{\Delta'}$ for some fan $\Delta'$.
	Since $X|_V\cong X|_U|_V\cong V\times X_{\Delta}$, we get a $T_0$-equivariant isomorphism $X_{\Delta}\to X_{\Delta'}$,
	so $\Delta=\Delta'$. Hence the restriction of $M_{U\times X_{\Delta}}$ to $V$ agrees with $M_{V\times X_{\Delta'}}$
	and the claim follows. Since the property of being a log structure is étale local, it follows that $M_X$ is a
	log structure on $X$. Similarly, log smoothness is étale local and stable under base-change, so log smoothness
	of $X\to S$ follows from the fact that toric varieties are log smooth over a point with trivial log structure.
	To see the latter, note that by \cite[Proposition 2.6.4]{ATheoryOfStaGillam2015},
	we reduce to the case of toric varieties, in which case the
	result follows from \cite[Theorem 3.1.8]{ogus}.
\end{proof}

In the following we will always assume that $X$ is a log scheme equipped with the log structure above.

\begin{lemma}
	\label{lem:exist_unique_map}
	Let $X$ be a toric orbifold bundle over $S$. Then there exists a unique $T$-equivariant log étale monomorphism
	$X\to T_{\log}$ representable by log DM stacks and compatible with the inclusions $T\to X$ and $T\to T_{\log}$.
\end{lemma}

\begin{proof}
	Step 1 (Reduce to local case): By uniqueness, \cite[Theorem 1.1 (3)]{fortman2025descentalgebraicstacks} and étale descent,
	it suffices to construct the morphism étale locally on $S$
	and $X$.
	So we may assume that $T\cong S\times \mathbb{G}_{m}^{n}$
	and $X\cong S\times X_{\sigma}$ for some rational polyhedral cone $\sigma$, where
	$X_{\sigma}$ denotes the affine toric variety with
	cone $\sigma$ (see \cite[Proposition 2.6.4]{ATheoryOfStaGillam2015}).

	Step 2 (Construct the morphism locally): A morphism $X\to \mathbb{G}_{m,\log}^{n}$ is equivalent to the
	datum of $n$ sections of $M_X^{\gp}(X)$. Note that we
	have a canonical inclusion $\mathbb{Z}^{n}\to M_X^{\gp}(X)$
	since $\mathbb{Z}^{n}$ is the character lattice of $\mathbb{G}_{m}^{n}$.
	We let $X\to \mathbb{G}_{m,\log}^{n}$ be the morphism defined
	by the images of the $n$ standard basis vectors.

	Step 3 (Verify compatibility, equivariance and log étaleness): Compatibility
	with $T\to X$ is equivalent to the statement that under the canonical
	restriction $M_X^{\gp}(X)\to \mathcal{O}_T^{\times}(T)\cong \mathcal{O}_S^{\times}(S)\langle t_1,\dots,t_n,\rangle$,
	the morphism corresponds to the map
	$$\mathbb{Z}^{n}\to \mathcal{O}_S^{\times}(S)\langle t_1,t_1^{-1},\dots,t_n,t_n^{-1}\rangle$$
	which sends the $i$-th standard basis vector to $t_i$. Hence we see that the above map is the unique extension
	of $\mathbb{G}_{m}^{n}\to \mathbb{G}_{m,\log}^{n}$. It is clear that this map is $\mathbb{G}_{m}^{n}$-equivariant.

	Finally, note that both $[X/T]$ and $[T_{\log}/T]$ are log étale over $S$ (this can be checked strict étale locally on $S$). Since $X$ is locally of finite presentation, by Lemma \ref{lem:let_when_both},
	the morphism $X\to T_{\log}$, obtained by base-change from $[X/T]\to [T_{\log}/T]$, is log étale.

	Step 4 (Show that it is a monomorphism): This is strict étale local on $S$, so we may assume that $X\cong S\times X_{\Delta}$
	for some stacky fan $\Delta$ and that $T$ is as above. From the construction of $X\to T_{\log}$, we see that the
	morphism is pulled back from a morphism $X_{\Delta}\to \mathbb{G}_{m,\log}^{n}$ and hence it suffices to show that
	this morphism is a monomorphism. For this, we observe that $X_{\Delta}$ represents the functor which to
	a log scheme $Z$ associates the set of homomorphisms $\phi\in \Hom(\mathbb{Z}^{n},M_Z^{\gp}(Z))$ such that strict
	étale locally on $Z$, the induced morphism $\overline{M}_Z(Z)^{\vee}\to \mathbb{Z}^{n}$ factors through $N_{\sigma}\cap \sigma$
	for some cone $\sigma\in \Delta$, where $N_{\sigma}$ denotes the
	finite-index sublattice corresponding to $\sigma$.
	(to see this, write $X_{\Delta}$ as a colimit of affine toric DM stacks $X_{\sigma}$ and use the universal property
	of $X_{\sigma}$). The map $X_{\Delta}\to \mathbb{G}_{m,\log}^{n}$ sends a morphism $\phi:\mathbb{Z}^{n}\to M_Z^{\gp}(Z)$
	to itself ($\mathbb{G}_{m,\log}^{n}$ represents the functor which sends $Z$ to $\Hom(\mathbb{Z}^{n},M_Z^{\gp}(Z))$).
	In particular, the map is a monomorphism.

	Step 5 (Show representability by log DM stacks): Since
	$X$ is a log DM stack, it suffices to show that the diagonal
	$$\mathbb{G}_{m,\log}\to \mathbb{G}_{m,\log}\times \mathbb{G}_{m,\log}$$
	is representable by log algebraic spaces of finite presentation. Note
	that the diagonal morphism is the base-change of the inclusion
	$0\to \mathbb{G}_{m,\log}$ of the zero-section by the morphism
	$\mathbb{G}_{m,\log}\times\mathbb{G}_{m,\log}\to \mathbb{G}_{m,\log}$
	sending $(a,b)$ to $ab^{-1}$. Factor $0\to \mathbb{G}_{m,\log}$
	as $0\to \mathbb{G}_m\to \mathbb{G}_{m,\log}$. The first
	map is representable by log schemes of finite presentation
	and the second one is the base-change of $0\to \mathbb{G}_{m,\trop}$
	under the projection $\mathbb{G}_{m,\log}\to \mathbb{G}_{m,\trop}$.
	Hence by \cite[Proposition 2.2.7.5]{TheLogarithmicMolcho2022}, it is
	representable by log schemes of finite presentation\footnote{The proposition
		statement only claims that it is representable by morphisms of finite type.
	However, the method of proof shows that it is in fact of finite presentation}.
\end{proof}

\begin{proposition}
	\label{prop:equiv_cat_tor_compact}
	The map which sends $X$ to $X\to T_{\log}$ defines an equivalence
	of categories between the category of complete toric orbifold bundles
	over $S$ and the category of log compactifications
	of $T\to S$ representable by log DM stacks.
\end{proposition}

\begin{proof}
	Note that the unit section $S\to T$ composed with the inclusion
	$T\to X$ gives a lift of the unit section $S\to T_{\log}$ to
	$X$. Hence by the preceding lemma, we know that when $X$ is a
	toric orbifold bundle, then $X\to T_{\log}$ is a partial log compactification.
	If $X$ is complete, then $X\to S$ satisfies the valuative criterion for properness.
	Since $\mathbb{G}_{m,\log}$ satisfies the valuative criterion for properness,
	so does $T_{\log}\to S$. It follows that $X\to T_{\log}$ satisfies the valuative
	criterion for properness. Note that $X$ is of finite presentation
	since this is strict étale local on $S$ and toric varieties are of finite presentation. Moreover, the diagonal
	of $T_{\log}$ is representable by morphisms of finite presentation by Lemma
	\ref{lem:part_cptifications_are_nice}, so it follows that $X\to T_{\log}$ is proper,
	i.e., $X$ is a log compactification.

	Conversely, assume that $X\to T_{\log}$ is a log compactification such
	that $X$ is representable by a log DM stack. We obtain
	a morphism $T\to X$ by composing the lift $S\to X$ of the
	unit section with the group action
	$$T\times X\to X.$$
	Note that the composition $T\to X\to T_{\log}$ is the canonical inclusion
	$T\to T_{\log}$ and hence in particular a monomorphism. Moreover,
	$T\to X$ is the base-change of the lift of the zero-section
	$S\to [X/T]$ over $X$. Since both $S$ and $[X/T]$ are log étale
	over $S$, Lemma \ref{lem:let_when_both} shows that
	$S\to [X/T]$ is log étale, so by base-change
	$T\to X$ is log étale. Since $T$ is equipped with the trivial
	log structure, it follows that $T\to X$ is a strict étale
	monomorphism, hence an open immersion \cite[Tag 025G]{stacks-project}.

	Assume we are given a morphism $X_1\to X_2$ of toric orbifold bundles.
	Then the composition
	$$X_1\to X_2\to T_{\log}$$
	is a $T$-equivariant
	log étale monomorphism which preserves the embedding
	of $T$. Hence by the uniqueness in Lemma \ref{lem:exist_unique_map},
	it follows that $X_1\to X_2$ is a morphism over $T_{\log}$ showing
	that the functor is fully faithful.

	It remains to show that strict étale locally on $S$ we have
	$X\cong S\times X_{\Delta}$ for some complete toric DM stack $X_{\Delta}$.
	By Theorem \ref{thm:trop_corresp}, there exists an Artin fan
	$S\to \mathcal{A}_{S}$ such that there exists a tropical
	torus $\mathcal{T}\to \Sigma_{\mathcal{A}_S}$ and a tropical compactification
	$\mathcal{X}\to \mathcal{T}$ such that
	$$X\cong T_{\log}\times_{\mathcal{A}_{\mathcal{T}}}\mathcal{A}_{\mathcal{X}}.$$
	Since $S$ has trivial log structure, the Artin fan $\mathcal{A}_{S}$
	is a point modulo a finite group. Hence after replacing $S$
	with an étale cover, we may assume that $\mathcal{A}_S$ is a point
	with trivial log structure. In this case, we have that $\mathcal{T}\cong T_N^{\trop}$
	for some lattice $N$. Since $X$ is a log compactification, the map $\mathcal{X}\to T_N^{\trop}$
	is surjective by Lemma \ref{lem:proper_iff_surj}. Hence we need only show that $\mathcal{X}$ is
	isomorphic to a stacky fan. Let $s\in S$ be a geometric point. Note that since $X_s\to s$ is log smooth,
	the space $X_s$ is in particular normal. Hence the underlying DM stack of $X_s$ is a toric DM stack by \cite[Definition 3.10.1]{ATheoryOfStaGillam2015}.
	By equivariance, the log structure on $X_s$ must be torus-invariant, i.e., it is induced by a linear combination
	of irreducible torus-invariant divisors. In order for $T_s\to T_{\log,s}$ to extend to a morphism $X_s\to T_{\log,s}$
	we must have that the canonical map $M\to \mathcal{O}_{T_s}^{\times}(T_s)$, extends to a map $M\to M_{X_s}^{\gp}(X_s)$.
	Since the image of $M$
	inside the function field $K(X_s)$ of $X_s$ is the set of standard monomials of the toric orbifold, which cut out the full toric boundary, it follows
	that the log structure on $X_s$ is the log structure associated to the toric boundary, i.e., the canonical log structure on
	$X_s$. The morphism
	$$X_s\to X\to \mathcal{A}_{\mathcal{X}}$$
	is strict, so we have a factorization
	$X_s\to \mathcal{A}_{\Delta}\to \mathcal{A}_{\mathcal{X}}$ where $\Delta$ is the stacky fan (or lattice KM fan) associated
	to $X_s$. Note that $\mathcal{A}_{\Delta}\to \mathcal{A}_{\mathcal{T}}$ is proper by Lemma \ref{lem:proper_iff_surj}.
	Hence $\mathcal{A}_{\Delta}\to \mathcal{A}_{\mathcal{X}}$ is a strict proper log étale monomorphism. Since a strict log étale
	monomorphism is an open immersion by \cite[025G]{stacks-project}, it follows that $\mathcal{A}_{\Delta}\to \mathcal{A}_{\mathcal{X}}$
	is an isomorphism onto a connected component of $\mathcal{A}_{\mathcal{X}}$. But $\mathcal{A}_{\mathcal{X}}$ is connected, since it has
	a cover by the irreducible log DM stack $X$ and hence $\mathcal{A}_{\Delta}\to \mathcal{A}_{\mathcal{X}}$ is an isomorphism.
	It follows that $\mathcal{X}$ is a complete stacky fan, so
	$$X=T_{\log}\times_{\mathcal{A}_{\mathcal{T}}}\mathcal{A}_{\Delta}\cong S\times X_{\Delta}$$
	as required.
\end{proof}

\begin{definition}
	Let $X_1$ and $X_2$ be toric orbifold bundles. We say that $X_1$ and $X_2$ are birationally equivalent if there
	exists a diagram:
	$$\begin{tikzcd}
		&X_3\arrow[dr]\arrow[dl]&\\
		X_1&&X_2
	\end{tikzcd}$$
	of proper morphisms of toric orbifold bundles
	representable by log algebraic spaces.
\end{definition}

\begin{remark}
	It is clear from the definition that two complete orbifold bundles are
	birationally equivalent if and only if they are birationally
	equivalent as log compactifications of $T_{\log}$.
\end{remark}

Fix an Artin fan $\mathcal{A}_{\Sigma}$ for $S$ as in Proposition \ref{prop:A_S_exists} and a
family of tropical tori $\mathcal{T}\to \Sigma$ such that $T\cong S\times_{\mathcal{A}_{\Sigma}}\mathcal{A}_{\mathcal{T}}$.

\begin{corollary}
	\label{cor:main_thm_tor}
	There is a bijection between the set of birational equivalence classes of complete toric orbifold bundles
	and the set of minimal generalized tropical compactifications of $\mathcal{T}$. This bijection is given by sending
	an equivalence class to the tropicalization of its unique minimal object, where we allow generalized log
	compactifications in the birational equivalence class.
\end{corollary}

\begin{proof}
	By Corollary \ref{cor:bir_can_assume_repr} and Lemma \ref{lem:cpt_bir_inv},
	two complete toric orbifold bundles are birationally equivalent if and
	only if they are birationally equivalent as generalized log compactifications. Moreover, by Lemma \ref{lem:always_bir_to_DM} and Proposition \ref{prop:equiv_cat_tor_compact},
	every birational equivalence class of generalized log compactifications contains a complete toric orbifold bundle.
	Finally, by Theorem \ref{thm:trop_corresp} and Corollary \ref{cor:unique_minimal_combinat}, there is a bijection between the set of birational equivalence
	classes of generalized log compactifications and the set of minimal generalized tropical compactifications of $\mathcal{T}$
	given by sending an equivalence class to the tropicalization of its unique minimal object. See also the proof of Corollary \ref{cor:main_thm_jac}
	for more details.
\end{proof}

The above result can be made more explicit in the case when $\mathcal{A}_{\Sigma}$ is a point, or, equivalently,
when $T\cong S\times \mathbb{G}_{m}^{n}$ for some $n\ge 0$.

\begin{theorem}
	\label{thm:combinat_class_toric_vars}
	Let $S$ be a smooth scheme with trivial log structure. Then there is a natural bijection between the set of birational
	equivalence classes of complete toric orbifold
	bundles compactifying $S\times\mathbb{G}_{m}^{n}$ and the set of subsets $U\subseteq \mathbb{R}^{n}$ of the form
	$$U=\bigcup_{\sigma\in \Delta}^{}\sigma\cap N_{\sigma},$$
	where $\Delta$ is some complete stacky fan in $\mathbb{R}^{n}$ and $N_{\sigma}$ denotes the finite-index sublattice of $\mathbb{Z}^{n}\cap \Span(\sigma)$
	in this stacky fan.
\end{theorem}

\begin{proof}
	This follows from the preceding corollary and Corollary \ref{cor:combinat_descr_min}.
\end{proof}

In the special case $S=\Spec(k)$ we obtain a generalization of the main result from \cite{Schmitt}:

\begin{corollary}
	\label{cor:lattice_colorings}
	There is a bijection between the set of birational equivalence
	classes of proper toric DM stacks with lattice $N$
	and the set of sublattice colorings (see \cite[Definition 4]{Schmitt}).
\end{corollary}

For the proof, we need the following lemma:

\begin{lemma}
	Let $\sigma\subseteq N_{\mathbb{R}}$ be an RPC for some lattice $N$. Assume
	that $\Span(\sigma)=N_{\mathbb{R}}$. Then $(\sigma\cap N)^{\gp}\cong N$.
\end{lemma}

\begin{proof}
	Since $\sigma\cap N\subseteq N$, we have $(\sigma\cap N)^{\gp}\subseteq N$.
	It suffices to show that any $x\in N$ can be written as a difference
	of elements in $\sigma\cap N$. Let $y\in \sigma\cap N$ be in the
	relative interior of $\sigma$. Hence for every $\phi\in \sigma^{\vee}$
	we have $\phi(y)>0$. Let $\{\phi_1,\dots,\phi_r\}$ be the primitive generators
	of the rays of $\sigma^{\vee}$. Then $\sigma=\bigcap_{i=1}^{r}\{x\in N_{\mathbb{R}}|\phi_i(x)\ge 0\}$.
	Pick $n\in \mathbb{N}$ large enough such that $\phi_i(ny)+\phi_i(x)\ge 0$ for all $i$.
	Then $ny+x\in \sigma$ and hence $x=(ny+x)-ny\in (\sigma\cap N)^{\gp}$.
\end{proof}

\begin{proof}[Proof of Corollary \ref{cor:lattice_colorings}]
	By Theorem \ref{thm:combinat_class_toric_vars}
	we get a bijection between the set of birational equivalence classes of toric
	DM stacks and the set of subsets
	$S\subseteq N_{\mathbb{R}}$ of the form
	$$S=\bigcup_{\sigma\in \Delta}^{}\sigma\cap N_{\sigma},$$
	for some complete stacky fan $\Delta$. We claim that the
	datum of $S$ is equivalent to the datum of a sublattice
	coloring of $N$. Indeed, given a sublattice coloring
	$(C_{N'})_{N'\in \mathcal{N}}$, we obtain a set as above
	by letting:
	$$S:=\bigcup_{N'\in \mathcal{N}}^{}C_{N'}\cap N'.$$
	Conversely, given $S$, define for each finite-index
	sublattice $N'\subseteq N$ a set $C_{N'}$ by letting
	$C_{N'}$ be the union of all maximal dimensional rational
	polyhedral cones $\sigma$ in $N_{\mathbb{R}}$ such that
	$\sigma\cap N'=\sigma\cap S$. We claim that this is a
	sublattice coloring. For this, observe that
	$$C_{N'}\supseteq \bigcup_{\sigma\in \Delta:N_{\sigma}=N'}\sigma.$$
	It suffices to show that these two sets are equal.
	Let $\sigma\subseteq C_{N'}$ be a maximal dimensional RPC
	such that $\sigma\cap N'=\sigma\cap S$. Assume that
	$\sigma$ is not contained in the locus $\bigcup_{\sigma'\in \Delta:N_{\sigma'}=N'}\sigma'$.
	Hence there exists $\sigma'\in \Delta$ maximal such that
	$\sigma\cap \sigma'$ intersects the relative interior of $\sigma$
	and $N_{\sigma'}\neq N'$. In particular, $\sigma\cap \sigma'$ is maximal dimensional,
	i.e., $\Span(\sigma\cap \sigma')=N_{\mathbb{R}}$. By the preceding lemma, we have
	$$N_{\sigma'}=(\sigma\cap \sigma'\cap N_{\sigma'})^{\gp}=(\sigma\cap \sigma'\cap S)^{\gp}=(\sigma\cap \sigma'\cap N')^{\gp}=N',$$
	a contradiction. Hence equality holds, showing that
	$C_{N'}$ is a union of finitely many rational polyhedral cones.
	It is clear that there are only finitely many non-empty
	$C_{N'}$ and that the set of non-empty $C_{N'}$'s forms
	a sublattice coloring. The two constructions above are inverse
	to one another, showing the result.
\end{proof}

\bibliographystyle{alpha}
\bibliography{bibliography}
\end{document}